\documentclass[12pt]{article}
\usepackage{amssymb,amsmath,graphicx,color,empheq}
\usepackage[colorlinks=true,urlcolor=blue,linkcolor=black,citecolor=black]{hyperref}
\usepackage{tikz}

\usepackage{mythm}
\usepackage{mypaths}

\makeatletter

\newsavebox{\MyBox}

\newlength{\fighskip} \fighskip=2pt
\newlength{\figvskip} \figvskip=3pt

\newcommand*{\figbox}[2]{{
  \def\figscale{#1}
  \def\arraystretch{0.8}
  \arraycolsep=0pt
  \begin{array}{c}
    \vbox{\vskip\figscale\figvskip
      \hbox{\hskip\figscale\fighskip
        \includegraphics[scale=\figscale]{#2}}}
  \end{array}}}

\newcommand*{\wideboxed}[1]{\setlength{\fboxsep}{1ex}%
  \fbox{\m@th$\displaystyle#1$}}

\def\ubrace#1_#2{%
  \underbrace{#1}_{\hb@xt@\z@{\hss$\scriptstyle#2$\hss}}}

\newcommand{\arXiv}[1]{\href{http://arxiv.org/abs/#1}{\texttt{arXiv:#1}}}

\newcommand{\RR}{\mathbb{R}}
\newcommand{\CC}{\mathbb{C}}

\newcommand{\calA}{\mathcal{A}}
\newcommand{\calB}{\mathcal{B}}
\newcommand{\calC}{\mathcal{C}}

\newcommand{\calE}{\mathcal{E}}
\newcommand{\calF}{\mathcal{F}}
\newcommand{\calG}{\mathcal{G}}
\newcommand{\calH}{\mathcal{H}}

\newcommand{\calK}{\mathcal{K}}
\newcommand{\calL}{\mathcal{L}}
\newcommand{\calM}{\mathcal{M}}

\newcommand{\calS}{\mathcal{S}}

\newcommand{\calU}{\mathcal{U}}

\newcommand{\ts}{\mkern1mu}

\newcommand{\const}{\text{const}}
\newcommand{\bydef}{:=}

\newcommand{\trans}{{\mathpalette\@trans{}}}
\newcommand*{\@trans}[2]{\raisebox{\depth}{$\m@th#1\intercal$}}

\newcommand{\hotimes}{\mathbin{\hat{\otimes}}}
\newcommand{\comm}{\mathrm{comm}}

\DeclareMathOperator{\Ker}{Ker}
\DeclareMathOperator{\Img}{Im}

\DeclareMathOperator{\End}{End}
\DeclareMathOperator{\Tr}{Tr}
\DeclareMathOperator{\sgn}{sgn}

\DeclareMathOperator{\spec}{spec}
\DeclareMathOperator{\ind}{ind}
\DeclareMathOperator{\Ta}{T}
\DeclareMathOperator{\Fix}{Fix}

\newcommand{\UU}{\mathrm{U}}

\newcommand{\cb}{\mathrm{cb}}
\newcommand{\id}{\mathrm{id}}
\newcommand{\Ba}{\mathrm{B}}
\newcommand{\Bo}{\mathbf{B}}
\newcommand{\La}{\mathrm{L}}
\newcommand{\Ra}{\mathrm{R}}
\newcommand{\Co}{\mathrm{C}}
\newcommand{\Ha}{\mathrm{H}}
\newcommand{\Sw}{\mathrm{SWAP}}
\newcommand{\Enc}{\mathrm{Enc}}
\newcommand{\Dec}{\mathrm{Dec}}

\newcommand{\Euc}{\mathrm{E}}
\newcommand{\wt}{\widetilde}
\newcommand*{\Ma}[1]{\mathbf{M}_{#1}}

\newcommand{\blangle}{\bigl\langle}
\newcommand{\brangle}{\bigr\rangle}

\newcommand*{\braket}[2]{\langle{#1},{#2}\rangle}
\newcommand*{\bbraket}[2]{\blangle{#1},\mkern1mu{#2}\brangle}

\newcommand{\eps}{\varepsilon}
\newcommand{\ph}{\varphi}

\renewcommand{\le}{\leqslant}
\renewcommand{\ge}{\geqslant}

\let\oldRe\Re
\let\Re\relax
\DeclareMathOperator{\Re}{\oldRe}

\let\oldIm\Im
\let\Im\relax
\DeclareMathOperator{\Im}{\oldIm}

\usetikzlibrary{cd}
\tikzset{
  baseline = -0.6ex,
  every picture/.style = {
    semithick},
  bdot/.style = {
    circle,
    draw=none,
    fill=black,
    minimum size=4pt,
    inner sep=0pt,
    outer sep=0pt},
  circ/.style = {
    circle,
    draw=black,
    fill=none,
    minimum size=10pt,
    inner sep=0.8pt},
  rect/.style = {
    rectangle,
    draw=black,
    fill=none,
    minimum size=0.15cm,
    inner xsep=1.2pt,
    inner ysep=2pt},
  base rect/.style = {
    rectangle,
    draw=black,
    fill=none,
    minimum size=0.15cm,
    anchor=base,
    inner xsep=1.2pt,
    inner ysep=1mm},
  Choi/.style = {
    row sep={0.4cm,between origins},
    column sep=0.4cm,
    inner sep=0pt,
    label distance=0.8pt},
  bridge height/.default=0.4cm,
}    

\newtheorem{Theorem}{Theorem}[section]
\newtheorem{Proposition}[Theorem]{Proposition}
\newtheorem{Lemma}[Theorem]{Lemma}
\newtheorem{Corollary}[Theorem]{Corollary}

\newtheorem[definition]{Definition}[Theorem]{Definition}
\newtheorem[remark]{Example}[Theorem]{Example}
\newtheorem[remark]{Remark}[Theorem]{Remark}
\newtheorem[remark]{Question}[Theorem]{Open question}

\newtheorem*[definition]{definition}{Definition}
\newtheorem*[remark]{example}{Example}
\newtheorem*[remark]{remark}{Remark}
\newtheorem*[remark]{question}{Open question}

\numberwithin{equation}{section}

\makeatother

\title{Almost-idempotent quantum channels and approximate $C^*$-algebras}

\author{Alexei Kitaev\\
\normalsize\it California Institute of Technology, Pasadena, CA 91125, U.S.A.}

\date{February 3, 2025}

\begin{document}

\maketitle

\begin{abstract}
Let $\Phi$ be a unital completely positive (UCP) map on the space of operators on some Hilbert space. We assume that $\Phi$ is $\eta$-idempotent, namely, $\|\Phi^2-\Phi\|_{\mathrm{cb}} \le\eta$, and construct an associated $\varepsilon$-$C^*$ algebra (of almost-invariant observables) for $\varepsilon=O(\eta)$. This type of structure has the axioms of a unital $C^*$ algebra but the associativity and other axioms involving the multiplication and the unit hold up to $\varepsilon$. We prove that any finite-dimensional $\varepsilon$-$C^*$ algebra $\mathcal{A}$ is $O(\varepsilon)$-isomorphic to some genuine $C^*$ algebra $\mathcal{B}$. These bounds are universal, i.e.\ do not depend on the dimensionality or other parameters. When $\mathcal{A}$ comes from a finite-dimensional $\eta$-idempotent UCP map $\Phi$, the $O(\eta)$-isomorphism and its inverse can be realized by UCP maps. This gives an approximate factorization of the quantum channel $\Phi^*$ into a decoding channel, producing a state on $\mathcal{B}$, and an encoding channel.
\end{abstract}

\setcounter{tocdepth}{2}
\tableofcontents

\paragraph*{Some common notation.} The Hermitian inner product $\braket{{}\cdot{}}{{}\cdot{}}$ is conjugate linear in the first argument and linear in the second argument. The Hermitian conjugation and similar operations are denoted by a dagger. The unit element of an algebra $\calA$ is denoted by $I_\calA$, whereas $1_\calL$ stands for the identity map on a vector space $\calL$. The space of bounded linear maps $\calL\to\calM$ between Banach spaces is denoted by $\Bo(\calL,\calM)$. In particular, $\Bo(\calL)$ is the Banach (or $C^*$) algebra of operators acting on a Banach (or Hilbert) space $\calL\not=0$. (We consider only unital algebras with $I\not=0$.) For example, $\Ma{m,k}=\Bo(\CC^k,\CC^m)$ and $\Ma{n}=\Bo(\CC^n)$ (for $n>0$) are the matrix spaces and algebras. This notation should not be confused with $\Ba_r(x)$ and $\bar{\Ba}_r(x)$, which refer to the open and closed balls of radius $r$ with center at $x$ in some metric space. 
\medskip

\section{Motivation: How to recognize an encoded subsystem?}\label{sec_motivation}

Under favorable circumstances, some physical systems encode states and observables of smaller, ``logical'' systems. This may be arranged by design or happen naturally, for example, in so-called topological phases of matter that feature protected degrees of freedom. In practice, the encoding is always approximate and may also be implicit (especially if the system is natural rather than engineered). The informal goal of this paper is to reconstruct the logical algebra of observables in such a situation. This kind of task is common in quantum many-body physics, where identifying ``effective degrees of freedom'' is necessary for the understanding of complex systems. Recent work~\cite{SKK19} provides a rigorous reconstruction of topological degrees of freedom in spin systems on a 2D lattice under specific conditions, which involve only exact equalities. We will formulate and solve the approximate reconstruction problem in an abstract form, in hope that the results could be applied broadly. However, let us begin with the ideal case, where everything is explicit and exact.

\subsection{Encoding and decoding}

In general, the logical algebra of observables is a finite-dimensional $C^*$ algebra. As such, it is isomorphic to an algebra of the form $\calA=\bigoplus_{j=1}^{m}\Bo(\calL_j)$. (In general, $\Bo(\calL)$ denotes the algebra of bounded linear operators on a nonzero Hilbert space $\calL$. In the finite-dimensional case, all linear operators are bounded.) A purely quantum system is described by a single space, whereas for a classical system, all spaces $\calL_j$ are one-dimensional. Logical observables are represented by their action on a code subspace $\calM$ of the physical Hilbert space $\calH$ (also finite-dimensional for simplicity) via an inclusion of unital $C^*$ algebras $w\colon\calA\to\Bo(\calM)$.

\begin{Proposition}\label{prop_hom_structure}
Let $\calL_1,\dots,\calL_m$ and $\calM$ be finite-dimensional Hilbert spaces, $\calA=\bigoplus_{j=1}^{m}\Bo(\calL_j)$, and let $w\colon\calA\to\Bo(\calM)$ be a $*$-homomorphism. Then there exist auxiliary Hilbert spaces $\calE_j$ and a unitary map $W$,
\begin{equation}\label{partial_dec}
W=(W_1,\dots,W_m)\,\colon\,
\calM\to\bigoplus_{j=1}^{m}\calL_j\otimes\calE_j,\qquad\:
W_j\colon\calM\to\calL_j\otimes\calE_j.
\end{equation}
such that
\begin{equation}
w(A_1,\dots,A_m)=\sum_{j=1}^{m}W_j^\dag(A_j\otimes 1_{\calE_j})W_j\quad\:
\text{for all }\, A=(A_1,\dots,A_m)\in\calA.
\end{equation}
The homomorphism $w$ is injective if and only if $\calE_j\not=0$ for all $j$.
\end{Proposition}
\begin{proof}
By definition, a $*$-homomorphism $w\colon\calA\to\Bo(\calM)$ makes $\calM$ into a $*$-representation of $\calA$. Such representations have the property that the orthogonal complement of any invariant subspace is also invariant. Thus, $\calM$ splits into irreducible representations of $\calA$, i.e.\ copies of $\Bo(\calL_j)$. If $\Bo(\calL_j)$ occurs $n_j$ times, we set $\calE_j=\CC^{n_j}$. The rest of the argument is straightforward.
\end{proof}

States on an arbitrary $C^*$ algebra $\calB$ belong to the space of linear functionals $\calB^*$. Physically realizable transformations of states are known as \emph{quantum channels}; they are dual to unital completely positive (UCP) maps of the corresponding algebras of observables. For example, dual to the inclusion $w$ is state decoding $w^*\colon\Bo(\calM)^*\to\calA^*$. If $\calB$ is finite-dimensional, then $\calB^*\cong\calB$, and therefore, states are represented by (block-diagonal) density matrices. In particular, a state on the algebra $\calA$ is given by a collection of positive Hermitian operators $\rho_j\in\Bo(\calL_j)$ such that $\sum_{j=1}^{m}\Tr\rho_j=1$. Quantum channels are represented by completely positive trace-preserving (CPTP) maps of density matrices. Thus, the decoding map takes a density matrix $\rho\in\Bo(\calM)$ to $w^*(\rho)=(\rho_1,\dots,\rho_m)$, where $\rho_j=\Tr_{\calE_j}(W_j\rho W_j^\dag)$. 

State encoding, as well as the decoding of states beyond the code subspace, involves some arbitrary choices. The encoding map has the form
\begin{equation}\label{Enc}
\Enc\colon\calA^*\to\Bo(\calH)^*,\qquad
\Enc(\rho_1,\dots,\rho_m)
= \sum_{j=1}^{m} J_\calM W_j^\dag(\rho_j\otimes\gamma_j)W_j J_\calM^\dag,
\end{equation}
where $\gamma_j$ is some density matrix on $\calE_j$ and $J_\calM\colon\calM\to\calH$ is the inclusion map. We now define the decoding of density matrices $\rho\in\Bo(\calH)$. Let $\calM^\perp$ be the orthogonal complement of $\calM$ and $J_{\calM^\perp}\colon\calM^\perp\to\calH$ the corresponding inclusion. Choose an arbitrary CPTP map $S\colon\Bo(\calM^\perp)^*\to\calA^*$, which can be represented by its components, i.e.\ completely positive maps $S_j\colon\Bo(\calM^\perp)^*\to\Bo(\calL_j)^*$. Then the decoding map is also defined in terms of components:
\begin{equation}\label{Dec}
\Dec\colon\Bo(\calH)^*\to\calA^*,\qquad
\Dec_j(\rho)=\Tr_{\calE_j}\bigl(W_j J_\calM^\dag\rho J_\calM W_j^\dag\bigr)
+S_j\bigl(J_{\calM^\perp}^\dag\rho J_{\calM^\perp}\bigr).
\end{equation}
One can readily see that $\Dec\,\Enc=1_{\calA^*}$. Therefore, the quantum channel $T=\Enc\,\Dec$ is idempotent, i.e.\ $T^2=T$. Such a channel may be interpreted as a recovery from a correctable error, which has affected only the additional spaces $\calE_j$ rather than the information stored in $\calL_j$. Indeed, $T$ keeps the encoded information intact but resets $\calE_j$ to fixed states $\gamma_j$ (thus preventing future errors from compounding the past errors); it also pushes the out-of-bounds state portion $\rho^\perp=J_{\calM^\perp}^\dag\rho J_{\calM^\perp}$ to the code subspace. We will see that any idempotent channel $T\colon\Bo(\calH)^*\to\Bo(\calH)^*$ can be represented in the form $\Enc\,\Dec$. Thus, \emph{an idempotent channel implicitly defines a logical algebra of observables}.

\begin{figure}
\centering
\(\displaystyle
\begin{array}{c@{\hspace{1.5cm}}c}
\figbox{1.0}{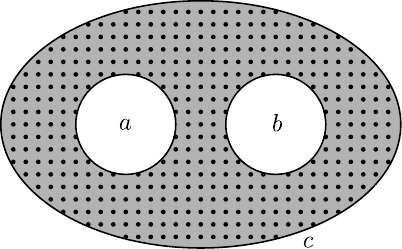} & \figbox{1.0}{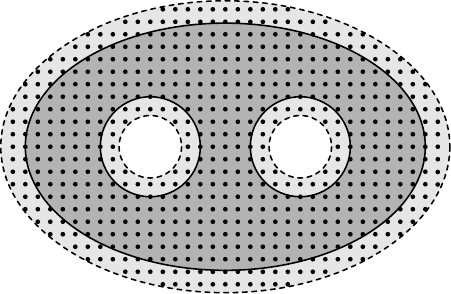}\\[3pt]
\text{(a)} & \text{(b)}
\end{array}
\)
\caption{(a)~A system of spins (depicted as dots) on a topological disk with two holes; (b)~Expanding the system by adding extra spins along the boundary.}
\label{fig_two_holes}
\end{figure}

\begin{example}[Physical example.]
Low-energy states of some local Hamiltonians in two spatial dimensions are described by unitary modular tensor categories (UMTCs). In particular, consider a spin system on a topological disk with two holes, see Figure~\ref{fig_two_holes}a. Let $\calH$ be the full Hilbert space and $\calM$ the subspace of low-energy states. (It is beyond the scope of this paper to try to make this construction rigorous.) Then
\begin{equation}
\calM\cong \bigoplus_{a,b,c}V^{ab}_{c}\otimes\calE_{ab}^{c},\qquad
V^{ab}_{c}=\calC(c,a\otimes b),
\end{equation}
where $a,b,c$ run over the equivalence classes of simple objects of some UMTC $\calC$. In physics, these are called ``superselection sectors''; $a$ and $b$ are associated with the holes, $c$ with the outer boundary, and $V^{ab}_{c}$ is the corresponding ``fusion space''. The logical algebra of observables is $\calA=\bigoplus_{a,b,c}\Bo(V^{ab}_{c})$, whereas the spaces $\calE_{ab}^{c}=\calE_a\otimes\calE_b\otimes\calE^c$ describe non-universal boundary modes. To reconstruct $\calA$, one may define an idempotent quantum channel $T=T_{-}T_{+}$, where $T_{+}$ corresponds to expanding the system as shown in Figure~\ref{fig_two_holes}b using the methods of Ref.~\cite{SKK19}, and $T_{-}$ is the operation of tracing out the added spins. The channel $T$ resets the boundary states but does not affect the fusion spaces.
\end{example}

\subsection{Idempotent and almost idempotent UCP maps}\label{sec_ch_intro}

We will work with the duals of quantum channels, i.e.\ UCP maps $\Phi\colon\Bo(\calH)\to\Bo(\calH)$. Such $\Phi$ is idempotent if and only if $T=\Phi^*$ is idempotent. The following structure theorem will be proved in Section~\ref{sec_idemp_structure}. It is related to the characterization of operators that are invariant under an arbitrary finite-dimensional UCP map \cite[Theorems~6.12 and Lemma~6.4]{Wolf} and of quantum channels stabilizing a given set of density matrices \cite{KoIm01}, \cite[Theorem~10]{HJPW03}.

\begin{Theorem}\label{th_idemp_structure}
Let $\calH$ be a finite-dimensional Hilbert space and $\Phi\colon\Bo(\calH)\to\Bo(\calH)$ an idempotent UCP map. Then there exist a subspace $\calM\subseteq\calH$ with the inclusion map $J_\calM\colon\calM\to\calH$ and a $C^*$ algebra $\calA$ with UCP maps $\Delta$ and $\Gamma$ as described by the diagram
\begin{equation}
\begin{tikzcd}[ampersand replacement=\&]
\& \Bo(\calH) \arrow[d,"\Co_\calM"] \\
\calA \arrow[ur,"\Delta"] \& \Bo(\calM) \arrow[l,"\Gamma"]
\end{tikzcd}\qquad\quad
\Co_\calM\colon X\mapsto J_\calM^\dag XJ_\calM,
\end{equation}
such that
\begin{equation}\label{Gamma_C_Delta}
\Gamma\Co_\calM\Delta=1_\calA,\qquad \Delta\Gamma\Co_\calM=\Phi.
\end{equation}
Furthermore, $w=\Co_\calM\Delta\colon\calA\to\Bo(\calM)$ is a homomorphism of $C^*$ algebras, and operators in the image of $\Delta$ are block-diagonal with respect to the decomposition $\calH=\calM\oplus\calM^\perp$.
\end{Theorem}

It will also be shown that $\Enc=(\Gamma\Co_\calM)^*$ and $\Dec=\Delta^*$ have the form \eqref{Enc} and \eqref{Dec}, respectively. For now, let us discuss more general corollaries of Theorem~\ref{th_idemp_structure}. First, equations \eqref{Gamma_C_Delta} imply that $\Delta$ is injective and that $\Img\Delta=\Img\Phi$, so $\calA$ may be identified with the vector space $\Img\Phi$. For example, $A,B,AB\in\calA$ are identified, respectively, with $X=\Delta(A)$,\, $Y=\Delta(B)$, and $Z=\Delta(AB)$. To represent the latter in a convenient form, we will use the second part of the theorem. It implies that $\Co_\calM(Z)=\Co_\calM(X)\,\Co_\calM(Y)=\Co_\calM(XY)$, and therefore,
\begin{equation}
Z=\Delta 1_\calA(AB)=\Delta\Gamma\Co_\calM\Delta(AB)
=\Delta\Gamma\Co_\calM(Z)=\Delta\Gamma\Co_\calM(XY)=\Phi(XY).
\end{equation}

The last equation means that upon the identification of $\calA$ with $\Img\Phi$, the multiplication in $\calA$ becomes the \emph{Choi-Effros product},
\begin{equation}
X\star Y=\Phi(XY).
\end{equation}
Choi and Effros showed that for any idempotent UCP map $\Phi\colon\Bo(\calH)\to\Bo(\calH)$ and any underlying Hilbert space $\calH$ (not only a finite-dimensional one), the subspace $\calA=\Img\Phi =\Ker(1-\Phi)\subseteq\Bo(\calH)$ equipped with this product, together with the norm and the involution $X\mapsto X^\dag$ inherited from $\Bo(\calH)$, satisfies all axioms of a $C^*$ algebra \cite[Theorem~3.1]{ChEf77}.

Physical applications require tolerance to small errors, so one should consider UCP maps that are only approximately idempotent. Moreover, meaningful properties do not change when the physical system is extended with new parts, as long as they do not interact with it. The second requirement suggests the use of the following definition~\cite{Paulsen}. A linear map $\Lambda\colon\calA'\to\calA''$ between two $C^*$ algebras is called \emph{completely bounded} if this \emph{completely bounded norm} is finite:
\begin{equation}
\|\Lambda\|_\cb=\sup_{n}\ts\sup_{X\not=0}
\frac{\|(1_{\Ma{n}}\otimes\Lambda)(X)\|}{\|X\|}\qquad
(X\in\Ma{n}\otimes\calA'),
\end{equation}
where $\Ma{n}=\Bo(\CC^n)$ is the algebra of $n\times n$ matrices. The completely bounded norm in finite dimensions is efficiently computable using semidefinite programming \cite[Chapter~3]{Watrous}. In quantum information literature, the dual norm for maps of linear functionals is used; it is denoted by $|\kern-0.6pt|\kern-0.6pt|\cdot|\kern-0.6pt|\kern-0.6pt|_1$ or $\|\cdot\|_{\Diamond}$.

A UCP map $\Phi\colon\Bo(\calH)\to\Bo(\calH)$ is called \emph{$\eta$-idempotent} if $\|\Phi^2-\Phi\|_\cb\le\eta$. A straightforward approach to such maps seems to be the approximation by idempotent maps. To this end, one could try to employ the function calculus for the Banach algebra of completely bounded linear maps on $\Bo(\calH)$ with the norm $\|\cdot\|_\cb$. Specifically, let\vspace{-3pt}
\begin{equation}
\wt{\Phi}=\theta(2\Phi-1),\qquad \text{where}\quad
\theta(x)=\begin{cases}1 & \text{if } x>0,\\ 0 & \text{if } x<0.\end{cases}
\end{equation}
The basic properties of $\wt{\Phi}$ follow from Proposition~\ref{prop_P} and the fact that $\|\Phi\|_\cb=1$. Specifically, $\wt{\Phi}$ is well-defined if $\eta<1/4$; it is idempotent and satisfies the bound
\begin{equation}
\|\wt{\Phi}-\Phi\|_\cb\le O(\eta).
\end{equation}
Unfortunately, $\wt{\Phi}$ is not guaranteed to be completely positive.

\begin{Example}
Let $\eta$ be a sufficiently small positive number. Consider the UCP map
\begin{gather}
\Phi\colon\Bo(\CC^2)\to\Bo(\CC^2),\qquad
\Phi(X)= P_0\Tr(\gamma_0 X)+P_1\Tr(\gamma_1 X),
\\[3pt]
P_0 = \begin{pmatrix} 1 & 0\\ 0 & 0\end{pmatrix},\quad\:
P_1 = \begin{pmatrix} 0 & 0\\ 0 & 1\end{pmatrix},\qquad
\gamma_0 =\begin{pmatrix}
1-\eta & \sqrt{\eta(1-\eta)}\\ \sqrt{\eta(1-\eta)} & \eta
\end{pmatrix},\quad\:
\gamma_1 =\begin{pmatrix} 0 & 0\\ 0 & 1\end{pmatrix}.
\end{gather}
One can check that $\|\Phi^2-\Phi\|_\cb=\eta\sqrt{1-\eta}$ and the map $\wt{\Phi}=\theta(2\Phi-1)$ acts as follows:
\begin{equation}
\wt{\Phi}(X)=P_0\Tr(\gamma X)+P_1\Tr(\gamma_1 X)
\end{equation}
where\vspace{-3pt}
\begin{equation}
\gamma = (1-\eta)^{-1} (\gamma_0 - \eta\gamma_1)
= \begin{pmatrix} 1 & \sqrt{\eta/(1-\eta)}\\ \sqrt{\eta/(1-\eta)} & 0\end{pmatrix}
\end{equation}
is not positive. On the other hand, there is an idempotent UCP map that is $O(\sqrt{\eta})$-close to $\Phi$.
\end{Example}

Is it possible to approximate all $\eta$-idempotent UCP maps by idempotent ones with accuracy $O(\sqrt{\eta})$ or some other function of $\eta$ that does not depend on the space dimensionality or other parameters? This seems to be an open problem; at least, I do not know of a solution. Nonetheless, one can use $\wt{\Phi}=\theta(2\Phi-1)$ to construct an approximate algebra of observables. As a vector space, this algebra is
\begin{equation}
\calA=\Img\wt{\Phi}=\Ker(1-\wt{\Phi})\subseteq\Bo(\calH).
\end{equation}
The multiplication is defined by the equation
\begin{equation}
X\star Y=\wt{\Phi}(XY)\qquad (X,Y\in\calA),
\end{equation}
whereas the unit, the norm and the involution $X\mapsto X^\dag$ are inherited from $\Bo(\calH)$. The axioms for such ``$\eps$-$C^*$ algebras'' are formulated in the next section. It will be shown in Section~\ref{sec_assoc_ecsa} that this particular algebra $\calA$ satisfies the axioms for $\eps=O(\eta)$. Since the hypothesis about $\Phi$ involves the completely bounded norm, we are able to prove an extended version of the previous statement, namely, that $\Ma{n}\otimes\calA$ satisfies the $\eps$-$C^*$ axioms uniformly in $n$.

The main result of this paper is that any finite-dimensional $\eps$-$C^*$ algebra $\calA$ is $\delta$-isomorphic to some $C^*$ algebra $\calB$ for $\delta=O(\eps)$. We will also show that if the norms are consistently defined for all algebras $\Ma{n}\otimes\calA$ and the uniform $\eps$-$C^*$ axioms hold, then there is a linear map $v\colon\calB\to\calA$ such that all maps $1_{\Ma{n}}\otimes v\colon \Ma{n}\otimes\calB\to\Ma{n}\otimes\calA$ are $\delta$-isomorphisms (where $\delta$ depends only on $\eps$, see Section~\ref{sec_tens_ext}). Applying this extended result to the previously defined $\calA$, we will prove (in Section~\ref{sec_approx_fact}) the existence of a $C^*$ algebra $\calB$ and UCP maps $\Delta\colon\calB\to\Bo(\calH)$ and $\Upsilon\colon\Bo(\calH)\to\calB$ such that $\|\Upsilon\Delta-1_{\mathcal{B}}\|_{\mathrm{cb}} \le O(\eta)$ and $\|\Delta\Upsilon-\Phi\|_{\mathrm{cb}} \le O(\eta)$. The dual maps $\Delta^*$ and $\Upsilon^*$ may be viewed as approximate variants of a decoding and an encoding, respectively.

\section{Main theorem on \texorpdfstring{$\eps$-$C^*$}{epsilon-C*} algebras and the proof strategy}\label{sec_main_thm}

Let $\eps$ be a sufficiently small nonnegative number. Individual definitions and statements will require (in many cases, implicitly) that $\eps$ be bounded by certain constants. For example, condition \eqref{ax_C*} below makes sense if $\eps<1$.

\begin{Definition}
An \emph{$\eps$-Banach algebra} is a Banach space $\calA$ endowed with a bilinear multiplication map $\calA\times \calA\to \calA$ such that
\begin{alignat}{2}
\label{ax_prodnorm}
\|XY\| &\le (1+\eps)\ts\|X\|\ts\|Y\|\qquad &&(X,Y\in \calA),\\[2pt]
\label{ax_assoc}
\|(XY)Z-X(YZ)\| &\le \eps\ts\|X\|\ts\|Y\|\ts\|Z\|\qquad &&(X,Y,Z\in \calA).
\end{alignat}
Such an algebra is called \emph{$\eps'$-commutative} if
\begin{equation}
\label{ax_comm}
\|XY-YX\|\le \eps'\ts\|X\|\ts\|Y\|\qquad (X,Y\in \calA).
\end{equation}
A \emph{$*\eps$-Banach algebra} is a complex $\eps$-Banach algebra with a conjugate linear involution $X\mapsto X^\dag$ satisfying the equations
\begin{equation}
\label{ax_*}
\|X^\dag\|=\|X\|,\qquad (XY)^\dag=Y^\dag X^\dag\qquad (X,Y\in \calA).
\end{equation}
An \emph{$\eps$-$C^*$ algebra} is one with this additional property:
\begin{equation}
\label{ax_C*}
\|X^{\dag}X\|\ge (1-\eps)\ts\|X\|^{2}\qquad (X\in \calA).
\end{equation}
(A bound from the other side, $\|X^{\dag}X\|\le (1+\eps)\ts\|X\|^{2}$, follows from \eqref{ax_prodnorm} and \eqref{ax_*}.) We assume that all algebras are unital. The unit element $I\in\calA$ should satisfy the approximate or exact conditions
\begin{alignat}{4}
\label{ax_eps_unit}
\|XI&-X\|\le\eps\ts\|X\|,\qquad &\|IX&-X\|\le\eps\ts\|X\|,\qquad
&\bigl|\|I\|&-1\bigr|\le\eps\qquad &&\text{(by default)},\quad \text{or}\\[2pt]
\label{ax_exact_unit}
XI&=X,\qquad &IX&=X,\qquad
&\|I\|&=1\qquad &&\text{(if specified)}.
\end{alignat}
If the involution is defined, one also requires that $I^\dag=I$.
\end{Definition}
(The exact unit condition \eqref{ax_exact_unit} can be achieved by an $O(\eps)$-change of both the unit and the multiplication, see Proposition~\ref{prop_unit}.)

\begin{Definition}
A \emph{$\delta$-homomorphism} from an $\eps'$-Banach algebra $\calA'$ to an $\eps''$-Banach algebra $\calA''$ is a bounded linear map $v\colon \calA'\to\calA''$ that almost preserves the unit and the multiplication:
\begin{align}
\label{hom_unit}
\|v(I)-I\|&\le \delta,\\[2pt]
\label{hom_mult}
\|v(XY)-v(X)v(Y)\|&\le \delta\ts\|X\|\ts\|Y\|\qquad (X,Y\in \calA').
\end{align}
A \emph{non-unital $\delta$-homomorphism} is defined by imposing only condition \eqref{hom_mult}. In the $*$-algebra setting, it is also required that $v(X^\dag)=v(X)^\dag$. A \emph{$\delta$-inclusion} is a $\delta$-homomorphism such that
\begin{equation}
(1-\delta)\ts\|X\|\le \|v(X)\|\le (1+\delta)\ts\|X\|\qquad (X\in \calA').
\end{equation}
A \emph{$\delta$-isomorphism} is a bijective $\delta$-inclusion.
\end{Definition}

Note that the inverse of a $\delta$-isomorphism is a ($\delta+O(\delta^2)$)-isomorphism. Here and in general, each instance of big-$O$ or similar notation stands for a concrete function, not depending on any additional data.

\begin{Theorem}\label{th_main}
Any finite-dimensional $\eps$-$C^*$ algebra $\calA$ is $O(\eps)$-isomorphic to some $C^*$ algebra $\calB$. (Importantly, the implicit constant in $O(\eps)$ does not depend on $\calA$ or its dimensionality.)\medskip
\end{Theorem}

The rest of this section is informal. A natural approach to proving theorem~\ref{th_main} would be to show that $\calA$ has trivial cohomology, or just that $H^3(\calA,\calA)=0$. Here we are referring to the Hochschild cohomology of associative algebras or its Banach version~\cite{Joh-coh}, somehow adapted to $\eps$-associativity. To see how it might work, consider the \emph{associativity defect} of $\calA$:
\begin{equation}
h\colon \calA^3\to\calA,\qquad
h(X,Y,Z)=(XY)Z-X(YZ).
\end{equation}
It satisfies the 3-cocycle equation,
\begin{equation}
W h(X,Y,Z) - h(WX,Y,Z) + h(W,XY,Z) - h(W,X,YZ) + h(W,X,Y) Z =0.
\end{equation}
Note that $h(X,Y,Z)$ is small, being bounded by $\eps\ts\|X\|\ts\|Y\|\ts\|Z\|$. If it could be represented as a coboundary with $O(\eps)$ relative accuracy,
\begin{equation}
h(X,Y,Z) = X g(Y,Z) - g(XY,Z) + g(X,YZ) - g(X,Y) Z
+O(\eps^2)\ts\|X\|\ts\|Y\|\ts\|Z\|,
\end{equation}
then the modified multiplication $X\cdot Y=XY+g(X,Y)$ would be $O(\eps^2)$-associative. The iteration of this procedure would provide exact associativity.

An associative algebra $\calA$ has trivial Hochschild cohomology if there is a \emph{diagonal}, i.e.\ an element $D=\sum_{j}A_j\otimes B_j\in\calA\otimes\calA$ satisfying the equations $\sum_{j}A_j\otimes B_jX=\sum_{j}XA_j\otimes B_j$ for all $X\in\calA$ and $\sum_{j}A_jB_j=I$. Indeed, in this case, an arbitrary cocycle $h$ is the coboundary of this $g$:
\begin{equation}
g(X,Y)=\sum_j A_{j}h(B_j,X,Y).
\end{equation}
For finite-dimensional $C^*$ algebras, a diagonal can be obtain as $D=\int dU\, (U^\dag\otimes U)$, where the integral is taken with respect to the Haar measure on the unitary group. Unfortunately, naive constructions of the Haar measure (or just the diagonal) in the $\eps$-associative setting have error bounds proportional to $n=\dim\calA$. So the outlined procedure of fixing the multiplication works only if $\eps<cn^{-1}$ for some constant $c$.

To actually prove the theorem, we take a different approach, but a cohomological argument is still used. We construct an $O(\eps)$-isomorphism between some $C^*$ algebra $\calB$ and the $\eps$-$C^*$ algebra $\calA$ incrementally. Before the process is complete, the map $v\colon \calB\to\calA$ is an $O(\eps)$-inclusion, and we enlarge $\calB$ and modify $v$ in little steps. This involves solving two problems. To even make the first step, from $\calB_0=\CC$ to $\calB_1=\CC\oplus\CC$, we need to find a \emph{nontrivial projection} in $\calA$. An approximate projection, satisfying the conditions $P^\dag=P$ and $\|P^2-P\|\le\delta$ for a suitable constant $\delta$, is sufficient as long as $P$ is not too close to $\pm I$. The existence of such a projection (assuming that $\dim\calA>1$) is the first important lemma. We prove it by first constructing (in Section~\ref{sec_unitaries}) an approximate unitary group $\calU\subseteq\calA$ as a $C^1$ manifold using standard analytic methods such as the implicit function theorem. The group laws are satisfied with $O(\eps)$ accuracy, and therefore, they also hold up to homotopy. Now, $P$ is a projection if and only if $U=2P-I$ is Hermitian and $U^2=I$. In other words, $U$ is a fixed point of the inversion map $\sigma\colon U\mapsto U^{-1}$ acting on $\calU$. We show that $\sigma$ has sufficiently many (i.e.\ more than two) fixed points using the Lefschetz-Hopf theorem and certain properties of H-spaces (see Section~\ref{sec_projection}).

Once the lemma is proved, we can keep building the finite-dimensional $C^*$ algebra $\calB$ in a similar fashion. At each step, $\calB$ is described by a partitioned index set and the basis $\{E_{jk}\}$ with $j$ and $k$ in the same part. The multiplication rules are $E_{jk}E_{lm}=\delta_{kl}E_{jm}$. At the beginning of the incremental process, the parts are singletons so that all basis elements are diagonal and $\calB$ is commutative. Each step involves replacing some $j$ and the corresponding $E_{jj}$ with $\{j',j''\}$ and $E_{j'j'}+E_{j''j''}$, respectively. When that is no longer possible, we add off-diagonal elements, $E_{jk}$ with $j\not=k$, until $v\colon \calB\to\calA$ becomes bijective. The detailed argument is structured in terms of \emph{merging} (Corollary~\ref{cor_merge_sum}) and \emph{extension} (Lemma~\ref{lem_extension}) procedures.

The second problem is to control errors. We must make sure that $v$ is a $\delta_0$-inclusion for some fixed $\delta_0=O(\eps)$. Simply extending the current $v$ to a new, larger $\calB$ gives a $\delta$-inclusion for $\delta$ greater than $\delta_0$. Thus, we need the following \emph{error reduction} result (Corollary~\ref{cor_improvement}): if there exists a $\delta$-inclusion of a finite-dimensional $C^*$ algebra $\calB$ into an $\eps$-$C^*$ algebra $\calA$ for $\delta$ less than a certain constant, then there is also an $\delta_0$-inclusion, where $\delta_0=O(\eps)$ does not depend on $\delta$. This can be shown by cohomological methods because $\calB$ does have a diagonal. Similar results are found in the literature; they are concerned with the existence of a $C^*$ subalgebra near an approximate $C^*$ subalgebra~\cite{Chr80} and a homomorphism of $C^*$ algebras near an approximately multiplicative map~\cite{Joh88}. Our argument is similar to the proof of Theorem~3.1 in~\cite{Joh88} but actually simpler due to the finite dimensionality condition.

\section{Some analytic tools}

The study of approximate algebras naturally involves ordinary (real or complex) Banach algebras. For many purposes, the coefficient field does not matter. Indeed, any complex Banach algebra can be considered with real coefficients. Conversely, an arbitrary real Banach algebra $\calA$ can be extended to a complex Banach algebra, namely, $\calB=\CC\otimes_\RR\calA$ with the projective tensor norm,
\begin{equation}
\|Z\|=\inf\biggl\{\sum_{j}|c_j|\ts\|X_j\|:\,
\sum_{j}c_j\otimes X_j=Z\biggr\}\qquad
(c_j\in\CC,\,\: X_j\in\calA).
\end{equation}
Note that the inclusion map $\calA\to\calB$ preserves the norm, and if $X\in\calA$ is invertible in $\calB$, then $X^{-1}\in\calA$.

Let $\calB$ be a Banach algebra and $f$ a function of one complex variable represented by its Taylor series $f(x)=\sum_{n=0}^{\infty}a_n(x-x_0)^n$ with the convergence radius $\ge r$. Then $f(X)$ is well-defined, commutes with $X$, and satisfies a certain bound for all elements $X\in\calB$ sufficiently close to $x_0I$:
\begin{gather}
\label{Taylor_simple}
f(X)=a_0I+\sum_{n=1}^{\infty}a_n(X-x_0I)^n\quad\:
\text{if }\, \|X-x_0I\|<r,\\
\label{Taylor_simple_bound}
\|f(X)-f(x_0I)\|\le \sum_{n=1}^{\infty}|a_n|\ts\|X-x_0I\|^n.
\end{gather}
For example, let $\alpha$ be an arbitrary complex number, $x_0$ a positive number, and let $f(x)=x^{\alpha}$. Then $f(X)$ is well-defined if $\|X-x_0I\|<x_0$. Expanding around an arbitrary element $X_0$ is more complex, but $X^{-1}=X_0^{-1}(I+(X-X_0)X_0^{-1})^{-1}$ exists if $\|X-X_0\|<\|X_0^{-1}\|^{-1}$.

These are some interesting and useful functions:
\begin{equation}\label{abs_sgn}
|X|=(X^2)^{1/2},\qquad \sgn(X)=X(X^2)^{-1/2}\qquad\quad
\bigl(\|X^2-x_0^2I\|<x_0^2\bigr).
\end{equation}
If $X$ is an approximate solution of the equation $X^2=I$, then a nearby exact solution, commuting with $X$, is given by $\sgn(X)$. Indeed, $\sgn(X)^2=I$ for all $X$, and
\begin{equation}
\|\sgn(X)-X\|\le \|X\|\,O(\|X^2-I\|)\quad\:
\text{if }\, \|X^2-I\|<1.
\end{equation}

\begin{Proposition}\label{prop_P}
Let $P$ be an element of a Banach algebra such that $\|P^2-P\|\le\delta<1/4$, and let
\begin{equation}
\wt{P}=\theta(2P-I),\qquad
\text{where}\quad \theta(X)=\frac{1}{2}\bigl(I+\sgn(X)\bigr).
\end{equation}
Then $\wt{P}^2=\wt{P}$,\, $\wt{P}$ commutes with $P$, and $\|\wt{P}-P\|\le \|2P-I\|\ts O(\delta)$.
\end{Proposition}
\begin{proof} The statement follow from the preceding discussion applied to $X=2P-I$, which satisfies the inequality $\|X^2-I\|\le 4\delta<1$.
\end{proof}

For a more complete picture, let us mention the spectrum and the holomorphic function calculus~\cite[chapter~3]{KaRi}. (We have not found use of these methods due to their fragility in the approximate setting and the lack of a $*$-representation, which could alleviate it.) The \emph{spectrum} of an element $X\in\calB$, 
\begin{equation}
\spec(X)=\spec_{\calB}(X)=\{\lambda\in\CC:\, X-\lambda I
\text{ is not invertible in }\calB\},
\end{equation}
is a nonempty closed subset of $\CC$ contained in the disk $\{\lambda\in\CC:\,|\lambda|\le\|X\|\}$. Note that even in the algebra of operators acting in a finite-dimensional Hilbert space, $\spec(X)$ can be very sensitive to perturbations (unless we assume that $X$ is normal). For example, consider this $X$, where $t$ is small:
\begin{equation}
X=\begin{pmatrix}0&1&0\\ 0&0&1\\ t&0&0 \end{pmatrix},\qquad
\spec(X)=\{t^{1/3},e^{2\pi i/3}t^{1/3},e^{-2\pi i/3}t^{1/3}\}.
\end{equation}

Any function $f$ holomorphic on an open set $S$ can be applied to Banach algebra elements whose spectrum is contained in $S$. Specifically,
\begin{equation}
f(X)=\frac{1}{2\pi i}\int_{C}(zI-X)^{-1}f(z)\,dz,
\end{equation}
where $C$ is the union of closed oriented curves enclosing $\spec(X)$ and contained in $S$. The spectral mapping theorem asserts that $\spec(f(x))=f(\spec(x))$. For example $|X|$ and $\sgn(X)$ are defined for all $X$ such that $\spec(X) \subseteq\{\lambda\in\CC:\,\Re\lambda\not=0\}$, or equivalently, $\spec(X^2) \subseteq\CC\setminus(-\infty,0]$.

If an element $X_0$ with $\spec(X_0)\subseteq S$ is fixed, then $f\mapsto f(X_0)$ is a homomorphism of the algebra of functions holomorphic on $S$ to the closed subalgebra of $\calB$ generated by $X_0$. This subalgebra is obviously commutative. If both $f$ and $X_0$ are fixed, then $f(X)$ admits a Taylor expansion in $X-X_0$, provided $\|X-X_0\|<\delta$, where $\delta$ is sufficiently small. (The terms in that Taylor series are general degree $n$ functions, not just $a_n(X-X_0)^n$.) One may set $\delta=\min_{z\in C}\|(zI-X_0)^{-1}\|^{-1}$, but this number is often difficult to estimate.\medskip

Let us now discuss more general functions on Banach spaces. We will use the notation
\begin{equation}
\Ba^M_r(x_0)=\{x\in M:\,d(x_0,x)<r\},\qquad
\bar{\Ba}^M_r(x_0)=\{x\in M:\,d(x_0,x)\le r\}\qquad
(x_0\in M)
\end{equation}
for open and closed balls in a metric space $M$. Let $\calA$ and $\calB$ be Banach spaces, and $f\colon \Ba^\calA_r(x_0)\to\calB$ a continuously differentiable map. The Frechet derivative of $f$ at point $x$ is denoted by $\partial_{x}f(x)$ or $\frac{\partial f(x)}{\partial x}$; it is a linear map from $\calA$ to $\calB$. Functions of multiple variables are understood as maps of direct sums of Banach spaces, equipped with an arbitrary norm between the maximum and the sum norms.

The inverse function theorem asserts that if $\frac{\partial f(x)}{\partial x}\big|_{x=x_0}$ is bijective, then $f$ has a continuously differentiable inverse in some neighborhood of $f(x_0)$. An explicit bound on the size of that neighborhood is given by the next lemma, which implies the theorem and is useful on its own right. The implicit function theorem is derived from the inverse function theorem by the standard argument.

\begin{Lemma}\label{lem_invfun}
Let $V\colon \calA\to\calB$ be a linear isomorphism of Banach spaces and $f\colon \Ba^\calA_r(x_0)\to\calB$ a continuously differentiable function such that\,\footnote{Here $1=1_\calA$ stands for the identity operator on $\calA$. (We use $I$ for the units of abstract algebras, which are not automatically defined and need to be specified.)}
\[
\|V^{-1}\partial_{x}f(x)-1\|\le c\qquad (x\in\Ba^\calA_r(x_0)),
\]
where $0\le c<1$. Then
\begin{enumerate}
\item $f$ is injective and satisfies the inequality
\[
(1-c)\|x_1-x_2\|\le \|V^{-1}(f(x_1)-f(x_2))\|\le (1+c)\|x_1-x_2\|\qquad
(x_1,x_2\in\Ba^\calA_r(x_0));
\]
\item The image of $f$ includes $f(x_0)+V(\Ba^\calA_{(1-c)r}(0))$.
\end{enumerate}
\end{Lemma}
\begin{proof}
For each given $y$, the equation $f(x)=y$ is equivalent to $x$ being a fixed point of the function
\[
g_y(x)=x+V^{-1}(y-f(x))\qquad (x\in\Ba^\calA_r(x_0)).
\]
Due to the hypothesis, the norm of the derivative of $g_y$ is at most $c$. Therefore, $\|g_y(x_1)-g_y(x_2)\|\le c\|x_1-x_2\|$, which can be written as $\|a-b\|\le c\|a\|$, where $a=x_1-x_2$ and $b=V^{-1}(f(x_1)-f(x_2))$. It follows that $(1-c)\|a\|\le\|b\|\le(1+c)\|a\|$, which is exactly the inequality in part~1 of the lemma.

Part~2 asserts that the equation $f(x)=y$ has a solution for any $y$ such that $\|V^{-1}(y-f(x_0))\|<(1-c)r$. To show this, we note that if $\|x-x_0\|<r$, then
\[
\|g_y(x)-x_0\| \le \|g_y(x)-g_y(x_0)\|+\|g_y(x_0)-x_0\|
< c\|x-x_0\|+(1-c)r \le r.
\]
Hence, $g_y$ maps the ball $\Ba^\calA_r(x_0)$ to itself. The fixed point can be obtained as the limit of the sequence $x_n=g_y(x_{n-1})$. It is a Cauchy sequence because $\|x_{n}-x_{n-1}\|<rc^{n-1}$.
\end{proof}

For an illustration of the outlined methods, let us calculate $\partial_X\sgn(X)$ in a Banach algebra $\calB$. To this end, we consider the map
\begin{equation}
f\colon \calB\times\calB\times\calB\to\calB\times\calB,\qquad
f(X,Y_1,Y_2)=(Y_1Y_2-I,\,XY_2-Y_1X).
\end{equation}
For each given $X$ such that $\sgn(X)$ is well-defined, the equation $f(X,Y_1,Y_2)=0$ has a solution $Y_1=Y_2=\sgn(X)$. By the implicit function theorem, this solution can be continued to nearby values of $X$ if $\frac{\partial f(X,Y_1,Y_2)}{\partial(Y_1,Y_2)}$ is nondegenerate, i.e.\ maps $\calB\times\calB$ bijectively onto itself. It is evident that
\[
\frac{\partial f(X,Y_1,Y_2)}{\partial(Y_1,Y_2)}
=\begin{pmatrix}
\partial_{Y_1}(Y_1Y_2-I) & \partial_{Y_2}(Y_1Y_2-I)\\
\partial_{Y_1}(XY_2-Y_1X) & \partial_{Y_2}(XY_2-Y_1X)
\end{pmatrix}
=\begin{pmatrix}\Ra_{Y_2}&\La_{Y_1}\\ -\Ra_{X}&\La_{X}\end{pmatrix},
\]
where $\La_{A}$ and $\Ra_{A}$ denote the operators of multiplication by $A$ on the left and on the right, respectively. If $Y_1=Y_2=\sgn(X)$, then all elements of the matrix commute, and its inverse is computed easily:
\[
\left(\frac{\partial f(X,Y_1,Y_2)}{\partial(Y_1,Y_2)}
\bigg|_{Y_1=Y_2=\sgn(X)}\right)^{-1}
=\begin{pmatrix}\La_{X}&-\La_{\sgn(X)}\\ \Ra_{X}&\Ra_{\sgn(X)}\end{pmatrix}
(\La_X\Ra_{\sgn(X)}+\La_{\sgn(X)}\Ra_{X})^{-1}.
\]
The operator  $\La_X\Ra_{\sgn(X)}+\La_{\sgn(X)}\Ra_{X} =\La_{\sgn(X)}(\La_{|X|}+\Ra_{|X|})\Ra_{\sgn(X)}$ is invertible if $\|X^2-I\|<1$ because this condition implies (via a special case of inequality \eqref{Taylor_simple_bound}, $\|Z^{1/2}-I\|\le 1-(1-\|Z-I\|)^{1/2}$ with $Z=X^2$) that $\bigl\||X|-I\bigr\|<1$. Now that the solution of the equation $f(X,Y_1,Y_2)=0$ is given by an implicit function, $(Y_1,Y_2)=g(X)=(\sgn(X),\sgn(X))$, we can compute its derivative:
\[
\frac{\partial g(X)}{\partial X}
=\biggl(-\left(\frac{\partial f(X,Y_1,Y_2)}{\partial(Y_1,Y_2)}\right)^{-1}\,
\frac{\partial f(X,Y_1,Y_2)}{\partial X}\biggr)\bigg|_{(Y_1,Y_2)=g(X)}.
\]
Thus,
\begin{equation}
\partial_X\sgn(X)=\frac{1-\La_{\sgn(X)}\Ra_{\sgn(X)}}{\La_{|X|}+\Ra_{|X|}}\qquad
\text{if}\quad \|X^2-I\|<1.
\vspace{8pt}
\end{equation}

\begin{Remark}\label{rem_X2}
The system of equations $Y_1Y_2=I$,\, $XY_2=Y_1X$ has a unique solution with $Y_1$ and $Y_2$ close to $X$ in any $\eps$-Banach algebra, provided $\|X^2-I\|$ is sufficiently small. However, in general, $Y_1\not=Y_2$. Even the single equation $Y^2=I$ may not have a solution sufficiently close to $X$. Indeed, suppose that the norm of $X$ is bounded by a constant, $\|X^2-I\|\le O(\eps)$, and there is a $Y$ such that $Y^2=I$ and $\|Y-X\|\le O(\eps)$. Representing $X$ as $Y+Z$ with $\|Z\|\le O(\eps)$, we readily see that $\|XX^2-X^2X\|\le O(\eps^2)$. But such a bound may not hold in an arbitrary $\eps$-Banach algebra. For this reason, Proposition~\ref{prop_P} does not generalize to the $\eps$-Banach or even the $\eps$-$C^*$ setting, and we will have to work with approximate projections.
\end{Remark}

\section{Basic properties of \texorpdfstring{$\eps$}{epsilon}-Banach algebras}

For each element $X$ of an $\eps$-Banach algebra $\calA$, the operators of left and right multiplication by $X$ are defined in the standard way:
\begin{equation}
\La_X(Z)=XZ,\qquad \Ra_X(Z)=ZX\qquad (Z\in\calA).
\end{equation}
Since $\calA$ is a Banach space, linear operators on it have a standard norm,
\begin{equation}
\|w\|=\sup_{\|Z\|\le 1}\|w(Z)\|,\qquad w\colon \calA\to\calA,
\end{equation}
and bounded linear operators form a Banach algebra. It follows from equation \eqref{ax_prodnorm} that
\begin{equation}
\|\La_X\|\le (1+\eps)\|X\|,\qquad \|\Ra_X\|\le (1+\eps)\|X\|.
\end{equation}
Applying $\La_X$ and $\Ra_X$ to the unit element gives these inequalities:
\begin{alignat}{3}
\|\La_X\|&\ge (1-O(\eps))\ts\|X\|,\qquad
&\|\Ra_X\|&\ge (1-O(\eps))\ts\|X\|\qquad &
&(\text{in general}),\\[2pt]
\|\La_X\|&\ge \|X\|,\qquad
&\|\Ra_X\|&\ge \|X\|\qquad &
&(\text{if the unit is exact}).
\end{alignat}
Furthermore, the approximate associativity axiom \eqref{ax_assoc} implies that for all $X,Y\in\calA$,
\begin{gather}
\|\La_X\La_Y-\La_{XY}\|\le \eps\ts\|X\|\ts\|Y\|,\qquad
\|\Ra_X\Ra_Y-\Ra_{YX}\|\le \eps\ts\|X\|\ts\|Y\|,\\[2pt]
\|\La_X\Ra_Y-\Ra_Y\La_X\|\le \eps\ts\|X\|\ts\|Y\|.
\end{gather}
Note that the operator $\La_X$ (and similarly, $\Ra_X$) is invertible if and only if the equation $XY=Z$ (respectively, $YX=Z$) has a unique solution $Y$ for all $Z\in\calA$.

We will also need a bound on the Frechet derivative $\partial_X(X^2-X)=\La_X+\Ra_X-1$ for $X$ in a neighborhood of the unit element. Since $\|IY-Y\|\le\eps\|Y\|$ for all $Y$, we have $\|\La_I-1\|\le \eps$. It follows that if $\|X-I\|\le\delta$, then
\[
\|\La_X-1\| \le \|\La_{X-I}\|+\|\La_I-1\|\le (1+\eps)\ts\|X-I\|+\eps\le O(\delta+\eps).
\]
Similarly, $\|\Ra_X-1\|\le O(\delta+\eps)$. Thus,
\begin{equation}\label{d_X2X}
\|\partial_X(X^2-X)-1\|\le O(\delta+\eps)\qquad (X\in\bar{\Ba}_{\delta}(I)).
\end{equation}

\begin{Proposition}\label{prop_unit}
Let $\calA$ be an $\eps$-Banach algebra with unit $I$. Then there exist a new unit $J\in\calA$ and a new multiplication denoted by a dot that make $\calA$ into an $O(\eps)$-Banach algebra with exact unit while being $O(\eps)$-close to the original unit and multiplication:
\begin{equation}
\|J-I\|\le O(\eps),\qquad \|X\cdot Y-XY\|\le O(\eps)\ts\|X\|\ts\|Y\|\qquad
(X,Y\in\calA).
\end{equation} 
This transformation respects the involution if one is present, i.e.\ $J^\dag=J$ and $(X\cdot Y)^\dag=Y^\dag\cdot X^\dag$.
\end{Proposition}

\begin{proof}
Inequality \eqref{d_X2X} and Lemma~\ref{lem_invfun} imply that if $\delta>2\eps$, then the image of $\Ba^\calA_{\delta}(I)$ under the map $f\colon X\mapsto X^2-X$ contains $0$. Thus, there is some $J\in\calA$ such that $\|J-I\|\le O(\eps)$ and $J^2=J$. Let us define the new multiplication as follows:
\begin{equation}
X\cdot Y= \Ra_{J}^{-1}(X)\,\La_{J}^{-1}(Y).
\end{equation}
It is evident that $\La_{J}^{-1}(J) =J =\Ra_{J}^{-1}(J)$; hence, $X\cdot J=X=J\cdot X$. All other requirements are also met.
\end{proof}

\section{Approximate unitary group}\label{sec_unitaries}

Let $\calA$ be an $\eps$-$C^*$ algebra with exact unit. The sets of Hermitian and anti-Hermitian elements in $\calA$ are denoted by $\calH$ and $i\calH$, respectively. An element $U\in\calA$ is called \emph{unitary} if $U^{\dag}U=I$ and $U$ also has a right inverse. The set of unitary elements is denoted by $\calU$. Clearly, $\calU=\bar{\calU}_0$, where
\begin{equation}\label{U_delta}
\bar{\calU}_\delta=\bigl\{X\in\calA:\, \|X^\dag X-I\|\le 2\delta\,
\text{ and $X$ has a right inverse}\bigr\}.
\end{equation}
(A sufficient condition for the existence of a right inverse is that $\La_X$ is invertible.) The analogous set $\calU_\delta$ is defined by the inequality $\|X^{\dag}X-I\|<2\delta$. We assume that $\eps$ and $\delta$ are bounded by suitable positive constants, $\eps_{\max}$ and $\delta_{\max}$, and will show that $\calU$ is a $C^1$ manifold with a tubular neighborhood $\calU_{\delta_1}$, where $0<\delta_1\le\delta_{\max}$ (see Proposition~\ref{prop_polar}).

\begin{Lemma}\label{lem_U_delta}
If $X\in\bar{\calU}_\delta$, then $\|X\|\le 1+O(\delta+\eps)$. Furthermore, the operator $\La_X$ is invertible, and
\begin{equation}
\label{u_inv_bound}
\|\La_X^{-1}\|\le 1+O(\delta+\eps),\qquad
\|\La_X^{-1}-\La_{X^\dag}\| \le O(\delta+\eps).
\end{equation}
\end{Lemma}
\begin{proof}
The condition $\|X^{\dag}X-I\|\le 2\delta$ implies that $\|X^{\dag}X\|\le 1+2\delta$. On the other hand, $\calA$ is an $\eps$-$C^*$ algebra, so $\|X^{\dag}X\|\ge (1-\eps)\|X\|^2$. It follows that $\|X\|\le ((1+2\delta)/(1-\eps))^{1/2}\le 1+O(\delta+\eps)$.

To show that $\La_X$ is invertible, let us first use the fact that $X$ has a right inverse, i.e.\ that $XY=I$ for some $Y$. We have
\[
\|X^\dag-Y\|=\|X^{\dag}(XY)-Y\|\le \|(X^{\dag}X-I)Y\|+O(\eps)\|X\|^2\|Y\|
\le (2\delta+O(\eps))\|Y\|,
\]
and therefore, $\|X^\dag\|\ge (1-2\delta-O(\eps))\|Y\|$. But $\|X^\dag\|=\|X\|\le 1+O(\delta+\eps)$, which gives the bound $\|Y\|\le 1+O(\delta+\eps)$. The latter implies that
\[
\|\La_X\La_Y-1\|=\|\La_X\La_Y-\La_{XY}\|\le \eps\ts\|X\|\ts\|Y\| <1,
\]
provided $\eps$ and $\delta$ are small enough. Thus, $\La_X\La_Y$ is invertible, and so $\La_Y(\La_X\La_Y)^{-1}$ is a right inverse of $\La_X$. Similarly, $(\La_{X^\dag}\La_X)^{-1}\La_{X^\dag}$ is a left inverse of $\La_X$, where the existence of $(\La_{X^\dag}\La_X)^{-1}$ follows from the inequality
\[
\|\La_{X^\dag}\La_X-1\|\le
\|\La_{X^\dag}\La_X-\La_{X^{\dag}X}\|+\|\La_{X^{\dag}X}-\La_{I}\|
\le \eps\|X\|^2+(1+\eps)\|X^{\dag}X-I\| \le 2\delta+O(\eps).
\]
We conclude that $\La_X^{-1}=(\La_{X^\dag}\La_X)^{-1}\La_{X^\dag}$ and that $\|(\La_{X^\dag}\La_X)^{-1}-1\|\le O(\delta+\eps)$, from which inequalities \eqref{u_inv_bound} are obtained easily.
\end{proof}

Let us introduce local coordinates around each element $V\in\calU_{\delta_{\max}}$:
\begin{equation}
\ph_V(X)=\La_V^{-1}(X-V)\qquad (\ph_V\colon \calA\to\calA),\qquad\quad
\ph_V^{-1}(A)=V(I+A).
\end{equation}
The corresponding transition functions,
\begin{equation}
\Phi_{V,W}=\ph_{V}\ph_{W}^{-1},\qquad
\Phi_{V,W}(A)=\La_V^{-1}(W(I+A)-V),
\end{equation}
have derivative close to the identity, provided $V$ and $W$ are close to each other:\footnote{Here, $O(\eta)$ stands for some operator of norm $\le O(\eta)$. We will use similar notation for matrices of partial derivatives. As long as errors are estimated up to a constant factor, it does not matter what matrix norm is used.}
\begin{equation}
\partial_A\Phi_{V,W}(A)=\La_V^{-1}\La_W=1+O(\|W-V\|).
\end{equation}
In some situations, it is convenient to separate the tangential and normal  coordinates:
\begin{equation}
A=A^\parallel+A^\perp\qquad(A^\parallel\in i\calH,\:\,A^\perp\in\calH),\qquad
\text{i.e.}\quad
A^\parallel=\frac{1}{2}(A-A^\dag),\quad A^\perp=\frac{1}{2}(A+A^\dag).
\end{equation}
If $A=\ph_V(X)$, the corresponding $A^\parallel$ and $A^\perp$ are denoted by $\ph_V^\parallel(X)$ and $\ph_V^\perp(X)$, respectively. We will parametrize unitary elements $U$ in a neighborhood of $V$ by $\ph_V^\parallel(U)$. The condition that $U$ is unitary is expressed by the equation $f(U)=0$, where
\begin{equation}
f(X)=\frac{1}{2}(X^{\dag}X-I),\qquad f\colon \calA\to\calH.
\end{equation}
In local coordinates, the function $f$ takes the form
\begin{equation}
f_V(A)\bydef f(\ph_V^{-1}(A))
=\frac{1}{2}\Bigl((I+A^\dag)V^\dag)(V(I+A))-I\Bigr).
\end{equation}

\begin{Lemma}\label{lem_gV}
Let $V\in\bar{\calU}_\delta$, and let $A^\parallel$ and $A^\perp$ range over all anti-Hermitian and Hermitian elements, respectively, of norm $<2\delta$. Then for each $A^\parallel$, the equation $f_V(A^\parallel+A^\perp)=0$ has a unique solution $A^\perp=g_V(A^\parallel)$. The function $g_V$ is continuously differentiable and satisfies the bounds
\begin{equation}
g_V(A^\parallel)=-f(V)+O(\eps\delta+\delta^2),\qquad
\frac{\partial g_V(A^\parallel)}{\partial A^\parallel} = O(\delta+\eps).
\end{equation}
\end{Lemma}
\begin{proof}
Using approximate associativity, we obtain the estimate
\[
f_V(A)=\frac{1}{2}\Bigl(V^{\dag}V-I
+(V^{\dag}V)A+A^{\dag}(V^{\dag}V)+A^{\dag}(V^{\dag}V)A+O(\eps\delta)\Bigr)
=\frac{1}{2}(V^{\dag}V-I)+A^\perp+O(\eps\delta+\delta^2),
\]
which implies that if $f_V(A)=0$, then $A^\perp=-f(V)+O(\eps\delta+\delta^2)$. By a similar calculation,
\[
\partial_{A}f_V(A):\: Z\,\to\, \frac{1}{2}
\Bigl(((I+A^\dag)V^{\dag})(VZ)+(Z^\dag V^\dag)(V(I+A))\Bigr)
=\frac{1}{2}(Z+Z^\dag)+O(\delta+\eps)\ts\|Z\|,
\]
and hence,
\[
\frac{\partial f_V(A^\parallel+A^\perp)}{\partial(A^\parallel,A^\perp)}
=(0,1)+O(\delta+\eps)\qquad
(A^\parallel\in\Ba^{i\calH}_{2\delta}(0),\:\,
A^\perp\in\Ba^{\calH}_{2\delta}(0)),
\]
where $(0,1)$ stands for the row vector of partial derivatives. By Lemma~\ref{lem_invfun}, for each given $A^\parallel$, the map $A^\perp\mapsto f_V(A^\parallel+A^\perp)$ is injective and its image includes the set $f_V(A^\parallel)+\Ba^\calH_{2\delta(1-O(\delta+\eps))}(0)$. Since $\|f_V(A^\parallel)\|\le \delta+O(\eps\delta+\delta^2)$, this set contains $0$, and so the equation $f_V(A^\parallel+A^\perp)=0$ has a unique solution, denoted by $g_V(A^\parallel)$. The bound on its derivative follows from the formula
\[
\frac{\partial g_V(A^\parallel)}{\partial A^\parallel}
=\biggl(-\left(\frac{\partial f(A^\parallel+A^\perp)}
{\partial A^\perp}\right)^{-1}\,
\frac{\partial f(A^\parallel+A^\perp)}
{\partial A^\parallel}\biggr)\bigg|_{A^\perp=g_V(A^\parallel)}.
\]
\end{proof}

The tangent space $\Ta_U\calU$ of the manifold $\calU$ at point $U$ is the image of the following map, which preserves the norm up to a $1+O(\eps)$ factor:
\begin{equation}
\frac{\partial\,\ph_U^{-1}(A^\parallel+g_U(A^\parallel))}{\partial A^\parallel}
=\La_U\left(1+\frac{\partial g_U(A^\parallel)}{\partial A^\parallel}\right):\:
i\calH\,\to\,\calA.
\end{equation}
Thus, $\Ta_U\calU$ is a closed subspace of $\calA$, which is almost-isometrically identified with $i\calH$ by an analogue of the Maurer-Cartan form, i.e.\ the inverse of the previous map:
\begin{equation}\label{MC_form}
\omega_U=\frac{\partial\,\ph_U^\parallel(X)}{\partial X}\bigg|_{X=U}
\colon\: Z\,\mapsto\, (\La_U^{-1}(Z))^\parallel,\qquad
%=\frac{1}{2}\biggl(\La_U^{-1}(Z)-(\La_U^{-1}(Z))^\dag\biggr),\qquad
\omega_U\colon \Ta_U\calU\to i\calH.
\end{equation}

\begin{Proposition}\label{prop_polar}
The polar map $(U,H)\mapsto UH$ diffeomorphically maps $\calU\times\Ba^\calH_{\delta}(I)$ onto some $S\subseteq\calA$ such that $\calU_{\delta-O(\eps\delta+\delta^2)}\subseteq S\subseteq \calU_{\delta+O(\eps\delta+\delta^2)}$.
\end{Proposition}

\begin{proof}
Let us use a variant of the polar map, $(U,Q)\mapsto U(I+Q)$ for $U\in\calU$ and $Q\in\Ba^\calH_{1.5\delta}(0)$. First, we note that if $V=U(I+Q)$ for some $U$ and $Q$, then $\|V-U\|=\|UQ\|\le 1.5\delta+O(\eps\delta+\delta^2)$, or in local coordinates,
\[
\|\ph_V(U)\|=\|\La_V^{-1}(U-V)\|\le 1.5\delta+O(\eps\delta+\delta^2)<2\delta.
\]
Therefore, for the purpose of finding the inverse image of $V$ under the polar map, it is sufficient to consider all elements $U=\ph_V^{-1}(A)$ with $A=A^\parallel+A^\perp$ such that $A^\parallel\in\Ba^{i\calH}_{2\delta}(0)$ and $A^\perp\in\Ba^\calH_{2\delta}(0)$. By Lemma~\ref{lem_gV}, the condition $U^\dag U=I$ is equivalent to $A^\perp=g_V(A^\parallel)$, so $U$ is parametrized by $A^\parallel$:
\[
U=\ph_V^{-1}(A)=V(I+A),\qquad
\text{where}\quad A=A^\parallel+g_V(A^\parallel)\qquad
(A^\parallel\in\Ba^{i\calH}_{2\delta}(0)).
\]
It is also convenient to represent the value of the polar map in the same coordinates:
\[
B =\ph_V(U(I+Q))=\La_V^{-1}\bigl((V(I+A))(I+Q)-V\bigr)
= A+Q+\La_V^{-1}((VA)Q).
\]
The condition $V=U(I+Q)$ is equivalent to $B=0$.

As a function of $A$ and $Q$, the element $B$ has the following derivatives:
\[
\frac{\partial B}{\partial A}=1+\La_V^{-1}\Ra_Q\La_V=1+O(\delta),\qquad
\frac{\partial B}{\partial Q}=1+\La_V^{-1}\La_{VA}=1+O(\delta).
\]
Expressing $B$ as $B^\parallel+B^\perp$ and $A$ as $A^\parallel+g_V(A^\parallel)$, we may regard $(B^\parallel,B^\perp)$ as a function of $(A^\parallel,Q)$. Using Lemma~\ref{lem_gV}, we get these bounds:
\begin{equation}\label{der_B_AQ}
(B^\parallel,B^\perp)=(A^\parallel,\,-f(V)+Q)+O(\eps\delta+\delta^2),\qquad
\frac{\partial(B^\parallel,B^\perp)}{\partial(A^\parallel,Q)}
=\begin{pmatrix}1&0\\ 0&1\end{pmatrix}+O(\eps+\delta).
\end{equation}
The second bound implies that (i) the map $p_V\colon (A^\parallel,Q)\mapsto(B^\parallel,B^\perp)$ is a diffeomorphism. Due to both bounds and Lemma~\ref{lem_invfun}, the equation $p_V(A^\parallel,Q)=0$ has a solution if $\|f(V)\|$ is less than $1.5\delta$ by an $O(\eps\delta+\delta^2)$ margin; furthermore, (ii) if $\|f(V)\|<\delta_-$ for a certain $\delta_-=\delta-O(\eps\delta+\delta^2)$, then $\|Q\|<\delta$. Conversely, (iii) if $\|Q\|<\delta$, then $\|f(V)\|<\delta+O(\eps\delta+\delta^2)$. The statements (i)--(iii) are exactly what we needed to prove.
\end{proof}

The inverse of the polar map is a pair of functions $u\colon \calU_{\delta_{\max}}\to\calU$ and $h\colon \calU_{\delta_{\max}}\to\calH$ such that $X=u(X)h(X)$ for all $X$. Proposition~\ref{prop_polar} implies that if $X\in\bar{\calU}_\delta$, then
\begin{equation}
\|X-u(X)\|\le \delta+O(\eps\delta+\delta^2).
\end{equation}
The derivatives of $v$ and $h$ are easy to estimate in the local coordinates based at an arbitrary point $V$ in an $O(\delta)$-neighborhood of $U$. Let us reuse the calculation from the previous proof, where 
$B^\parallel=\ph_V^\parallel(X)$,\, $B^\perp=\ph_V^\perp(X)$,\, $A^\parallel=\ph_V^\parallel(u(X))$, and $Q=h(X)-I$. Inverting the derivative in equation \eqref{der_B_AQ}, we get
\begin{equation}
\frac{\partial\bigl(\ph_V^\parallel(u(X)),\,h(X)\bigr)}
{\partial\bigl(\ph_V^\parallel(X),\,\ph_V^\perp(X)\bigr)}
=\begin{pmatrix}1&0\\ 0&1\end{pmatrix}+O(\eps+\delta).
\end{equation}

We are now in a position to define an approximate group structure on $\calU$. The unit $e\in\calU$, the multiplication map $\mu\colon \calU\times\calU\to\calU$, and the inversion map $\sigma\colon \calU\to\calU$ are defined as follows:
\begin{equation}\label{Ugr_ops}
e=I,\qquad \mu(U,V) = u(UV),\qquad \sigma(U)=u(U^\dag).
\end{equation}
It is evident that if $U,V\in\calU$, then $UV\in\bar{\calU}_{O(\eps)}$. Similarly,
\[
UU^\dag =\La_U\La_{U^\dag}(I) =I-\La_U(\La_U^{-1}-\La_{U^\dag})(I)=I+O(\eps)
\]
due to the second inequality in \eqref{u_inv_bound}, so $U^\dag\in\bar{\calU}_{O(\eps)}$. To deal with the map $u$ in equation \eqref{Ugr_ops}, we may work in $\bar{\calU}_{\delta_0}$, where $\delta_0$ is chosen such that $\delta_0=O(\eps)$ and
\begin{equation}
UV\in\bar{\calU}_{\delta_0},\quad\: U^\dag\in\bar{\calU}_{\delta_0}\qquad
\text{for all }\, U,V\in\calU.
\end{equation}
Thus, we obtain the approximate group laws
\begin{gather}
d\bigl(\mu(\mu(U,V),W),\, \mu(U,\mu(V,W))\bigr)\le O(\eps),\\[2pt]
d\bigl(\mu(\sigma(U),U), e\bigr)\le O(\eps),\qquad
d\bigl(\mu(U,\sigma(U)), e\bigr)\le O(\eps),
\end{gather}
where $d(X,Y)=\|X-Y\|$ is the metric on $\calU$. And, of course,
\begin{equation}
\mu(e,U)=\mu(U,e)=U,\qquad \sigma(e)=e.
\end{equation}
Various derivatives are perturbed by $1+O(\eps)$ factors relative to what one would expect in a Lie group. Thus,
\begin{gather}
\label{der_gr_product}
\frac{\partial\,\ph_{\mu(U,V)}^\parallel(\mu(X,Y))}
{\partial\,\ph_U^\parallel(X)}
\bigg|_{X=U,\,Y=V} \!\!= \La_V^{-1}\Ra_V+O(\eps),\qquad
\frac{\partial\,\ph_{\mu(U,V)}^\parallel(\mu(X,Y))}
{\partial\,\ph_U^\parallel(X)}
\bigg|_{X=U,\,Y=V} \!\!= 1+O(\eps),\\[3pt]
\label{der_gr_inverse}
\frac{\partial\,\ph_{\sigma(U)}^\parallel(\sigma(X))}
{\partial\,\ph_U^\parallel(X)}
\bigg|_{X=U} \!= -\La_{U}\Ra_{U}^{-1}+O(\eps).
\end{gather}

From the topological perspective, $\calU$ is an up-to-homotopy group, or an associative H-space with an inversion map. Let us elaborate on that. By definition, an \emph{H-space} is a topological space $M$ with a basepoint $e$ and a continuous multiplication map $\mu\colon M\to M$ such that $\mu(e,e)=e$ and
\begin{equation}\label{H-unit}
\bigl(x\mapsto \mu(x,e)\bigr)
\,\sim\,\bigl(x\mapsto x\bigr)
\,\sim\,\bigl(x\mapsto \mu(e,x)\bigr),
\end{equation}
where ``$f\sim g$'' means that $f$ and $g$ are homotopic through basepoint-preserving maps. The multiplication on $\calU$ satisfies this condition because the maps in questions are equal. An H-space $M$ is called \emph{associative} if
\begin{equation}
\bigl((x,y,z)\mapsto \mu(\mu(x,y),z)\bigr)
\,\sim\,\bigl((x,y,z)\mapsto \mu(x,\mu(y,z))\bigr).
\end{equation}
In our case, the homotopy can be constructed by projecting the straight path between $\mu(\mu(x,y),z)$ and $\mu(x,\mu(y,z))$ onto $\calU$ using the map $u$. An \emph{inversion map} is a continuous basepoint-preserving map $\sigma\colon M\to M$ such that
\begin{equation}\label{homo_inverse}
\bigl(x\mapsto \mu(\sigma(x),x)\bigr)
\,\sim\,\bigl(x\mapsto e\bigr)
\,\sim\,\bigl(x\mapsto \mu(x,\sigma(x))\bigr).
\end{equation}
These properties also hold for $\calU$. Surprisingly, the topological argument in the next section does not require associativity, and of the two homotopies in \eqref{homo_inverse}, only the first is used. When only the existence of that homotopy is required, $\sigma$ is called a \emph{left inversion map}.

\section{The existence of a nontrivial projection}\label{sec_projection}

A \emph{projection} in an $\eps$-$C^*$ algebra $\calA$ is a Hermitian element $P$ such that $P^2=P$. Such elements are generally hard to come by, see Remark~\ref{rem_X2}, so will content ourselves with \emph{$\delta$-projections} for sufficiently small $\delta$. They are defined by the conditions
\begin{equation}\label{delta_P}
P^\dag=P,\qquad \|P^2-P\|\le\delta.
\end{equation}
It follows from the second condition that
\[
(1-\eps)\|P\|^2-\delta \le \|P\| \le (1+\eps)\|P\|^2+\delta,
\]
and hence,
\begin{equation}\label{P_alternatives}
\|P\|\le O(\delta)\quad \text{or}\quad \bigl|\|P\|-1\bigr|\le O(\delta+\eps).
\end{equation}
A $\delta$-projection is called \emph{nonvanishing} if the second alternative holds. If $P$ is a $\delta$-projection, then $I-P$ is a $\delta'$-projection, where $\delta'=\delta$ if $\calA$ has exact unit and $\delta'=\delta+O(\eps)$ in general. A $\delta$-projection $P$ is called \emph{nontrivial} if both $P$ and $I-P$ are nonvanishing.

\begin{Lemma}\label{lem_nontriv_projection} Any $\eps$-$C^*$ algebra $\calA$ such that\, $1<\dim\calA<\infty$\, has a nontrivial $O(\eps)$-projection.
\end{Lemma}

\begin{proof}[Proof by reduction to a topological statement.]
Without loss of generality, we may assume that $\calA$ has exact unit. An element $P\in\calA$ is a projection if and only if $X=2P-I$ is Hermitian and satisfies the equation $X^2=I$. Equivalently, $X$ is both Hermitian and unitary. An approximate version of these conditions,
\begin{equation}\label{approx_X2I}
\|U^\dag-U\|\le\delta\qquad (U\in\calU),
\end{equation}
implies that $P=\frac{1}{4}(2I+U+U^\dag)$ is an $O(\delta+\eps)$-projection. In particular, fixed points of the inversion map $\sigma$, i.e.\ solutions of the equation
\begin{equation}\label{X2I}
\sigma(U)=U\qquad (U\in\calU),
\end{equation}
satisfy inequality \eqref{approx_X2I} with $\delta=O(\eps)$. Note that the trivial fixed points $I$ and $-I$ are isolated because the derivative of $\sigma$ is $O(\eps)$-close to $-1$ in local coordinates in constant-size neighborhoods of $\pm I$. It follows that any other fixed point is at least constant distance away from $\pm I$, and therefore, corresponds to a nontrivial $O(\eps)$-projection. It suffices to show that such a fixed point exists.

To simplify the problem a bit, let us consider the manifold $\breve{\calU}=\calU_e/\UU(1)$, where $\calU_e$ is the connected component of $\calU$ containing the unit element, and the quotient is taken with respect to the equivalence relation $U\equiv cU$ for $c\in\UU(1)=\{c\in\CC:\,|c|=1\}$. The maps $\mu$ and $\sigma$ preserve this equivalence relation, giving rise to some maps $\breve{\mu}$ and $\breve{\sigma}$ on $\breve{\calU}$. A point $\breve{U}\in\breve{\calU}$ is a fixed point of $\breve{\sigma}$ if and only if it is the equivalence class of some $U\in\calU_e$ such that $\sigma(U)=e^{i\phi}U$\, ($\phi\in\RR$). It corresponds to two fixed point of $U$, namely, $e^{i\phi/2}U$ and $-e^{i\phi/2}U$. We will show that if $n=\dim\calA$ is finite but greater that $1$, and thus, $0<\dim\breve{\calU}=n-1<\infty$, then $\breve{\sigma}$ has some fixed point $\breve{U}\not=\breve{e}$ (where $\breve{e}$ is the equivalence class of $e=I$, the unit of $\calU$ and the original algebra).\medskip

Let us list and explain the properties of $\breve{\calU}$ and $\breve{\sigma}$ that are needed to complete the proof:
\begin{enumerate}
\item $\breve{\calU}$ is a compact $C^1$ manifold. Indeed, $\calU$ is a $C^1$ manifold. Being a bounded closed subset of a finite-dimensional Banach space, it is compact, and so is its connected component $\calU_e$. The space $\breve{\calU}$ is the quotient of $\calU_e$ by a smooth free action of $\UU(1)$; this action is proper because $U(1)$ is compact. By the quotient manifold theorem \cite[theorem~21.10]{Lee-sm}, $\breve{\calU}$ is a $C^1$ manifold, and its compactness follows from the compactness of $\calU_e$.
\item $\breve{\calU}$ is orientable because the Maurer-Cartan form on $\calU$ trivializes the tangent bundles of both $\calU$ and $\breve{\calU}$.
\item $\breve{\sigma}\colon \breve{\calU}\to\breve{\calU}$ is a $C^1$ map with an isolated fixed point $\breve{e}$. (A fixed point of a map $f$ is called \emph{isolated} if $\det(1-f'(x))\not=0$, where $f'(x)$ is the tangent map, expressed as $\partial_yf(y)|_{y=x}$ in local coordinates.) Furthermore, $\det(1-\breve{\sigma}'(\breve{e}))>0$ because in local coordinates, $\partial_x\breve{\sigma}(x)|_{x=\breve{e}}=-1+O(\eps)$.
\item $\breve{\calU}$ is connected.
\item $\breve{\calU}$ is homeomorphic to a finite CW complex. This follows from the first property because every compact $C^1$ manifold admits a finite triangulation.
\item $\breve{\calU}$ is an H-space with the base point $\breve{e}$, the multiplication map $\breve{\mu}$, and a left inversion map $\breve{\sigma}$.
\end{enumerate}

In general, the number of fixed points of a smooth map $f$ on a compact orientable manifold $M$ can be estimated using the Lefschetz-Hopf theorem. It asserts that if $f$ has a finite number of fixed points and they are all isolated, then
\begin{equation}
\sum_{x\in\Fix(f)}\ind(f,x)=\Lambda(f).
\end{equation}
The \emph{fixed point index} $\ind(f,x)$ and the \emph{Lefschetz number} $\Lambda(f)$ are defined as follows:
\begin{alignat}{2}
\ind(f,x)&=\sgn\det(1-f'(x)),\qquad
& f'(x)\colon \Ta_xM&\to\Ta_xM,\\[2pt]
\Lambda(f)&=\sum_{k}(-1)^k\Tr f^{*k},\qquad
& f^{*k}\colon H^k(M;\RR)&\to H^k(M;\RR),
\end{alignat}
where $f^{*k}$ is the induced cohomology map in dimension $k$. Applying this theorem to $\breve{\sigma}\colon \breve{\calU}\to\breve{\calU}$, let us suppose for a moment that $\breve{e}$ is the only fixed point. It is isolated and has index $1$. On the other hand, Proposition~\ref{prop_H-group} below implies that the Lefschetz number $\Lambda(\breve{\sigma})$ equals the sum of the Betti numbers, $\sum_k\dim H^k(\breve{\calU};\RR)$. This sum is at least $2$ because $H^k(\breve{\calU};\RR)\not=0$ for $k=0$ and $k=\dim\breve{\calU}>0$. We have arrived at a contradiction, indicating that $\breve{\sigma}$ has another fixed point.
\end{proof}

\begin{Proposition}\label{prop_H-group}
Let $M$ be a connected CW complex such that $\dim H^*(M;\RR)<\infty$. If $M$ is an H-space with a left inversion map $\sigma$, then\, $\Tr\sigma^{*k}=(-1)^k\dim H^k(M;\RR)$.
\end{Proposition}

To make a long story short, $\sigma^{*k}$ is represented by a triangular matrix with $(-1)^k$ on the diagonal using Hopf's structure theorem and the augmentation filtration. The long story begins with some background information. A homogeneous element $a$ of the cohomology space $A^k=H^k(M;\RR)$ is said to have dimension $|a|=k$. The graded space  $A=\bigoplus_{k}A^k$ is a ring (or an $\RR$-algebra) with respect to the cup product, which has the property $a\smile b=(-1)^{|a|\,|b|}b\smile a$. The sign in this formula should be attributed to the swap map of the graded tensor product,
\begin{equation}
\Sw\colon A\otimes A\to A\otimes A,\qquad
a\otimes b\mapsto (-1)^{|a|\,|b|}b\otimes a.
\end{equation}
With this understanding, the cup product is commutative because ${\smile} = {\smile}\circ\Sw$, where the second expression means the swap map followed by
\begin{equation}
{\smile}\colon A\otimes A\to A.
\end{equation}
The swap map is used in many other cases. For example, to give $A\otimes A$ the structure of an algebra, one has to swap $b$ and $c$ in the expression $(a\otimes b)\smile(c\otimes d)$. The result is $(-1)^{|b|\,|c|}(a\smile c)\otimes(b\smile d)$.

Another important operation is the cross product ${\times}\colon H^*(M;\RR)\otimes H^*(M;\RR)\to H^*(M\times M;\RR)$. In fact, the cup product is obtained as the composition of the cross product with the cohomology map induced by the diagonal, $x\mapsto(x,x)\colon M\to M\times M$. The cross product is a homomorphism of graded algebras; it is an isomorphism if $M$ is a CW complex and $H^k(M;\RR)$ is finite-dimensional for each $k$. (The last statement is a variant of the Kunneth theorem, see Theorem~3.16 in~\cite{Hat-at}.) If $M$ is an H-space with the multiplication map $\mu$, there is also the induced cohomology map $\mu^*\colon H^*(M;\RR)\to H^*(M\times M;\RR)$. Combining it with the inverse of the cross product gives a \emph{comultiplication}
\begin{equation}
\Delta\colon A\to A\otimes A.
\end{equation}
There is also a \emph{counit} $\tau\colon A\to\RR$, which is induced by the map of an abstract point (with the cohomology space $\RR$) to the basepoint $e\in M$. The axiom \eqref{H-unit} of an H-space implies that\footnote{In this discussion, we denote identity maps by ``$\id$'' to avoid any confusion.}
\begin{equation}\label{coalg_prop}
(\id\otimes\tau)\circ\Delta = \id = (\tau\otimes\id)\circ\Delta.
\end{equation}
This property is dual to that of an algebra unit, in particular, of $I\in A$:
\begin{equation}\label{alg_prop}
I\smile a = a = a\smile I.
\end{equation}
The comultiplication is compatible with the cup product and the unit; namely, it is a homomorphism between graded algebras $A$ and $A\otimes A$:
\begin{equation}\label{comult_prop}
\Delta\circ{\smile} =({\smile}\otimes{\smile}) \circ (\id\otimes\Sw\otimes\id)
\circ (\Delta\otimes\Delta),\qquad \Delta(I)=I\otimes I.
\end{equation}
(In simpler terms, the first equation reads: $\Delta(a\smile b)=\Delta(a)\smile\Delta(b)$ for all $a,b\in A$.) The counit is also compatible, being an augmentation of the algebra $A$:
\begin{equation}\label{counit_prop}
\tau\circ{\smile}=\tau\otimes\tau,\qquad \tau(I)=I.
\end{equation}
The ``$\smile$'' symbol is often omitted for brevity, i.e.\ one writes $ab$ instead of $a\smile b$. Note that while the cup product is associative and commutative, the comultiplication may not be coassociative or cocommutative. It is coassociative (cocommutative) if the map $\mu$ is associative (resp.\ commutative) up to homotopy. Let us call a vector space (over $\RR$ or some other field) with a unit, a multiplication, a counit, and a comultiplication satisfying equations \eqref{coalg_prop}--\eqref{counit_prop} a \emph{bialgebra}.\footnote{We depart from the convention that bialgebras are both associative and coassociative. Milnor and Moore in their comprehensive study of graded Hopf algebras~\cite{MiMo65} called the more general algebraic structure ``quasi-Hopf algebra''. Unfortunately, both this term and ``quasi-bialgebra'' have received much more specific meanings.}

If the $H$-space $M$ is connected, the corresponding bialgebra splits (as a graded $\RR$-space) as $A=\RR\oplus A^+$, where $A^+=\bigoplus_{k>0}A^k$. Furthermore, the unit of $A$ coincides with the unit of $\RR$, and the counit acts as the identity map on $\RR$ and vanishes on $A^+$. We say that in this situation, $A$ is a \emph{connected commutative associative bialgebra}.\footnote{Hatcher \cite[section~3.C]{Hat-at} uses the term ``Hopf algebra''. The structure in question meets the standard definition of a Hopf algebra (including the existence of an antipode~\cite{MiMo65}) if the comultiplication is coassociative.} The connectivity together with equation \eqref{counit_prop} has an important corollary for the comultiplication. Specifically, applying the maps $\id\otimes\tau$ and $\tau\otimes\id$ to $\Delta(a)$ for $a\in A^+$, one finds that
\begin{equation}\label{comult_structure}
\Delta(a)=a\otimes I+I\otimes a+\sum_{j}a_j'\otimes a_j''\qquad
(a,a_j',a_j''\in A^+).
\end{equation}
A fundamental theorem of Hopf (Theorem~3C.4 in~\cite{Hat-at}) asserts that every connected commutative associative bialgebra over a field of characteristic $0$ is isomorphic, as a graded algebra, to the free commutative associative algebra on some homogeneous generators $x_j$. As a vector space, it is spanned by monomials $x_{j_1}^{p_1}\cdots x_{j_s}^{p_s}$ with $p_r\in\{1,2,\ldots\}$ if $|x_{j_r}|$ is even and $p_r=1$ if $|x_{j_r}|$ is odd, because the square of an odd-dimensional element is $0$. Finally, we take into account the condition $\dim A<\infty$. It excludes the existence of even-dimensional generators. 

Thus, if $M$ is a connected CW complex and an H-space, and the cohomology algebra $A=H^*(M;\RR)$ is finite-dimensional, then $A$ is isomorphic as a graded algebra to the exterior algebra on some odd-dimensional generators $x_1,\dots,x_m$. Note that this property does not completely determine the comultiplication or the map $\sigma^*$ we are interested in. To proceed, we need another standard construction, the \emph{augmentation filtration} $A\supseteq A^+\supseteq  (A^+)^2\supseteq  (A^+)^3\supseteq\cdots$, where $(A^+)^p$ is the ideal of $A$ generated by products of $p$ homogeneous elements of positive dimension. Similarly, the algebra $A\otimes A$, is filtered by the ideal $(A\otimes A)^+=\sum_{k'+k''>0}A^{k'}\otimes A^{k''}$ and its powers. In this notation, $\Delta(x_j)$ is congruent to $x_j\otimes I+I\otimes x_j$ modulo $((A\otimes A)^+)^2$. One can also determine $\Delta(x_1^{p_1}\cdots x_m^{p_m})$ modulo $((A\otimes A)^+)^{\sum_lp_l+1}$. For each $k$, we define the filtration
\begin{equation}\label{aug_filtration}
 A^k=F^{0,k}\supseteq F^{1,k-1}\supseteq\cdots\supseteq
F^{k,0}\supseteq F^{k+1,-1}=0,\qquad
\text{where}\quad F^{p,q}=(A^+)^p\cap A^{p+q}.
\end{equation}

\begin{proof}[Proof of Proposition~\ref{prop_H-group}.]
On general grounds, the cohomology map $\sigma^*$ induced by $\sigma\colon M\to M$ is an endomorphism of $A=H^*(M;\RR)$. Therefore, the obvious inclusion $\sigma^*(A^+)\subseteq A^+$ implies that $\sigma^*((A^+)^p)\subseteq(A^+)^p$ for all $p$. It follows that each individual map $\sigma^{*k}\colon A^k\to A^k$ preserves the filtration \eqref{aug_filtration}, and so there are well-defined maps of the quotient spaces:
\begin{equation}
\sigma^{*(p,q)}\colon E^{p,q}\to E^{p,q},\qquad
\text{where}\quad E^{p,q}=F^{p,q}/F^{p+1,q-1}.
\end{equation}
In these terms, the trace can be expressed as follows:
\begin{equation}\label{Tr_pq}
\Tr\sigma^{*k}=\sum_{p+q=k}\Tr\sigma^{*(p,q)}.
\end{equation}

Now, the condition that $\sigma$ is a left inversion map means that $x\mapsto\mu(\sigma(x),x)$ is homotopic to $x\mapsto e$ through basepoint-preserving maps. The corresponding cohomology maps are equal. Note that $x\mapsto\mu(\sigma(x),x)$ is the composition of three maps: $x\mapsto(x,x)$,\, $\sigma\times\id$, and $\mu$. Therefore,
\begin{equation}
{\smile}\circ(\sigma^*\otimes\id)\circ\Delta = I\circ\tau,
\end{equation}
where $I\in A$ on the right-hand side is interpreted as a map from $\RR$ to $A$. Combining the last equation with \eqref{comult_structure}, we get
\begin{equation}
\sigma^*(a)+a+\sum_{j}\sigma^*(a_j')a_j'' = 0\qquad (a,a_j',a_j''\in A^+).
\end{equation}
In particular, if $a=x_l$ is one of the generators of $A$, then $\sigma^*(a)$ is congruent to $-a$ modulo $(A^+)^2$. This condition determines $\sigma^*(a)$ for all basis elements of $A$ modulo a suitable ideal:
\begin{equation}
\sigma^*(x_1^{p_1}\cdots x_m^{p_m})
\equiv (-1)^{p}x_1^{p_1}\cdots x_m^{p_m}\quad
\bigl(\mathrm{mod}\:(A^+)^{p+1}\bigr),\qquad
p_1,\ldots,p_m\in\{0,1\},\quad p=\sum_{l}p_l.
\end{equation}
Note that $(-1)^{p}=(-1)^{|x_1^{p_1}\cdots x_m^{p_m}|}$ because all generators are odd-dimensional. On the other hand, the elements $x_1^{p_1}\cdots x_m^{p_m}\bmod (A^+)^{p+1}$ form a basis of the bigraded algebra $\bigoplus_{p,q}E^{p,q}$. Thus, $\sigma^{*(p,q)}=(-1)^{p+q}\ts\id$. In combination with \eqref{Tr_pq}, this implies that $\Tr\sigma^{*k}=(-1)^k\dim A^k$.
\end{proof}

\section{Subspaces associated with projections}\label{sec_subspaces}

To each pair of projections $(P,Q)$ in a $C^*$ algebra $\calB$, one assigns the vector space $\calS_{P,Q}=\{PXQ:\,X\in\calB\}$. In this section, we discuss an approximate version of this construction. Recall that a $\delta$-projection in an $\eps$-$C^*$ algebra $\calA$ is a Hermitian element $P$ such that $\|P^2-P\|\le\delta$. For each pair $(P,Q)$ of $\delta$-projections, there are two slightly different maps that could be called ``compressions'': $\La_P\Ra_Q\colon X\mapsto P(XQ)$ and $\Ra_Q\La_P\colon X\mapsto (PX)Q$. We will use their symmetric combination, $\frac{1}{2}(\La_P\Ra_Q+\Ra_Q\La_P)$. It is an $O(\delta+\eps)$-idempotent element of the algebra of operators acting on $\calA$, and therefore, can be approximated by an idempotent using Proposition~\ref{prop_P}. The \emph{compression map} thus defined,
\begin{equation}
\Co_{P,Q}\colon\calA\to\calA,\qquad
\Co_{P,Q}=\theta(\La_P\Ra_Q+\Ra_Q\La_P-1),
\end{equation}
satisfies the equations
\begin{equation}
\Co_{P,Q}^2=\Co_{P,Q},\qquad\:
\Co_{P,Q}(X)^\dag=\Co_{Q,P}(X^\dag)\quad\: (X\in\calA),
\end{equation}
and both $\|\La_P\Ra_Q-\Co_{P,Q}\|$ and $\|\Ra_Q\La_P-\Co_{P,Q}\|$ are bounded by $O(\delta+\eps)$. The image of this map, $\calS_{P,Q}=\Img\Co_{P,Q}=\Ker(1-\Co_{P,Q})$, is a closed linear subspace of $\calA$. There is a variant of this construction where only $\La_P$ or only $\Ra_Q$ is applied, and one writes $\Co_{P,1},\calS_{P,1}$ or $\Co_{1,Q},\calS_{1,Q}$, respectively. (This is the same as $\Co_{P,I}$, etc.\ if the unit is exact.)

When $P=Q$, the abbreviations $\Co_{P}$, $\calS_P$ are used. It is clear that $\calS_P=0$ if and only if $P$ is sufficiently close to $0$ as described by the first alternative in \eqref{P_alternatives}. (Recall that in the opposite case, we call $P$ ``nonvanishing''.) Similarly, $\Co_P=1$ and $\calS_P=\calA$ if and only if $P$ is close to $I$. If $\dim\calS_P=1$, we say that $P$ is a \emph{one-dimensional $\delta$-projection}.

If $P_1,P,Q_1,Q$ are $\delta$-projections such that $\|P_1P-P_1\|\le\delta$ and $\|Q_1Q-Q_1\|\le\delta$, then  $\calS_{P_1,Q_1}$ is almost contained in $\calS_{P,Q}$. Indeed, if $X=\Co_{P_1,Q_1}(X)$, there is a chain of approximate equalities with $O(\delta+\eps)\|X\|$ accuracy:
\[
X= \Co_{P_1,Q_1}(X)\approx P_1(XQ_1)\approx (PP_1)(X(Q_1Q))\approx
P\bigl(\bigl(P_1(XQ_1)\bigr)Q\bigr)\approx P(XQ)\approx \Co_{P,Q}(X).
\]
Thus,
\begin{equation}\label{almost_inc}
\|\Co_{P,Q}(X)-X\|\le O(\delta+\eps)\|X\|\qquad (X\in\calS_{P_1,Q_1}).
\end{equation}
For any triple $(P,Q,R)$ of $\delta$-projections, one defines the \emph{compressed product} between elements of the corresponding subspaces:
\begin{equation}\label{compr_prod}
(X,Y)\mapsto X\cdot Y\,\colon\,
\calS_{P,Q}\times\calS_{Q,R}\to \calS_{P,R},\qquad X\cdot Y=\Co_{P,R}(XY).
\end{equation}
It is close to the ambient product in $\calA$, namely, $\|X\cdot Y-XY\|\le O(\delta+\eps)\|X\|\ts\|Y\|$. If $P$ is a nonvanishing projection, the compressed product turns the Banach space $\calS_P$ (which is closed under the involution $X\mapsto X^\dag$) into an $O(\delta+\eps)$-$C^*$ algebra with unit $\wt{P}=\Co_{P}(P)$.\medskip

The proof of the main theorem involves breaking a given $\eps$-$C^*$ algebra into pieces and putting them back together. Let us begin with a basic property: if $P_1$ and $P_2$ are $\delta$-projections and $\|P_1P_2\|\le\delta$, then $P=P_1+P_2$ is an $O(\delta)$-projection and $\|P_jP-P_j\|\le O(\delta)$ for $j=1,2$. The next lemma is meant to construct a certain bijection and estimate its norm up to a constant factor. Some of the assumptions and other details are ultimately not important. In fact, the lemma will be used for $p,q\le 2$, so all reasonable variants of the inequalities involved differ only by constant factors.

\begin{Lemma}\label{lem_alpha}
Let $P,\dots,P_p$ and $Q_1,\dots,Q_q$ be Hermitian elements of an $\eps$-$C^*$ algebra $\calA$ such that
\[
\|P_jP_l-\delta_{jl}P_l\|\le\delta,\qquad \|Q_kQ_m-\delta_{km}Q_m\|\le\delta.
\]
Suppose that $P=\sum_jP_j$ and $Q=\sum_kQ_k$ are $\delta$-projections satisfying the additional conditions $\|P_jP-P_j\|\le\delta$ and $\|Q_kQ-Q_k\|\le\delta$, and that $pq(\delta+\eps)$ is less that a certain constant. Let
\begin{equation}\label{alpha_jk}
\alpha_{jk}\colon X\mapsto\Co_{P,Q}(X)\colon\calS_{P_j,Q_k}\to\calS_{P,Q},\qquad
\alpha=\sum_{j,k}\alpha_{jk}\colon \bigoplus_{j,k}\calS_{P_j,Q_k}\to\calS_{P,Q}.
\end{equation}
Then $\alpha$ is a linear bijection satisfying the bounds
\begin{equation}
\|\alpha\|\le pq+O(pq(\delta+\eps)),\qquad
\|\alpha^{-1}\|\le 1+O(pq(\delta+\eps)),
\end{equation}
where\, $\bigoplus_{j,k}\calS_{P_j,Q_k}$ is equipped with the maximum norm, $\|(X_{11},\dots,X_{pq})\|=\max_{j,k}\|X_{jk}\|$.
\end{Lemma}

\begin{proof}
Let $(X_{11},\dots,X_{pq})\in\bigoplus_{j,k}\calS_{P_j,Q_k}$. The conditions on $P$ and $Q$ imply $\|\Co_{P,Q}(X_{jk})-X_{jk}\|\le O(\delta+\eps)\|X_{jk}\|$. Since $\alpha(X_{11},\dots,X_{pq})=\sum_{j,k}\Co_{P,Q}(X_{jk})$, we get the required bound on $\|\alpha\|$.

Now, let
\[
\beta_{jk}\colon X\mapsto\Co_{P_j,Q_j}(X)
\colon\calS_{P,Q}\to\calS_{P_j,Q_k},\qquad
\beta=\sum_{j,k}\beta_{jk}\colon
\calS_{P,Q}\to\bigoplus_{j,k}\calS_{P_j,Q_k}.
\]
Then $\beta_{jk}\alpha_{lm}(X_{lm}) =\delta_{jl}\delta_{km}X_{lm}+O(\delta+\eps)\|X_{lm}\|$. Hence, $\beta\alpha_{lm}(X_{lm})$ is approximately equal to $(\ldots,0,X_{lm},0,\ldots)$, with the error bounded by $O(\delta+\eps)\|X_{lm}\|$. It follows that $\beta\alpha=1+\gamma$, where $1$ stands for the identity map on $\bigoplus_{j,k}\calS_{P_j,Q_k}$ and $\|\gamma\|\le pq\ts O(\delta+\eps)$. The last number is less than $1$ by one of the hypotheses, so $\beta\alpha$ is invertible and $(\beta\alpha)^{-1}\beta$ is a left inverse of $\alpha$. As a separate consideration,
\[
\|\alpha_{jk}\beta_{jk}(X)-(P_jX)Q_k\|\le O(\delta+\eps)\|X\|,
\]
and hence, $\|\alpha\beta-1\|\le pq\ts O(\delta+\eps)$, where $1$ is the identity map on $\calS_{P,Q}$. Similarly to the previous argument, $\alpha\beta$ is invertible and $\beta(\alpha\beta)^{-1}$ is a right inverse of $\alpha$. Of course, the two inverses are equal to each other. Thus, $\|\alpha^{-1}\|\le (1+O(pq(\delta+\eps)))\ts\|\beta\|$, and it is evident that $\|\beta\|\le 1+O(\delta+\eps)$. This gives the required bound on $\|\alpha^{-1}\|$.
\end{proof}

The rest of this section is concerned with special properties of one-dimensional projections.

\begin{Lemma}\label{lem_PQ_Hilb}
Let $P$ and $Q$ be $\delta$-projections in an $\eps$-$C^*$ algebra. If $Q$ is one-dimensional, then $\calS_{P,Q}$ is a Hilbert space with the inner product $\braket{\cdot}{\cdot}$ defined by the equation
\begin{equation}\label{1dQ_ip}
Y^\dag\cdot X=\braket{Y}{X}\,\wt{Q}\qquad (X,Y\in\calS_{P,Q}),
\end{equation}
where $Y^\dag\cdot X=\Co_{Q}(Y^{\dag}X)$ and $\wt{Q}=\Co_Q(Q)$. Furthermore,
\begin{equation}\label{1dQ_ip_norm}
\bigl|\braket{X}{X}-\|X\|^2\bigr|\le O(\delta+\eps)\ts\|X\|^2\qquad
(X\in\calS_{P,Q}).
\end{equation}
\end{Lemma}

\begin{proof}
The inner product $\braket{Y}{X}$ defined by \eqref{1dQ_ip} is linear in $X$ and conjugate linear in $Y$, and the complex conjugate of $\braket{Y}{X}$ equals $\braket{X}{Y}$. The positivity and the completeness axioms of a Hilbert space follow from \eqref{1dQ_ip_norm}, so it is sufficient to prove the latter. For this purpose, we may replace the space $\calS_{P,Q}$ with $\calS_{1,Q}$ because the former is almost contained in the latter, namely, $\|\Co_{1,Q}(X)-X\|\le O(\delta+\eps)\|X\|$ for all $X\in\calS_{P,Q}$.

It is evident that $\|X^{\dag}\cdot X\|$ and $\|\wt{Q}\|$ are $O(\eps+\delta)$-close to $\|X\|^2$ and $1$, respectively; therefore,
\begin{equation}\label{1dQ_norm1}
\bigl||\braket{X}{X}|-\|X\|^2\bigr|\le O(\delta+\eps)\ts\|X\|^2\qquad
(X\in\calS_{1,Q}).
\end{equation}
This inequality implies that if $\braket{X}{X}=0$, then $X=0$. To show the positivity, let us assume the opposite, i.e.\ that $\braket{X}{X}<0$ for some $X\in\calS_{1,Q}$. Note that $\braket{Q'}{Q'}>0$ for $Q'=\Co_{1,Q}(Q)$. Now, consider the expression $f(t)=\braket{X+tQ'}{X+tQ'}$ for $t\in[0,\infty)$. The function $f$ is continuous and changes sign from negative at $t=0$ to positive at sufficiently large $t$; hence, it vanishes at some $t>0$. It follows that $X=-tQ'$, and therefore, $\braket{X}{X}>0$, in contradiction with our previous assumption. Thus, $\braket{X}{X}\ge 0$ for all $X$, and so \eqref{1dQ_ip_norm} follows from \eqref{1dQ_norm1}.
\end{proof}

Let $Q$ be a one-dimensional $\delta$-projection in an $\eps$-$C^*$ algebra $\calA$. For any $\delta$-projections $P,R\in\calA$, we define a map $\Ha^{Q}_{P,R}\colon\calS_{P,R}\to\Bo(\calS_{R,Q},\calS_{P,Q})$. (Here $\Bo(\calS_{R,Q},\calS_{P,Q})$ is the space of bounded linear maps from $\calS_{R,Q}$ to $\calS_{P,Q}$ with the norm induced by the Euclidean norm $\|X\|_\Euc=\sqrt{\braket{X}{X}}$. Due to inequality \eqref{1dQ_ip_norm}, the latter is equal to $\|X\|$ up to a $1\pm O(\eps+\delta)$ factor.) For each $Z\in\calS_{P,R}$ and $X\in\calS_{R,Q}$, the element $\Ha^{Q}_{P,R}(Z)(X)\in\calS_{P,Q}$ is defined by the condition
\begin{equation}\label{Ha_def}
(Y^\dag\cdot Z)\cdot X+Y^\dag\cdot(Z\cdot X)
=2\,\bbraket{Y}{\Ha^{Q}_{P,R}(Z)(X)}\,\wt{Q}\quad \text{for all }\,Y\in\calS_{P,Q}.
\end{equation}
This symmetric definition enjoys the exact equality
\begin{equation}\label{Ha_dag}
\Ha^{Q}_{R,P}(Z^\dag)=\Ha^{Q}_{P,R}(Z)^\dag.
\end{equation}
On the other hand, $\|\Ha^{Q}_{P,R}(Z)(X)-Z\cdot X\|_\Euc\le O(\delta+\eps)\ts\|Z\|\ts\|X\|_\Euc$. Thus,
\begin{equation}\label{Ha_prod}
\|\Ha^{Q}_{P,R}(Z\cdot W)-\Ha^{Q}_{P,S}(Z)\ts\Ha^{Q}_{S,R}(W)\|
\le O(\delta+\eps)\ts\|Z\|\ts\|W\|,
\end{equation}
where the norm on the left is induced by the Euclidean norm. Let us consider some special cases. The map $\Ha^{Q}_{P,P}\colon\calS_{P}\to\Bo(\calS_{P,Q})$ is an $O(\delta+\eps)$-homomorphism due to \eqref{Ha_dag}, \eqref{Ha_prod}, and the inequality $\|\Ha^{Q}_{P,P}(\wt{P})-1_{\calS_{P,Q}}\|\le O(\delta+\eps)$. The domain of $\Ha^{Q}_{P,Q}$ is $\calS_{P,Q}\subseteq\calA$ with the ambient norm, whereas the target space is, effectively, $\calS_{P,Q}$ with the Euclidean norm because $\calS_{Q,Q}=\CC\ts\wt{Q}\cong\CC$. Furthermore, $\Ha^{Q}_{P,Q}$ is $O(\delta+\eps)$-close to the identity map, and so is $\Ha^{Q}_{Q,P}\colon\calS_{Q,P}\to\calS_{Q,P}$.

\begin{Lemma}\label{lem_PQR}
Let $P$, $Q$, $R$ be $\delta$-projections in an $\eps$-$C^*$ algebra. If $Q$ is one-dimensional, then
\[
\bigl|\|X\cdot Y\|-\|X\|\ts\|Y\|\bigr|\le O(\delta+\eps)\ts\|X\|\ts\|Y\|\qquad
(X\in\calS_{P,Q},\:\, Y\in\calS_{Q,R}).
\]
\end{Lemma}

\begin{proof}
Note that $\|X\cdot Y\|\approx\|XY\|\approx\sqrt{\|(XY)^{\dag}(XY)\|}$ with $O(\delta+\eps)\ts\|X\|\|Y\|$ accuracy. Now,
\[
(XY)^{\dag}(XY)=Y^\dag((X^{\dag}X)Y)+O(\eps)\ts\|X\|^2\|Y\|^2
=\|X\|^2\ts Y^{\dag}Y+O(\delta+\eps)\ts\|X\|^2\|Y\|^2
\]
because $X^{\dag}X\approx X^\dag\cdot X =\braket{X}{X}\wt{Q}\approx \|X\|^2\wt{Q}$ with $O(\delta+\eps)\ts\|X\|^2$ accuracy.
\end{proof}

\begin{Lemma}\label{lem_1d_proj}
If $P$ and $Q$ are one-dimensional $\delta$-projections in an $\eps$-$C^*$ algebra, then $\dim\calS_{P,Q}\le 1$.
\end{Lemma}

\begin{proof}
Let us suppose the contrary. Since $Q$ is one-dimensional, there exist two elements $X,Y\in\calS_{P,Q}$ that are orthogonal and have unit length with respect to the inner product \eqref{1dQ_ip}. Thus, $Y^\dag\cdot X=0$ while $\|X\|$ and $\|Y\|$ being $O(\delta+\eps)$-close to $1$. But this contradicts Lemma~\ref{lem_PQR} (for the projections $Q,P,Q$) because $P$ is one-dimensional.
\end{proof}

One-dimensional $\delta$-projections $P$ and $Q$ are called \emph{equivalent} if $\dim\calS_{P,Q}=1$. This is indeed an equivalence relation; in particular, it is transitive due to Lemma~\ref{lem_PQR}.

\section{Properties and approximation of \texorpdfstring{$\delta$}{delta}-homomorphisms}

We first consider general properties of $\delta$-homomorphisms between $\eps$-$C^*$ algebras. The approximation of $\delta$-homomorphisms by $O(\eps)$-homomorphisms is discussed in the case where the domain is a finite-dimensional $C^*$ algebra. As before, all statements are implicitly quantified like this: ``There exist some positive constants $\eps_{\max}$ and $\delta_{\max}$ such that for all $\eps\le\eps_{\max}$ and $\delta\le\delta_{\max}$ \ldots'' Only in Corollary~\ref{cor_improvement}, which is directly used in the proof of the main theorem, the constants are explicit.

\begin{Proposition}\label{prop_delta_hominc}
Let $v\colon \calA'\to\calA''$ be a non-unital $\delta$-homomorphism of $\eps$-$C^*$ algebras. Then in general, $\|v\|\le 1+O(\delta+\eps)$. If $\|v(X)\|\ge\eta\|X\|$ for some $\eta>2\delta$ and all $X$, then $\|v(X)\|\ge a\|X\|$ for $a=1-O(\delta+\eps)$ and all $X$. Furthermore, if $\|v(I_{\calA'})-I_{\calA''}\|$ is less than a certain positive constant, then $\|v(I_{\calA'})-I_{\calA''}\|\le O(\delta+\eps)$.
\end{Proposition}

\begin{proof}
Let
\[
a=\inf\{\|v(X)\|/\|X\|:\,X\not=0\},\qquad
b=\|v\|=\sup\{\|v(X)\|/\|X\|:\,X\not=0\}.
\]
Due to the $C^*$ property and the $\delta$-homomorphism condition $\|v(X^{\dag}X)-v(X)v(X)^\dag\|\le\delta\|X\|^2$, we have
{\setlength\abovedisplayskip{-3pt}\begin{gather*}
(1-\eps)\|X\|^2\le\|X^{\dag}X\|\le(1+\eps)\|X\|^2,\\[2pt]
\sqrt{(\|v(X^{\dag}X)\|-\delta\|X\|^2)/(1+\eps)}
\le\|v(X)\|\le \sqrt{(\|v(X^{\dag}X)\|+\delta\|X\|^2)/(1-\eps)}.
\end{gather*}}%
From these inequalities and the fact that $a\|X^{\dag}X\|\le \|v(X^{\dag}X)\|\le b\|X^{\dag}X\|$,\smallskip\ we obtain the bound $\sqrt{((1-\eps)a-\delta)/(1+\eps)}\ts\|X\|\le \|v(X)\| \le\sqrt{((1+\eps)b+\delta)/(1-\eps)}\ts\|X\|$. Thus,
\[
a\ge\sqrt{((1-\eps)a-\delta)/(1+\eps)},\qquad
b\le\sqrt{((1+\eps)b+\delta)/(1-\eps)}.
\]
The second inequality makes sense if $\eps<1$ and implies that $b\le 1+O(\delta+\eps)$. Similarly, the first inequality is valid if $a>2\delta$ and $\eps<1/2$, and in this case, $a\ge 1-O(\delta+\eps)$.

Now, suppose that $\|v(I_{\calA'})-I_{\calA''}\|<\delta_0$, where $\delta_0$ is a sufficiently small positive constant. For all $X\in\Ba_{\delta_0}(I_{\calA''})$, we have the estimate \eqref{d_X2X}, namely, $\|\partial_{X}(X^2-X)-1_{\calA''}\|\le O(\delta_0+\eps)$. Together with Lemma~\ref{lem_invfun}, it implies the bound
\begin{equation}\label{X2X_I2I}
\|(X^2-X)-(I_{\calA''}^2-I_{\calA''})\|\ge(1-O(\delta_0+\eps))\ts\|X-I_{\calA''}\|.
\end{equation}
Let $X=v(I_{\calA'})$; then $\|X^2-v(I_{\calA'}^2)\|\le O(\delta)$ and $\|v(I_{\calA'}^2)-X\|\le\|v\|\ts\|I_{\calA'}^2-I_{\calA'}\|\le O(\eps)$. Hence, $\|X^2-X\|\le O(\delta+\eps)$, and so the bound $\|X-I_{\calA''}\|\le O(\delta+\eps)$ follows from inequality \eqref{X2X_I2I}.
\end{proof}

\begin{Lemma}\label{lem_approx}
For any $\delta$-homomorphism $v$ from a finite-dimensional $C^*$ algebra $\calB$ to an $\eps$-$C^*$ algebra $\calA$, there is an $O(\eps)$-homomorphism $\tilde{v}\colon \calB\to\calA$ such that $\|\tilde{v}(X)-v(X)\|\le O(\delta)\|X\|$ for all $X\in\calB$.
\end{Lemma}

The proof involves the concept of a diagonal. Let us give its general definition for completeness, and then specialize to the finite-dimensional case. For arbitrary Banach algebras $\calA$ and $\calB$, the projective tensor product $\calA\hotimes\calB$ is also a Banach algebra. In more detail, $\calA\hotimes\calB$ is the completion of $\calA\otimes\calB$ endowed with the projective tensor norm,
\begin{equation}
\|C\|=\inf\biggl\{\sum_{j}\|A_j\|\ts\|B_j\|:\,
\sum_{j}A_j\otimes B_j=C\biggr\}\qquad (A_j\in\calA,\,\: B_j\in\calB),
\end{equation}
which has the property $\|C'C''\|\le\|C'\|\ts\|C''\|$. For arbitrary $\tilde{D}=\sum_{j}A_j\otimes B_j\in\calB\otimes\calB$ and $X\in\calB$, we write
\begin{equation}
X\tilde{D}=\sum_{j}XA_j\otimes B_j,\qquad
\tilde{D}X=\sum_{j}A_j\otimes B_jX,\qquad
\pi(\tilde{D})=\sum_{j}A_jB_j.
\end{equation}
These operations extend from $\calB\otimes\calB$ to $\calB\hotimes\calB$. An element $D\in\calB\hotimes\calB$ is called a \emph{diagonal} if
\begin{equation}
XD=DX\quad\text{for all }\, X\in\calB,\qquad \pi(D)=I_\calB.
\end{equation}
(Johnson used a weaker concept of approximate diagonal~\cite[definition~1.1]{Joh72}, which is better suited for the infinite-dimensional setting.)

Every finite-dimensional $C^*$ algebra has a standard diagonal, $D=\int dU\, (U^\dag\otimes U)$, where the integral is taken with respect to the Haar measure on the unitary group. Note that $\|D\|=1$ because the integral can be approximated by finite sums, i.e.\ convex combinations of $U^\dag\otimes U$. Due to Caratheodory's theorem and the compactness of the unitary group, $D$ itself is representable as such a convex combination. Thus,
\begin{equation}\label{finite_diag}
D=\sum_jA_j\otimes B_j,\qquad \sum_j\|A_j\|\ts\|B_j\|=1.
\end{equation}
One can show that $D$ is the unique norm $1$ diagonal and that $A_j\otimes B_j=p_j\ts U_j^\dag\otimes U_j$, where $p_j\ge 0$,\, $\sum_jp_j=1$, and $U_j$ is unitary. For the algebra of operators acting in a $d$-dimensional Hilbert space, a representation of $D$ in this form is given by generalized Pauli operators $S_{jk}$:
\begin{equation}\label{Pauli_diag}
D=\sum_{j=0}^{d-1}\sum_{k=0}^{d-1}d^{-2}S_{jk}^\dag\otimes S_{jk},\qquad
\braket{e_l}{S_{jk}e_m}=e^{2\pi ikm/d}\,\delta_{l,\,m+j\bmod d},
\end{equation}
where $\braket{e_l}{S_{jk}e_m}$ are the matrix elements of $S_{jk}$ in some orthonormal basis $\{e_0,\dots,e_{d-1}\}$. The diagonal of $\bigoplus_{l=1}^{m}\Bo(\CC^{d_l})$ is obtained by combining the component diagonals $D_l=\sum_j p_{lj}\ts U_{lj}^\dag\otimes U_{lj}$ into a sum over $j=(j_1,\dots,j_m)$ with $p_{j_1,\dots,j_m}=p_{1j_1}\cdots p_{mj_m}$ and $U_{j_1,\dots,j_m}=U_{1j_1}\oplus\cdots\oplus U_{mj_m}$.

\begin{proof}[Proof of Lemma~\ref{lem_approx}.] The \emph{multiplicativity defect} of a linear map $u\colon\calB\to\calA$ is defined as follows:
\begin{equation}
G_u\colon \calB\times\calB\to\calA,\qquad G_u(X,Y)=u(XY)-u(X)u(Y).
\end{equation}
Let $g=G_v$ be the multiplicativity defect of the $\delta$-homomorphism $v$. By a simple calculation using the associativity of $\calB$, the $\eps$-associativity of $\calA$, and the bound $\|v\|\le\const$ (which follows from Proposition~\ref{prop_delta_hominc}), we obtain an approximate 2-cocycle equation:
\begin{equation}
\begin{aligned}
\hspace{1.5cm}&\hspace{-1.5cm} v(X)g(Y,Z)-g(XY,Z)+g(X,YZ)-g(X,Y)v(Z)\\
&= -v\bigl((XY)Z-X(YZ)\bigr)+\bigl(v(X)v(Y)\bigr)v(Z)-v(X)\bigl(v(Y)v(Z)\bigr)\\
&= O(\eps)\ts\|X\|\ts\|Y\|\ts\|Z\|.
\end{aligned}
\end{equation}
We will consider linear maps $u\colon\calB\to\calA$ of the general form $u=v+w$ such that $\|w\|\le O(\delta)$. The multiplicativity defect of such a map is
\begin{equation}
G_{v+w}(X,Y)=g(X,Y)-F_w(X,Y)+O(\delta^2)\ts\|X\|\ts\|Y\|,
\end{equation}
where
\begin{equation}
F_w(X,Y)=v(X)w(Y)-w(XY)+w(X)v(Y).
\end{equation}

First, let us try this version of $w$, which uses a diagonal of the form \eqref{finite_diag}:
\begin{equation}
w'(X) = \sum_{j}v(A_j)g(B_j,X).
\end{equation}
We have $\|w'(X)\|\le O(\delta)\ts\|X\|$ since $v$ is a $\delta$-homomorphism and $\sum_{j}\|A_j\|\ts\|B_j\|=1$. In the calculation of $F_{w'}$, we will also use a corollary of the equation $\sum_{j}XA_j\otimes B_j=\sum_{j}A_j\otimes B_jX$, namely, $\sum_{j}v(XA_j)g(B_j,Y) =\sum_{j}v(A_j)g(B_jX,Y)$:
\begin{align}\label{g1_calc}
F_{w'}(X,Y) &=\sum\nolimits_j
\bigl(v(X)(v(A_j)g(B_j,Y))-v(A_j)g(B_j,XY)+(v(A_j)g(B_j,X))v(Y)\bigr)
\nonumber\\[2pt]
&= \begin{aligned}[t]
&\sum\nolimits_j
\bigl(v(XA_j)g(B_j,Y)-v(A_j)g(B_j,XY)+v(A_j)(g(B_j,X)v(Y))\bigr)\\
&+O(\eps\delta+\delta^2)\ts\|X\|\ts\|Y\|\quad
\textstyle (\text{by $\eps$-associativity and because $v$ is a $\delta$-homomorphism})
\end{aligned}
\nonumber\\[2pt]
&= \begin{aligned}[t]
&\sum\nolimits_{j}v(A_j)\bigl(g(B_jX,Y)-g(B_j,XY)+g(B_j,X)v(Y)\bigr)
+O(\delta\eps+\delta^2)\ts\|X\|\ts\|Y\|\\
&\textstyle (\text{because } \sum_{j}v(XA_j)g(B_j,Y)=\sum_{j}v(A_j)g(B_jX,Y))
\end{aligned}
\nonumber\\[2pt]
&= \sum\nolimits_{j}v(A_j)\bigl(v(B_j)g(X,Y)\bigr)+O(\delta^2+\eps)\ts\|X\|\ts\|Y\|\quad\:
(\text{due to the 2-cocycle equation})
\nonumber\\[2pt]
&=g(X,Y)+O(\delta^2+\eps)\ts\|X\|\ts\|Y\|\quad\:
{\textstyle (\text{because } \sum\nolimits_jA_jB_j=I_\calB)}.
\end{align}
Thus, $G_{v+w'}(X,Y)\le O(\delta^2+\eps)\ts\|X\|\ts\|Y\|$. In general, $w'$ does not commute with the involution, so let us also consider $w''(X)=w'(X^\dag)^\dag$. It is easy to see that
\[
g(Y^\dag,X^\dag)^\dag=g(X,Y),\qquad 
F_{w''}(X,Y)=F_{w'}(Y^\dag,X^\dag)^\dag
=g(X,Y)+O(\delta^2+\eps)\ts\|X\|\ts\|Y\|.
\]
Hence, $v^{(1)}=v+\frac{1}{2}(w'+w'')$ has the multiplicativity defect bounded by $O(\delta^2+\eps)\ts\|X\|\ts\|Y\|$ as well as commuting with the involution, which implies that $v^{(1)}$ is an $O(\delta^2+\eps)$-homomorphism. (The approximate preservation of the unit follows from the last part of Proposition~\ref{prop_delta_hominc}.)

Iterating the procedure as in Newton's method, we obtain some $v^{(s)}$ for $s=1,2,\ldots$. If $\eps>0$, a suitable $\tilde{v}=v^{(s)}$ (with the corresponding $\delta_s$ bounded by $O(\eps)$) is obtained after a finite number of steps. If $\eps=0$, then $\tilde{v}=\lim_{s\to\infty}v^{(s)}$ is a homomorphism of $C^*$ algebras. (This is a special case of Theorem~3.1 from~\cite{Joh88}.)
\end{proof}

In conclusion, let us formulate a straightforward corollary of Lemma~\ref{lem_approx} and Proposition~\ref{prop_delta_hominc}, where all constants are mentioned explicitly.
\begin{Corollary}[Error reduction]\label{cor_improvement}
There exist some positive constants $\eps_{\max}$, $\delta_{\max}$, and $c_0$ such that for all $\eps<\eps_{\max}$, if a finite-dimensional $C^*$ algebra $\calB$ is $\delta_{\max}$-included into an $\eps$-$C^*$ algebra $\calA$, there is also a $c_0\eps$-inclusion. If the original inclusion is bijective, then so is the new inclusion.
\end{Corollary}

\section{Proof of the main theorem}\label{sec_proof_main}

The proof of Theorem~\ref{th_main} involves finding nontrivial projections in algebras of the form $\calS_{P}$ defined at the beginning of Section~\ref{sec_subspaces}. We also need some merging and extension lemmas.

\begin{Lemma}\label{lem_merging}
Let $\Pi_1,\Pi_2$ be projections in a $C^*$ algebra $\calB$ such that $\Pi_1+\Pi_2=I$, and similarly, let $P_1,P_2$ be $\delta$-projections in an $\eps$-$C^*$ algebra $\calA$ such that $\|P_1+P_2-I\|\le\delta$. Consider some linear maps $\gamma_{jk}\colon\calS_{\Pi_j,\Pi_k}\to\calS_{P_j,P_k}$ (for $j,k\in\{1,2\}$) satisfying the following conditions:
\begin{align}
\label{merging0}
\gamma_{kj}(X^\dag)&=\gamma_{jk}(X)^\dag,\\[2pt]
\label{merging1}
\|\gamma_{jl}(XY)-\gamma_{jk}(X)\cdot\gamma_{kl}(Y)\|
&\le\delta\ts\|X\|\ts\|Y\|,\\[2pt]
\label{merging2}
\|\gamma_{jj}(\Pi_j)-P_j\|&\le\delta,\\[2pt]
\label{merging3}
(1-\delta)\ts\|X\|\le\|\gamma_{jk}(X)\| &\le(1+\delta)\ts\|X\|.
\end{align}
(The dot in \eqref{merging1} denotes the compressed product \eqref{compr_prod}.) Then the combined map
\begin{equation}
\gamma\colon
\begin{pmatrix}X_{11} & X_{12}\\ X_{21} & X_{22}\end{pmatrix}
\mapsto\sum_{j,k}\gamma_{jk}(X_{jk})\,\colon\,\calB\to\calA\qquad
(X_{jk}\in\calS_{\Pi_j,\Pi_k})
\end{equation}
is an $O(\delta+\eps)$-inclusion. If all maps $\gamma_{jk}$ are bijective, then $\gamma$ is also bijective.
\end{Lemma}

\begin{proof}
Let $\mu$ be the canonical bijection from $\calB$ to $\bigoplus_{j,k}\calS_{\Pi_j,\Pi_k}$ equipped with the maximum norm. In this notation, $\gamma=\alpha\bigl(\sum_{j,k}\gamma_{jk}\bigr)\mu$, where $\alpha$ is defined by equation \eqref{alpha_jk}. Equations \eqref{merging0}--\eqref{merging2} imply that $\gamma$ is an $O(\delta)$-homomorphism. It follows from equation \eqref{merging3} and Lemma~\ref{lem_alpha} that $a\|X\|\le\|\gamma(X)\|\le b\|X\|$ for all $X\in\calB$, where $a$ and $b$ are some positive constants. Thus, $\gamma$ is an $O(\delta+\eps)$-inclusion due to Proposition~\ref{prop_delta_hominc}. If all maps $\gamma_{jk}$ are bijective, then the direct sum $\sum_{j,k}\gamma_{jk}$ is also bijective, and so is $\gamma$ because both $\alpha$ and $\mu$ are linear bijections.
\end{proof}

\begin{Corollary}[Merging]\label{cor_merge_sum}
Let $\calB_1$ and $\calB_2$ be $C^*$ algebras, and let $P_1,P_2$ be $\delta$-projections in an $\eps$-$C^*$ algebra $\calA$ such that $\|P_1+P_2-I\|\le\delta$. Consider some $\delta$-inclusions $v_j\colon\calB_j\to\calS_{P_j}$ (for $j=1,2$). Then the combined map
\begin{equation}
v\colon\, (X_1,X_2)\mapsto v_1(X_1)+v_2(X_2)\,\colon\,
\calB_1\oplus\calB_2\to\calA
\end{equation}
is an $O(\delta+\eps)$-inclusion. If $v_1,v_2$ are bijective and $\calS_{P_1,P_2}=0$, then $v$ is also bijective.
\end{Corollary}

The second application of Lemma~\ref{lem_merging} involves an auxiliary statement about $\delta$-projections $P_1,\dots,P_p,Q_1,\dots,Q_q$ satisfying certain conditions, which will be formulated in terms of inclusions of finite-dimensional commutative $C^*$ algebras. Such an algebra $\calB$ is described by a \emph{projection basis} $\{\Pi_1,\dots,\Pi_n\}\subseteq\calB$ such that $\Pi_j^\dag=\Pi_j$,\, $\Pi_j\Pi_k=\delta_{jk}\Pi_k$, and $\sum_{j=1}^{n}\Pi_j=I$.

\begin{Lemma}\label{lem_add_dim}
Let $\calB$ and $\calC$ be commutative $C^*$ algebras with projection bases $\{\Pi_1,\dots,\Pi_p\}\subseteq\calB$ and $\{\Sigma_1,\dots,\Sigma_p\}\subseteq\calC$, respectively. Consider non-unital $\delta$-inclusions $v\colon\calB\to\calA$ and $w\colon\calC\to\calA$, where $\calA$ is an $\eps$-$C^*$ algebra. Let $P_j=v(\Pi_j)$,\, $Q_k=w(\Sigma_k)$, and also $P=v(I_\calB)$, $Q=w(I_\calC)$. Then there is a bounded linear bijection between $\calS_{P,Q}$ and $\bigoplus_{j,k}\calS_{P_j,Q_k}$; in particular, $\dim\calS_{P,Q}=\sum_{j,k}\dim\calS_{P_j,Q_k}$.
\end{Lemma}

\begin{proof}
The non-unital $\delta$-inclusion condition implies that $P_j$ and $P_k$, as well as $P_{[1,j]}=\sum_{r=1}^{j}P_r$ and $Q_{[1,k]}=\sum_{s=1}^{k}Q_s$, are $\delta$-projections. Furthermore, $\|P_{[1,j-1]}P_j\|\le\delta$ and $\|Q_{[1,k-1]}Q_k\|\le\delta$. By Lemma~\ref{lem_alpha}, there is a bounded linear bijection between $\calS_{P_j,Q_{[1,k]}}$ and $\calS_{P_j,Q_{[1,k-1]}}\oplus\calS_{P_j,Q_k}$ for each $k$. The composition of these bijections yields a bijection between $\calS_{P_j,Q}$ and $\bigoplus_{k}\calS_{P_j,Q_k}$. Similarly, a bijection exists between $\calS_{P_{[1,j]},Q}$ and $\calS_{P_{[1,j-1]},Q}\oplus\calS_{P_j,Q}$ for each $j$, and thus, between $\calS_{P,Q}$ and $\bigoplus_{j}\calS_{P_j,Q}$. This immediately implies the required statement.
\end{proof}

Let us recall that $\Ma{n}=\Bo(\CC^n)$ stands for the $C^*$ algebra of complex $n\times n$ matrices. More generally, $\Ma{n,k}=\Bo(\CC^k,\CC^n)$ is the space of $n\times k$ matrices, i.e.\ linear maps from $\CC^k$ to $\CC^n$. We consider higher-dimensional spaces as extensions of lower-dimensional spaces: $\CC^1\subset\CC^2\subset\cdots$. Thus, $\Ma{1}\subset \Ma{2}\subset\cdots$. The algebra $\Ma{n}$ has a standard basis $\{E_{jk}:\,j,k=1,\dots,n\}$ with the relations
\begin{equation}
E_{jk}E_{lm}=\delta_{kl}E_{jm},\qquad E_{jk}^\dag=E_{kj},\qquad
I=\sum_{j}E_{jj}.
\end{equation}
It contains an $n$-dimensional commutative subalgebra, which consists of diagonal matrices and has a projection basis $\{E_{jj}:\,j=1,\dots,n\}$.

\begin{Lemma}[Extension]\label{lem_extension}
Let $\calA$ be an $\eps$-$C^*$ algebra, and let $P,Q\in\calA$ be $\delta$-projections such that $\|P+Q-I\|\le\delta$. Suppose that $v\colon \Ma{n}\to\calS_{P}$ is a $\delta$-isomorphism, $\dim\calS_Q=1$, and $\calS_{P,Q}\not=0$. Then $v$ can be extended to an $O(\delta+\eps)$-isomorphism $v_+\colon\Ma{n+1}\to\calA$.
\end{Lemma}

\begin{proof}
The vector space $\calS_{P,Q}$ will play a key role; let us determine its dimensionality. Since $v\colon \Ma{n}\to\calS_{P}$ is a $\delta$-isomorphism, the elements $P_j=v(E_{jj})$ are one-dimensional $\delta$-projections in $\calS_P$, which are equivalent to each other. Hence, $P_1,\dots,P_n$ are equivalent one-dimensional $O(\delta+\eps)$-projections in $\calA$, implying that either $\dim\calS_{P_j,Q}=0$ for all $j$ or $\dim\calS_{P_j,Q}=1$ for all $j$. Using Lemma~\ref{lem_add_dim}, we conclude that $\dim\calS_{P,Q}$ is either $0$ or $n$. The first possibility contradicts the hypothesis of the Lemma; thus, $\dim\calS_{P,Q}=n$.
%We will later construct a particular linear bijection $U_1\colon\CC^n\to\calS_{P,Q}$. There is also an obvious bijection $U_2\colon\CC\to\calS_{Q,Q}$, namely, $U_2(c)=c\wt{Q}$.

In what follows, we use the notation
\begin{equation}
P_1=P,\quad P_2=Q,\qquad
h_{jk}=\Ha^{Q}_{P_j,P_k}\colon\calS_{P_j,P_k}\to\Bo(\calS_{P_k,Q},\calS_{P_j,Q}).
\end{equation}
Here $\calS_{P_j,Q}$ is regarded as a Hilbert space with the Hermitian inner product given by Lemma~\ref{lem_PQ_Hilb}; for the definition of $\Ha^{Q}_{P,R}$ and its properties, see \eqref{Ha_def}, \eqref{Ha_dag}, \eqref{Ha_prod}, and the subsequent comments. In particular, $h_{11}\colon\calS_{P}\to\Bo(\calS_{P,Q})$ is an $O(\delta+\eps)$-homomorphism, and $h_{12}$, $h_{22}$, and $h_{21}$ are $O(\delta+\eps)$-close to identity maps. The composition $h_{11}v\colon\Ma{n}\to \Bo(\calS_{P,Q})$ is an $O(\delta+\eps)$-homomorphism of $C^*$ algebras. By Lemma~\ref{lem_approx}, it can be approximated by a homomorphism $\mu_{11}\colon\Ma{n}\to \Bo(\calS_{P,Q})$ with $O(\delta+\eps)$ accuracy. Since $\Ma{n}$ has no nontrivial ideals, $\mu_{11}$ is injective. It is actually an isomorphism because both $\Ma{n}$ and $\Bo(\calS_{P,Q})$ are $n^2$-dimensional. Therefore, $h_{11}$, which is $O(\delta+\eps)$-close to $\mu_{11}v^{-1}$, is an $O(\delta+\eps)$-isomorphism. We conclude that the maps $h_{jk}^{-1}\colon\Bo(\calS_{P_k,Q},\calS_{P_j,Q})\to\calS_{P_j,P_k}$ exist for all $j$ and $k$ and satisfy the conditions
\begin{align}
\label{merging0h}
h_{kj}^{-1}(X^\dag)&=h^{-1}_{jk}(X)^\dag,\\[2pt]
\label{merging1h}
\|h^{-1}_{jl}(XY)-h^{-1}_{jk}(X)\cdot h^{-1}_{kl}(Y)\|
&\le O(\delta+\eps)\ts\|X\|\ts\|Y\|,\\[2pt]
\label{merging2h}
\|h^{-1}_{jj}(\Pi_j)-P_j\|&\le O(\delta+\eps),\\[2pt]
\label{merging3h}
(1-O(\delta+\eps))\ts\|X\|\le\|h^{-1}_{jk}(X)\| &\le(1+O(\delta+\eps))\ts\|X\|.
\end{align}

Being an isomorphism of finite-dimensional $C^*$ algebras, $\mu_{11}$ has the form $\mu_{11}(A)=U_1AU_1^\dag$ for some unitary map $U_1\colon\CC^n\to\calS_{P,Q}$. We also consider the obvious map $U_2\colon c\mapsto c\wt{Q}$ from $\CC$ to $\calS_{Q,Q}$. There is a total of four maps $\mu_{jk}\colon\Ma{n_j,n_k} \to\Bo(\calS_{P_k,Q},\calS_{P_j,Q})$ defined by the equation $\mu_{jk}(A)=U_jAU_k^\dag$, where $n_1=n$ and $n_2=1$. Together, they constitute an isomorphism of $C^*$ algebras $\Ma{n+1}\to\Bo(\calS_{P,Q}\oplus\calS_{Q,Q})$. Let us now define some maps $\gamma_{jk}\colon\Ma{n_j,n_k}\to\calS_{P_j,P_k}$:
\begin{equation}
\gamma_{11}(A_{11})=v(A_{11}),\quad\:
\gamma_{12}(A_{12})= U_1(A_{12}),\quad\:
\gamma_{21}(A_{21})=(U_1(A_{21}^\dag))^\dag,\quad\:
\gamma_{22}(A_{22})=A_{22}\wt{Q},
\end{equation}
where $A_{11}$ is an $n\times n$ matrix, $A_{12}\in\Ma{n,1}=\CC^n$ is a column vector, $A_{21}\in\Ma{1,n}=(\CC^n)^*$ is a row vector, and $A_{22}$ is a complex number. It follows from the definitions of $\gamma_{jk}$ and $\mu_{jk}$ that $\|h_{jk}\gamma_{jk}-\mu_{jk}\|\le O(\delta+\eps)$. Thus, equations \eqref{merging0}--\eqref{merging3} from the hypothesis of Lemma~\ref{lem_merging} (with $\delta$ replaced by $\delta'=O(\delta+\eps)$) follow from \eqref{merging0h}--\eqref{merging3h}. Applying the lemma, we obtain the required $O(\delta+\eps)$-isomorphism $v_+\colon\Ma{n+1}\to\calA$, namely, $v_+\colon\left(\begin{smallmatrix}A_{11} & A_{12}\\ A_{21} & A_{22}\end{smallmatrix}\right) \mapsto\sum_{j,k}\gamma_{jk}(A_{jk})$.
\end{proof}

\begin{proof}[Proof of Theorem~\ref{th_main}.]
Let $c_0$ be the constant from Corollary~\ref{cor_improvement}. We will construct a $c_0\eps$-isomorphism $v$ from some $C^*$ algebra $\calB$ to the $\eps$-$C^*$ algebra $\calA$ in three stages. The first stage yields a $c_0\eps$-inclusion $v_\comm\colon\calB_\comm\to\calA$, where $\calB_\comm$ is a commutative $C^*$ algebra. The second stage involves the parallel construction of $c_0\eps'$-isomorphisms from some matrix algebras to approximate direct summands of $\calA$. At the third stage, those algebras and isomorphisms are merged.\medskip

\noindent\textbf{Stage 1:} Let $v_\comm\colon\calB_\comm\to\calA$ be a $c_0\eps$-inclusion of a commutative $C^*$ algebra of maximum dimensionality into $\calA$. Then $\calB$ has a projections basis $\{\Pi_1,\dots,\Pi_m\}$. It is evident that each $P_j=v_\comm(\Pi_j)$ is an $O(\eps)$-projection. We now show that these projections are one-dimensional.

Let us suppose the opposite. Without loss of generality, we may assume that the last   projection fails to be one-dimensional, i.e.\ $\dim\calS_{P_m}>1$. Recall that, $\calS_{P_m}$ is an $O(\eps)$-$C^*$ algebra with unit $\wt{P}_m=\Co_{P_m}(P_m)$. By Lemma~\ref{lem_nontriv_projection}, there exists an $O(\eps)$-projection $P'\in\calS_{P_m}$ such that both $P'$ and $P''=\wt{P}_m-P'$ are nonvanishing, i.e.\ sufficiently far from $0$. Note that $P'+P''=\wt{P}_m$ is $O(\eps)$-close to $P_m$. Let $\calB_\comm^{(1)}$ and $\calB_\comm^{(2)}$ be the commutative $C^*$ algebras with projection bases $\{\Pi_1,\dots,\Pi_{m-1}\}$ and $\{\Pi',\Pi''\}$, and let $P_{[1,m-1]}=\sum_{j=1}^{m-1}P_j$. We now $O(\eps)$-include $\calB_\comm^{(1)}$ and $\calB_\comm^{(2)}$ into $\calS_{P_{[1,m-1]}}$ and $\calS_{P_m}$, respectively:
\begin{gather*}
v_\comm^{(1)}(\Pi_j)=\Co_{P_{[1,m-1]}}(P_j)\in\calS_{P_{[1,m-1]}}\quad
\text{for }\, j=1,\dots,m-1,\\[2pt]
v_\comm^{(2)}(\Pi')=P'\in\calS_{P_m},\qquad
v_\comm^{(2)}(\Pi'')=P''\in\calS_{P_m}.
\end{gather*}
Merging $v_\comm^{(1)}$ and $v_\comm^{(2)}$ using Corollary~\ref{cor_merge_sum}, we obtain an $O(\eps)$-inclusion $v_\comm^+\colon\calB_\comm^+\to\calA$, where $\calB_\comm^+=\calB_\comm^{(1)}\oplus\calB_\comm^{(2)}$. Due to Corollary~\ref{cor_improvement}, there also exists a $c_0\eps$-inclusion of $\calB_\comm^+$ into $\calA$. But this contradicts the maximum dimensionality assumption.\medskip

\noindent\textbf{Setup for stages 2 and 3:} Because $v_\comm$ is a $c_0\eps$-inclusion, the sum of $P_j=v_\comm(\Pi_j)$ over an arbitrary subset of $\{1,\dots,m\}$ is a $c_0\eps$-projection. Let $\eps'=O(\eps)$ denote a suitable upper bound such that $\eps'\ge c_0\eps$, and moreover, $\calS_P$ for any $c_0\eps$-projection $P$ is an $\eps'$-$C^*$ algebra. The indices $j$ of the one-dimensional projections $P_j$ are subdivided into equivalence classes such that $\dim\calS_{P_j,P_k}$ is $1$ if $j$ and $k$ are in the same class and $0$ otherwise. Associated with each equivalence class $C$ are the $c_0\eps$-projection $P_C=\sum_{j\in C}P_j$ and the $\eps'$-$C^*$ algebra $\calS_{P_C}$.\medskip

\noindent\textbf{Stage 2:} Each equivalence class $C$ is considered separately. Without loss of generality, we may assume that $C=\{1,\dots,s\}$. Let
\[
\Pi_{[1,r]}=\sum_{j=1}^{r}\Pi_j,\qquad P_{[1,r]}=\sum_{j=1}^{r}P_j,\qquad
\calA_r=\calS_{P_{[1,r]}}\qquad (r=1,\dots,s).
\]
Note that $P_{[1,r]}$ is a $c_0\eps$-projection, $\calA_r$ is an $\eps'$-$C^*$ algebra, and $\calA_s=\calS_{P_C}$. We construct $c_0\eps'$-isomorphisms $v_r\colon\Ma{r}\to\calA_r$ for $r=1,\dots,s$ inductively. To obtain $v_r$ from $v_{r-1}$, we first fit $v_{r-1}$, $P_{[1,r-1]}$, and $P_r$ into $\calA_r$ by applying the compression map $\Co_{P_{[1,r]}}$:
\[
v_{r-1}'\colon X\mapsto\Co_{P_{[1,r]}}(v(X)),\qquad
P_{[1,r-1]}'=\Co_{P_{[1,r]}}(P_{[1,r-1]}),\qquad
P_r'=\Co_{P_{[1,r]}}(P_r).
\]
Then we extend the $O(\eps)$-isomorphism $v_{r-1}'\colon \Ma{r-1}\to\calA_{r-1}'$ to an $O(\eps)$-isomorphism $v_{r-1}^+\colon \Ma{r}\to\calA_{r}$ using Lemma~\ref{lem_extension}. Finally, we use Corollary~\ref{cor_improvement} to replace $v_{r-1}^+$ with a $c_0\eps'$-isomorphism $v_r$.\medskip

\noindent\textbf{Stage 3.} At this point, we have constructed $c_0\eps'$-isomorphisms $v_C\colon\Ma{|C|}\to\calS_{P_C}$ for all equivalence classes $C$. Note that by Lemma~\ref{lem_add_dim}, $\calS_{P_C,P_D}=0$ if the classes $C$ and $D$ are distinct. This allows us to successively merge the isomorphisms $v_C$ for different $C$. Each step includes the application of Corollary~\ref{cor_merge_sum} followed by the use of Corollary~\ref{cor_improvement} to reduce the errors.
\end{proof}

\section{Tensor extensions}\label{sec_tens_ext}

In this section, we show that under suitable conditions, not only is the previously constructed map $v\colon\calB\to\calA$ a $\delta$-isomorphism, but also all maps $1_{\Ma{n}}\otimes v\colon \Ma{n}\otimes\calB\to\Ma{n}\otimes\calA$ are  $\delta$-isomorphisms for the same $\delta=O(\eps)$. Since $\calB$ is a $C^*$ algebra, $\Ma{n}\otimes\calB$ is also a $C^*$ algebra. (The multiplication, the unit, and the involution on $\Ma{n}\otimes\calB$ are defined in a straightforward way. To define the norm, one can first represent $\calB$ as a closed $*$-subalgebra of $\Bo(\calH)$ for some Hilbert space $\calH$. Then $\Ma{n}\otimes\calB\subseteq\Bo(\CC^n\otimes\calH)$ is equipped with the obvious norm. Being a $C^*$ norm, it is unique.) The multiplication, the unit, and the involution on the $\eps$-$C^*$ algebra $\calA$ are also automatically extended to $\Ma{n}\otimes\calA$, but the norm is not. Therefore, we have to assume the existence of separate norms $\|\cdot\|_n$ subject to some consistency conditions. This structure is called an (abstract) operator space~\cite[Chapter 13]{Paulsen}. We proceed with specific definitions.\smallskip

The unit $n\times n$ matrix $I_{\Ma{n}}$ will be denoted by $I_n$ for brevity, and the zero $n\times k$ matrix (with elements of any type) is denoted by $0_{n,k}$. Let $\calL$ be a complex vector space; then any $X\in\Ma{n,k}\otimes\calL$ may be regarded as an $n\times k$ matrix with elements $[X]_{pq}\in\calL$ for $p\in\{1,\dots,n\}$ and $q\in\{1,\dots,k\}$. A multiplication, i.e.\ a bilinear map $\calL\times\calL\to\calL$, extends to such matrices in the standard way, $[XY]_{pq}=\sum_l[X]_{pl}[Y]_{lq}$. An arbitrary element $I\in\calL$ extends to $I_n\otimes I$. A conjugate linear involution $X\mapsto X^\dag$ on $\calL$ extends to conjugate linear maps $\Ma{n,k}\otimes\calL\to\Ma{k,n}\otimes\calL$ using the rule $[X^\dag]_{pq}=([X]_{qp})^\dag$. 

\begin{Definition}\label{def:opspace}
A complex vector space $\calL$ is called an \emph{operator space} if each space $\Ma{n}\otimes\calL$ (for $n=1,2,\ldots$) is equipped with a norm $\|\cdot\|_n$ satisfying the following axioms:
\begin{alignat}{2}
\label{ax_R1}
\|AXB\|_n &\le\|A\|\ts\|X\|_k\ts\|B\|\qquad &
&(A\in\Ma{n,k},\quad B\in\Ma{k,n},\quad X\in\Ma{k}\otimes\calL),\\[3pt]
\label{ax_R2}
\left\|\begin{pmatrix}X&0\\ 0&Y\end{pmatrix}\right\|_{k+n} &=\max\bigl\{\|X\|_k,\|Y\|_n\bigr\}\qquad &
& (X\in\Ma{k}\otimes\calL,\quad Y\in\Ma{n}\otimes\calL).
\end{alignat}
The norm on $\calL$ itself is defined by identifying $\calL$ with $\Ma{1}\otimes\calL$. An operator space is called \emph{self-adjoint} if it is equipped with a conjugate linear involution $\dagger$ that preserves all norms $\|\cdot\|_n$.
\end{Definition}
In particular, all linear subspaces of $C^*$ algebras satisfy axioms \eqref{ax_R1}, \eqref{ax_R2}. Conversely, an arbitrary operator space $\calL$ can be identified with a subspace of $\Bo(\calH)$ for some Hilbert space $\calH$ such that the norms agree \cite{Rua88}, \cite[Theorem~13.4]{Paulsen}. In the self-adjoint case, one can arrange that the involution on $\calL$ is represented by the Hermitian conjugation on $\Bo(\calH)$.\smallskip

For $n'\le n$, there is a standard inclusion of $\Ma{n'}\otimes\calL$ into $\Ma{n}\otimes\calL$:
\begin{equation}
X\,\mapsto\,\begin{pmatrix}X & 0_{n',n-n'}\\ 0_{n-n',n'} & 0_{n-n'}\end{pmatrix}
=AXA^\dag,\qquad \text{where}\quad
A=\begin{pmatrix}I_{n'}\\ 0_{n-n',n'}\end{pmatrix}.
\end{equation}
This inclusion is isometric due to axiom \eqref{ax_R1} applied to both the present map and its left inverse, which has the form $Y\mapsto A^{\dag}YA$. The norm $\|\cdot\|_{n',n''}$ on $\Ma{n',n''}\otimes\calL$ is defined by including this space in $\Ma{n}\otimes\calL$ for an arbitrary $n\ge\max\{n',n''\}$. Such norms satisfy the properties similar to \eqref{ax_R1} and \eqref{ax_R2}.

As another corollary of the axioms, if $C\in\Ma{n}$ and $X\in\calL$, then $\|C\otimes X\|_n=\|C\|\ts\|X\|$. To see this, let us use the singular value decomposition $C=ADB^{\dag}$, where $D\in\Ma{k}$ is diagonal and $A^{\dag}A=B^{\dag}B=I_k$. Then $\|C\otimes X\|_n=\|D\otimes X\|_k$ due to \eqref{ax_R1}, and $\|D\otimes X\|_k=\|D\|\ts\|X\|$ due to \eqref{ax_R2}. Thus, the map $X\mapsto I_n\otimes X$ is an isometric inclusion of $\calL$ into $\Ma{n}\otimes\calL$. It is also evident that if $\calL$ is complete as a normed space, then all spaces $\Ma{n}\otimes\calL$ are complete.

\begin{Definition}
An \emph{extended $\eps$-$C^*$ algebra} is a complete self-adjoint operator space $\calA$ with a multiplication and a unit that make each space $\Ma{n}\otimes\calA$ into an $\eps$-$C^*$ algebra. An \emph{extended $\delta$-homomorphism} is a linear map $v\colon\calA'\to\calA''$ such that for each $n$, the map $1_{\Ma{n}}\otimes v\colon \Ma{n}\otimes\calA'\to\Ma{n}\otimes\calA''$ is a $\delta$-homomorphism.
\end{Definition}

One may define an \emph{extended $\delta$-inclusion} as a map satisfying the hypothesis of the following variant of Proposition~\ref{prop_delta_hominc}. An \emph{extended $\delta$-isomorphism} is such a map that is also bijective.

\begin{Proposition}\label{prop_inc_ext}
Let $v\colon\calA'\to\calA''$ be an extended $\delta$-homomorphism of $\eps$-$C^*$ algebras. If $\|v(X)\|\ge\eta\|X\|$ for some $\eta>2\delta$ and all $X\in\calA'$, then $\|(1_{\Ma{n}}\otimes v)(X)\|_n\ge a\|X\|_n$ for $a=1-O(\delta+\eps)$ independent of $n$ and all $X\in\Ma{n}\otimes\calA'$.
\end{Proposition}
\begin{proof}
Let us consider the numbers
\[
a_n=\inf\bigl\{\|(1_{\Ma{n}}\otimes v)(X)\|_n/\|X\|_n:\,
X\in\Ma{n}\otimes\calA',\,\, X\not=0\bigr\}\qquad
(a_1\ge a_2\ge a_3\ge\cdots\ge 0)
\]
and show that $a_{2n}\ge a_{n}/2$. For any $X=\left(\begin{smallmatrix}X_{11}&X_{12}\\ X_{21}&X_{22}\end{smallmatrix}\right)$ with $X_{pq}\in\Ma{n}\otimes\calA'$, we have
\[
\|X\|_{2n}\le 2\|X\|_{n,\max},\qquad
\text{where}\quad \|X\|_{n,\max}=\max_{p,q}\|X_{pq}\|_n.
\]
This follows from the operator space axioms and the formula $X=\left(\begin{smallmatrix}X_{11}&0\\ 0&X_{22}\end{smallmatrix}\right) +\left(\begin{smallmatrix}0&I_n\\ I_n&0\end{smallmatrix}\right) \left(\begin{smallmatrix}X_{21}&0\\ 0&X_{12}\end{smallmatrix}\right)$. On the other hand,
\[
\|(1_{\Ma{2n}}\otimes v)(X)\|_{2n}\ge \|(1_{\Ma{2n}}\otimes v)(X)\|_{n,\max}
\ge a_n\|X\|_{n,\max}.
\]
Therefore, $\|(1_{\Ma{2n}}\otimes v)(X)\|_{2n}\ge (a_n/2)\ts\|X\|_{2n}$ for all $X$, which means that $a_{2n}\ge a_{n}/2$.

By Proposition~\ref{prop_delta_hominc}, if $a_n>2\delta$, then $a_n\ge 1-\delta'$ for some $\delta'=O(\delta+\eps)$ that does not depend on $n$. We have $a_1\ge\eta>2\delta$, and hence, $a_1\ge 1-\delta'$. It follows by induction that $a_n\ge 1-\delta'$ for all $n$, provided $(1-\delta')/2>2\delta$. The last inequality holds under the standard assumption that $\delta$ and $\eps$ are sufficiently small.
\end{proof}

\begin{Lemma}\label{lem_approx_ext}
For any extended $\delta$-homomorphism $v$ from a finite-dimensional $C^*$ algebra $\calB$ to an $\eps$-$C^*$ algebra $\calA$, there is an extended $O(\eps)$-homomorphism $\tilde{v}\colon \calB\to\calA$ such that the corresponding maps $v_n=1_{\Ma{n}}\otimes v$ and $\tilde{v}_n=1_{\Ma{n}}\otimes\tilde{v}$ satisfy the inequality  $\bigl\|\tilde{v}_n(X)-v_n(X)\|\le O(\delta)\|X\|$ for all $X\in\Ma{n}\otimes\calB$.
\end{Lemma}

\begin{proof} Adapting the proof of Lemma~\ref{lem_approx}, we consider the multiplicativity defect of $v_n$:
\begin{equation}
g_n\colon (\Ma{n}\otimes\calB)\times(\Ma{n}\otimes\calB)
\to \Ma{n}\otimes\calA,\qquad g_n(X,Y)=v_n(XY)-v_n(X)v_n(Y),
\end{equation}
It satisfies the approximate 2-cocycle equation because $\Ma{n}\otimes\calA$ is an $\eps$-$C^*$ algebra. Using the diagonal $D=\sum_{j}A_j\otimes B_j \in\calB\otimes\calB$, we define the map $w_n'\colon \Ma{n}\otimes\calB \to\Ma{n}\otimes\calA$ as follows:
\begin{equation}
w_n'(X) = \sum_{j}v_n(I_n\otimes A_j)\, g_n(I_n\otimes B_j,X).
\end{equation}
To proceed, we need to extend the equation $\sum_{j}v(XA_j)g(B_j,Y)=\sum_{j}v(A_j)g(B_jX,Y)$ used in the calculation \eqref{g1_calc} from $X,Y\in\calB$ to $X,Y\in\Ma{n}\otimes\calB$. The more general equation,
\begin{equation}
\sum_{j}v_n\bigl(X(I_n\otimes A_j)\bigr)\, g_n(I_n\otimes B_j,Y)
=\sum_{j}v_n(I_n\otimes A_j)\, g_n\bigl((I_n\otimes B_j)X,Y\bigr),
\end{equation}
is derived by considering the matrix elements:
\begin{align*}
&\biggl[\sum_{j}v_n(X(I_n\otimes A_j))\,
g_n((I_n\otimes B_j),Y)\biggr]_{kl}
=\sum_{j}\sum_{p} v([X]_{kp}A_j)\, g(B_j,[Y]_{pl})\\
&\hspace{2cm}=\sum_{j}\sum_{p} v(A_j)\, g(B_j[X]_{kp},[Y]_{pl})
=\biggl[\sum_{j}v_n(I_n\otimes A_j)\,
g_n((I_n\otimes B_j)X,Y)\biggr]_{kl}.
\end{align*}
Here we have used the property of the diagonal $\sum_{j}ZA_j\otimes B_j=\sum_{j}A_j\otimes B_jZ$ for $Z=[X]_{kp}$. The rest of the proof is the same as for Lemma~\ref{lem_approx}.
\end{proof}

\begin{Theorem}\label{th_main_ext}
For any finite-dimensional extended $\eps$-$C^*$ algebra $\calA$, there exist a $C^*$ algebra $\calB$ and an extended $O(\eps)$-isomorphism $v\colon\calB\to\calA$. (The implicit constant in $O(\eps)$ does not depend on $\calA$ or its dimensionality.)
\end{Theorem}

\begin{proof} We adapt the proof of the main theorem. All new objects, such as nontrivial $\delta$-projections, are constructed for the original algebra $\calA$, but we need to show that their tensor extensions have similar properties. In the case of $\delta$-projections, that is straightforward: if $P\in\calA$ satisfies the conditions $P^\dag=P$ and $\|P^2-P\|\le\delta$, then so does the element $I_n\otimes P\in\Ma{n}\otimes\calA$ because the map $P\mapsto I_n\otimes P$ commutes with the involution and the multiplication as well as preserving the norm. Once such properties are established, one can apply the original lemmas to the $\eps$-$C^*$ algebra $1_{\Ma{n}}\otimes\calA$. We now discuss the required extensions of the arguments in Section \ref{sec_subspaces}.\medskip

\noindent\textbf{Compression maps and subspaces associated with projections.} The map $\Co_{P,Q}\colon\calA\to\calA$ is extended to $1_{\Ma{n}}\otimes\Co_{P,Q}=\Co_{I_n\otimes P,\,I_n\otimes Q}$. This new map has the same properties as the original one because $I_n\otimes P$ and $I_n\otimes Q$ are $\delta$-projections in the $\eps$-$C^*$ algebra $1_{\Ma{n}}\otimes\calA$. The subspace $\calS_{P,Q}$ is extended to $\Ma{n}\otimes\calS_{P,Q}=\calS_{I_n\otimes P,\,I_n\otimes Q}$.\medskip

\noindent\textbf{One-dimensional projections and the corresponding Hilbert spaces.} Lemma~\ref{lem_PQ_Hilb} says that if $P,Q\in\calA$ are $\delta$-projections and $Q$ is one-dimensional, then the space $\calS_{P,Q}$ is equipped with the Hermitian inner product given by \eqref{1dQ_ip}. The extended version of this space, $\CC^n\otimes\calS_{P,Q}=\Ma{n,1}\otimes\calS_{P,Q}$, satisfies essentially the same equation:
\begin{equation}\label{1dQ_ip_ext}
Y^\dag\cdot X=\braket{Y}{X}\,\wt{Q}\qquad (X,Y\in\Ma{n,1}\otimes\calS_{P,Q}),
\end{equation}
where $\braket{Y}{X}=\sum_{l}\braket{[Y]_{l1}}{[X]_{l1}}$. It is still true that
\begin{equation}\label{1dQ_ip_norm_ext}
\bigl|\braket{X}{X}-\|X\|_{n,1}\bigr| \le O(\delta+\eps)\ts\|X\|_{n,1},\qquad
(X\in\Ma{n,1}\otimes\calS_{P,Q}),
\end{equation}
and hence, $\|X\|_\Euc=\sqrt{\braket{X}{X}}$ is equal to $\|X\|$ up to a $1\pm O(\eps+\delta)$ factor. The map $\Ha^{Q}_{P,R}$ defined by equation \eqref{Ha_def} extends to $1_{\Ma{n}}\otimes\Ha^{Q}_{P,R}\colon \Ma{n}\otimes\calS_{P,R}\to\Ma{n}\otimes\Bo(\calS_{R,Q},\calS_{P,Q})$, where the target space may be identified with $\Bo(\CC^n\otimes\calS_{R,Q},\CC^n\otimes\calS_{P,Q})$. Equations \eqref{Ha_dag}, \eqref{Ha_prod}, and the special properties of $\Ha^{Q}_{P,P}$, $\Ha^{Q}_{P,Q}$, and $\Ha^{Q}_{Q,P}$ are generalized in a straightforward way.\medskip

Corollary~\ref{cor_improvement} (error reduction) should be adapted to extended inclusions using Lemma~\ref{lem_approx_ext} and Proposition~\ref{prop_inc_ext}. The arguments in Section~\ref{sec_proof_main} require only trivial modifications, namely, one should use the norms $\|\cdot\|_n$ in certain places.
\end{proof}

\section{Details on unital completely positive  maps}\label{sec_channels}

Here we give some background on unital completely positive maps and derive the properties of idempotent UCP maps that were mentioned in Section~\ref{sec_ch_intro} without proofs. Our main focus is on the finite-dimensional case. Some results (in this and the next sections) can be extended to infinite dimensions with a bit of work, but we discuss such generalizations separately so as not to complicate the basic arguments.

\subsection{The Choi and Stinespring representations}

If $\calK$ and $\calH$ are finite-dimensional Hilbert spaces, then any UCP map $\Phi\colon\Bo(\calK)\to\Bo(\calH)$ admits a \emph{Choi representation}
\begin{equation}\label{Choi}
\Phi(X)=V^\dag(X\otimes 1_\calF)V,\qquad
V^{\dag}V=1_{\calH},\qquad
\text{where}\quad
V\colon\calH\to\calK\otimes\calF.
\end{equation}
It can be constructed in a standard way using the auxiliary space $\calF=\calK^*\otimes\calH$. (Choi's original construction~\cite{Cho75} was formulated in terms of matrices, but it does have the form $V^\dag(X\otimes 1_\calF)V$.) Equation \eqref{Choi} can also be represented graphically, using the vertical dimension for the tensor product:
\begin{equation}\label{Choi1}
\Phi(X)= \begin{tikzpicture}
\matrix[Choi]
{
  \coordinate (l); &
  \node[bdot] (Vd) {}; &
  \node[circ] (X) {$X$}; &
  \node[bdot] (V) {}; &
  \coordinate (r); \\
};
\draw (l) -- (Vd) -- (X) -- (V) -- (r);
\draw (l) -- (Vd) -- (X) -- (V) -- (r);
\draw (Vd) to[bridge] (V);
\end{tikzpicture}\,,\qquad
\begin{tikzpicture}
\matrix[Choi]
{
  \coordinate (l); &
  \node[bdot] (Vd) {}; &[0.3cm]
  \node[bdot] (V) {}; &
  \coordinate (r); \\
};
\draw (l) -- (Vd) -- (V) -- (r);
\draw (l) -- (Vd) -- (V) -- (r);
\draw (Vd) to[bridge=0.35cm] (V);
\end{tikzpicture}
=\begin{tikzpicture}
\draw (-0.4,0) -- (0.4,0);
\end{tikzpicture}\,,\qquad
\text{where}\quad
\begin{tikzpicture}
\matrix[Choi,matrix anchor=V.center]
{
  \coordinate[label=left:$\scriptstyle\calF$] (m1); & & \\[-0.5mm]
  \coordinate[label=left:$\scriptstyle\calK$] (m); &
  \node[bdot] (V) {}; &
  \coordinate[label=right:$\scriptstyle\calH$] (r); \\
};
\draw (m) -- (V) -- (r);
\draw (m1) to[ramp from] (V);
\end{tikzpicture} =V.
\end{equation}

In infinite dimensions, even special UCP maps $\Bo(\calK)\to\CC$, i.e.\ states on $\Bo(\calK)$, are not necessarily representable in that way. For example, a state may vanish on all compact operators, which precludes its representation by a unit vector $V\in\calK\hotimes\calF$. That said, any state can be represented by a unit vector in some Hilbert space via the Gelfand-Naimark-Segal (GNS) construction.

Let $\calA$ be a $C^*$ algebra, $I$ its unit, $\calH$ a Hilbert space, and $\Phi\colon\calA\to\Bo(\calH)$ a UCP map. A \emph{Stinespring representation} of $\Phi$ has the general form
\begin{equation}\label{Stinespring_def}
\Phi(X)=V^\dag\, u(X)\, V,
\end{equation}
where $u\colon\calA\to\Bo(\calG)$ is a $*$-homomorphism and $V\colon\calH\to\calG$ is an isometry. The \emph{canonical Stinespring representation} \cite{Sti55}, \cite[Theorem~4.1]{Paulsen} is a generalization of the GNS construction. It uses the space $\wt{\calG}=\calA\otimes\calH$ with the possibly degenerate inner product\footnote{On the right-hand side, vectors are interpreted as linear maps from $\CC$. Thus, $\xi\colon\CC\to\calH$ and $\eta^\dag\colon\calH\to\CC$.}
\begin{equation}
\braket{B\otimes\eta}{A\otimes\xi}=\eta^\dag\,\Phi(B^\dag A)\,\xi\qquad
(A,B\in\calA,\quad \xi,\eta\in\calH).
\end{equation}
The space $\calG$ is the completion of the quotient of $\wt{G}$ by the null subspace, and $u$ and $V$ are defined as follows:
\begin{equation}
u(C)\ts(A\otimes\xi)=CA\otimes\xi,\qquad V\xi=I\otimes\xi.
\end{equation}

\begin{Proposition}\label{prop_KLHG}
Let $J\colon\calL\to\calK$ be an isometry, $w\colon\Bo(\calL)\to\Bo(\calH)$ a $*$-homomorphism, and let $(\calG,u,V)$ be the canonical Stinespring representation of the UCP map $\Phi\colon X\mapsto w(J^\dag XJ)$. Then the following diagrams commute:
\begin{equation}
\begin{tikzcd}
\Bo(\calK) \arrow[d,swap,"X\mapsto J^\dag XJ\,"]
\arrow[r,"u"] \arrow[dr,"\Phi"] &
\Bo(\calG) \arrow[d,"\,X\mapsto V^\dag XV"] \\
\Bo(\calL) \arrow[r,"w"] & \Bo(\calH)
\end{tikzcd}
\qquad\quad
\begin{tikzcd}
\Bo(\calK) \arrow[d,leftarrow,swap,"Y\mapsto JYJ^\dag\,"] \arrow[r,"u"] &
\Bo(\calG) \arrow[d,leftarrow,"\,Y\mapsto VYV^\dag"] \\
\Bo(\calL) \arrow[r,"w"] & \Bo(\calH)
\end{tikzcd}
\end{equation}
\end{Proposition}
The commutativity of the first diagram follows immediately from the hypothesis and the general form \eqref{Stinespring_def} of a Stinespring representation. The second diagram, specific to the canonical Stinespring representation, is not necessary to achieve our goals but offers some convenience.

\begin{proof}
The commutativity of the second diagram means that these maps are equal:
\begin{equation}
s',s''\colon \Bo(\calL)\to\Bo(\calG),\qquad\:
s'(Y)=u(JYJ^\dag),\qquad s''(Y)=V\ts w(Y)\ts V^\dag.
\end{equation}
To prove that $s'=s''$, it is sufficient to show that
\[
\bbraket{B\otimes\eta}{u(JYJ^\dag)\ts(A\otimes\xi)}
=\bbraket{B\otimes\eta}{V\ts w(Y)\ts V^\dag (A\otimes\xi)}
\]
for all $A,B\in\Bo(\calK)$,\,\, $\xi,\eta\in\calH$, and $Y\in\Bo(\calL)$. The expression with $u(JYJ^\dag)$ is easy to evaluate:
\begin{equation}\label{uJYJd}
\begin{aligned}
\bbraket{B\otimes\eta}{u(JYJ^\dag)\ts(A\otimes\xi)}
&=\eta^\dag\, \Phi(B^\dag JYJ^\dag A)\, \xi
=\eta^\dag\, w(J^\dag B^\dag JYJ^\dag A J)\, \xi\\[2pt]
&=\eta^\dag\, w(J^\dag B^\dag J)\, w(Y)\, w(J^\dag AJ)\, \xi
=\eta^\dag\, \Phi(B^\dag)\, w(Y)\, \Phi(A)\, \xi.
\end{aligned}
\end{equation}
To evaluate $\braket{B\otimes\eta}{VZV^\dag(A\otimes\xi)}$ for an arbitrary $Z\in\Bo(\calH)$, we represent $Z$ as the limit of some net $(Z_\alpha)$ in the weak operator topology such that each $Z_\alpha$ has finite rank. For a rank~$1$ operator, $Z=\mu\nu^\dag$, the calculation is straightforward:
\[
\bbraket{B\otimes\eta}{V\mu\nu^\dag V^\dag(A\otimes\xi)}
=\braket{B\otimes\eta}{I\otimes\mu}\,\braket{I\otimes\nu}{A\otimes\xi}
=\eta^\dag\,\Phi(B^\dag)\, \mu\nu^\dag\, \Phi(A)\,\xi.
\]
Extending to finite rank operators by linearity and to arbitrary $Z\in\Bo(\calH)$ using the limit, we get
\begin{equation}\label{VwYVd}
\bbraket{B\otimes\eta}{VZV^\dag(A\otimes\xi)}
=\eta^\dag\,\Phi(B^\dag)\, Z\, \Phi(A)\,\xi.
\end{equation}
It remains to substitute $w(Y)$ for $Z$ and see that the right-hand sides of \eqref{uJYJd} and \eqref{VwYVd} are equal.
\end{proof}

\subsection{Some other properties of UCP maps}

The following inequality holds for any UCP map $\Phi\colon\calA\to\calB$ between $C^*$ algebras:
\begin{equation}\label{PhiXdX}
\Phi(X^\dag)\,\Phi(X)\le \Phi(X^\dag X).
\end{equation}
(It means that the right-hand side minus the left-hand side is a positive element of $\calB$.) Indeed, consider the  matrix
\[
Z=\begin{pmatrix} X^\dag X & X^\dag\\ X & I_\calA \end{pmatrix}
=\begin{pmatrix} X^\dag\\ I_\calA \end{pmatrix}
\begin{pmatrix} X & I_\calA \end{pmatrix}
\ge 0,
\]
which may be regarded as an element of $\Ma{2}\otimes\calA$. (Recall that $\Ma{n}=\Bo(\CC^n)$ denotes the algebra of matrices with complex entries.) Since $\Phi$ is completely positive, we have $(1_{\Ma{2}}\otimes\Phi)(Z)\ge 0$, and hence,
\[
\begin{aligned}
0 &\le\begin{pmatrix} -I_\calB & \Phi(X^\dag) \end{pmatrix}\,
(1_{\Ma{2}}\otimes\Phi)(Z)\,
\begin{pmatrix} -I_\calB \\ \Phi(X) \end{pmatrix}
=\begin{pmatrix} -I_\calB & \Phi(X^\dag) \end{pmatrix}
\begin{pmatrix}
\Phi(X^\dag X) & \Phi(X^\dag)\\ \Phi(X) & \Phi(I_\calA)
\end{pmatrix}
\begin{pmatrix} -I_\calB \\ \Phi(X) \end{pmatrix}
\\
&=\Phi(X^\dag X)-\Phi(X^\dag)\,\Phi(X).
\end{aligned}
\]
Another general inequality is $\|\Phi(X)\|\le\|X\|$. It follows from \eqref{PhiXdX} and the positivity of $\Phi$, which implies that $\Phi(X^\dag X)\le \Phi(\|X\|^2I_\calA) =\|X\|^2I_\calB$. Replacing $\Phi$ with $1_{\Ma{n}}\otimes\Phi$ for an arbitrary $n$, we conclude that
\begin{equation}
\|\Phi\|_\cb \le 1.
\end{equation}

Let $\calK$ and $\calH$ be finite-dimensional Hilbert spaces, and $\Phi\colon\Bo(\calK)\to\Bo(\calH)$ a UCP map. The \emph{carrier} of $\Phi$ is the support of $\Phi^*(\rho_0)$, where $\rho_0$ is an arbitrary full-rank density matrix on $\calH$.

\begin{Lemma}\label{lem_carrier}
The carrier of a UCP map $\Phi\colon\Bo(\calK)\to\Bo(\calH)$ with the Choi representation $(\calF,V)$ is the smallest subspace $\calM\subseteq\calK$ such that $\calM\otimes\calF$ contains the image of $V$. In other words, an arbitrary subspace $\calM'\subseteq\calK$ contains the carrier if and only if $\calM'\otimes\calF\supseteq\Img V$.
\end{Lemma}
\begin{proof}
Let $\calM'\subseteq\calK$, and let us denote the orthogonal projection onto $\calM'$ by $\Pi_{\calM'}$. The subspace $\calM'$ contains the carrier, i.e.\ the support of $\rho_1=\Phi^*(\rho_0)$, if and only if $\Tr((1-\Pi_{\calM'})\rho_1)=0$. Since $\Tr((1-\Pi_{\calM'})\,\Phi^*(\rho_0))=\Tr(\Phi(1-\Pi_{\calM'})\,\rho_0)$, the density matrix $\rho_0$ has full rank, and $\Phi(1-\Pi_{\calM'})\ge0$, the vanishing of the trace is equivalent to the vanishing of the whole operator $\Phi(1-\Pi_{\calM'})$. The latter is equal to $Z^\dag Z$ for $Z=((1-\Pi_{\calM'})\otimes 1_\calF)V$. Thus, $\calM'$ contains the carrier if and only if $Z=0$, which is the same as $\calM'\otimes\calF\supseteq\Img V$.
\end{proof}

\begin{Corollary}\label{cor_carrier}
Let $\calM$ be the carrier of a UCP map $\Phi\colon\Bo(\calK)\to\Bo(\calH)$ with the Choi representation $(\calF,V)$. Then an operator $X\in\Bo(\calK)$ annihilates all vectors in $\calM$ if and only if $(X\otimes 1_\calF)V=0$.
\end{Corollary}

\begin{proof}
The two conditions in question can be written as $\Ker X\supseteq\calM$ and $(\Ker X)\otimes\calF\supseteq\Img V$, respectively. By Lemma~\ref{lem_carrier}, these conditions are equivalent.
\end{proof}
For another corollary, $\Phi$ is determined by its action on operators on its carrier via the equation
\begin{equation}\label{PhiX_M}
\Phi(X)=\Phi(\Pi_{\calM}X\Pi_{\calM})\qquad (X\in\Bo(\calK)).
\end{equation}
Indeed, $\Phi(X)=V^\dag(X\otimes 1_\calF)V =V^\dag(\Pi_{\calM}\otimes 1_\calF)(X\otimes 1_\calF)(\Pi_{\calM}\otimes 1_\calF)V =\Phi(\Pi_{\calM}X\Pi_{\calM})$.

\subsection{The structure of finite-dimensional idempotent UCP maps}\label{sec_idemp_structure}

Let $\calH$ be a finite-dimensional Hilbert space and $\Phi\colon\Bo(\calH)\to\Bo(\calH)$ an idempotent UCP map. The equation $\Phi^2(X)=\Phi(X)$ can be expressed diagrammatically by analogy with \eqref{Choi1}:
\begin{equation}
\begin{tikzpicture}
\matrix[Choi]
{
  \coordinate (l); &
  \node[bdot] (Vd1) {}; &
  \node[bdot] (Vd) {}; &[-2mm]
  \node[circ] (X) {$X$}; &[-2mm]
  \node[bdot] (V) {}; &
  \node[bdot] (V1) {}; &
  \coordinate (r); \\
};
\draw (l) -- (Vd1) -- (Vd) -- (X) -- (V) -- (V1) -- (r);
\draw (Vd1) to[bridge=0.6cm] (V1);
\draw (Vd) to[bridge] (V);
\end{tikzpicture}
=\begin{tikzpicture}
\matrix[Choi]
{
  \coordinate (l); &
  \node[bdot] (Vd) {}; &[-2mm]
  \node[circ] (X) {$X$}; &[-2mm]
  \node[bdot] (V) {}; &
  \coordinate (r); \\
};
\draw (l) -- (Vd) -- (X) -- (V) -- (r);
\draw (Vd) to[bridge] (V);
\end{tikzpicture}\:.
\end{equation}
The next equation follows from the previous one:
\begin{equation}\label{eq_idemp}
\begin{tikzpicture}
\matrix[Choi,matrix anchor=V.center]
{
  \coordinate (l2); &[1mm] &[-2mm] &[-2mm] & & &[1mm]; \\[-1mm]
  \coordinate (l1); & & & & & &; \\[2mm]
  \coordinate (l); &
  \node[bdot] (Vd) {}; &
  \node[circ] (X) {$X$}; &
  \node[bdot] (V) {}; &
  \node[bdot] (V1) {}; &
  \node[bdot] (V2) {}; &
  \coordinate (r); \\
};
\draw (l) -- (Vd) -- (X) -- (V) -- (V1) -- (V2) -- (r);
\draw (Vd) to[bridge] (V);
\draw (l1) to[ramp from] (V1);
\draw (l2) to[ramp from] (V2);
\end{tikzpicture}
\,=\,\begin{tikzpicture}
\matrix[Choi,matrix anchor=V.center]
{
  \coordinate (l2); &[4mm] &[-1mm] &[-2mm] &[-2mm] & &[1mm]; \\[-1mm]
  \coordinate (l1); & & & & & &; \\[2mm]
  \coordinate (l); &
  \node[bdot] (V1) {}; &
  \node[bdot] (Vd) {}; &
  \node[circ] (X) {$X$}; &
  \node[bdot] (V) {}; &
  \node[bdot] (V2) {}; &
  \coordinate (r); \\
};
\draw (l) -- (V1) -- (Vd) -- (X) -- (V) -- (V2) -- (r);
\draw (Vd) to[bridge] (V);
\draw (l1) to[ramp from] (V1);
\draw (l2) to[ramp from] (V2);
\end{tikzpicture}\:.
\end{equation}
Indeed, let $Z$ be the difference between the left-hand side and the right-hand side. Then
\[
\begin{aligned}
Z^\dag Z &
=\quad\begin{tikzpicture}
\matrix[Choi,matrix anchor=V.center]
{
  \coordinate (l); &
  \node[bdot] (Vd2) {}; &
  \node[bdot] (Vd1) {}; &
  \node[bdot] (Vda) {}; &[-1.5mm]
  \node[circ] (Xd) {$X^\dag$}; &[-1.5mm]
  \node[bdot] (Va) {}; &
  \node[bdot] (Vdb) {}; &[-2mm]
  \node[circ] (X) {$X$}; &[-2mm]
  \node[bdot] (Vb) {}; &
  \node[bdot] (V1) {}; &
  \node[bdot] (V2) {}; &
  \coordinate (r); \\
};
\draw (l) -- (Vd2) -- (Vd1) -- (Vda) -- (Xd) -- (Va)
  -- (Vdb) -- (X) -- (Vb) -- (V1) -- (V2) -- (r);
\draw (Vdb) to[bridge=0.45cm] (Vb);
\draw (Vda) to[bridge=0.5cm] (Va);
\draw (Vd1) to[bridge=0.7cm] (V1);
\draw (Vd2) to[bridge=0.9cm] (V2);
\end{tikzpicture}\,
-\,\,\begin{tikzpicture}
\matrix[Choi,matrix anchor=V.center]
{
  \coordinate (l); &
  \node[bdot] (Vd2) {}; &
  \node[bdot] (Vd1) {}; &
  \node[bdot] (Vda) {}; &[-1.5mm]
  \node[circ] (Xd) {$X^\dag$}; &[-1.5mm]
  \node[bdot] (Va) {}; &
  \node[bdot] (V1) {}; &
  \node[bdot] (Vdb) {}; &[-2mm]
  \node[circ] (X) {$X$}; &[-2mm]
  \node[bdot] (Vb) {}; &
  \node[bdot] (V2) {}; &
  \coordinate (r); \\[2pt]
};
\draw (l) -- (Vd2) -- (Vd1) -- (Vda) -- (Xd) -- (Va)
  -- (V1) -- (Vdb) -- (X) -- (Vb) -- (V2) -- (r);
\draw (Vdb) to[bridge=0.45cm] (Vb);
\draw (Vda) to[bridge=0.5cm] (Va);
\draw (Vd1) to[bridge=0.7cm] (V1);
\draw (Vd2) to[bridge=0.9cm] (V2);
\end{tikzpicture}
\\[2pt]
&\quad-\,\begin{tikzpicture}
\matrix[Choi,matrix anchor=V.center]
{
  \coordinate (l); &
  \node[bdot] (Vd2) {}; &
  \node[bdot] (Vda) {}; &[-1.5mm]
  \node[circ] (Xd) {$X^\dag$}; &[-1.5mm]
  \node[bdot] (Va) {}; &
  \node[bdot] (Vd1) {}; &
  \node[bdot] (Vdb) {}; &[-2mm]
  \node[circ] (X) {$X$}; &[-2mm]
  \node[bdot] (Vb) {}; &
  \node[bdot] (V1) {}; &
  \node[bdot] (V2) {}; &
  \coordinate (r); \\
};
\draw (l) -- (Vd2) -- (Vda) -- (Xd) -- (Va) -- (Vd1)
  -- (Vdb) -- (X) -- (Vb) -- (V1) -- (V2) -- (r);
\draw (Vdb) to[bridge=0.45cm] (Vb);
\draw (Vda) to[bridge=0.5cm] (Va);
\draw (Vd1) to[bridge=0.7cm] (V1);
\draw (Vd2) to[bridge=0.9cm] (V2);
\end{tikzpicture}\,
+\,\:\begin{tikzpicture}
\matrix[Choi,matrix anchor=V.center]
{
  \coordinate (l); &
  \node[bdot] (Vd2) {}; &
  \node[bdot] (Vda) {}; &[-1.5mm]
  \node[circ] (Xd) {$X^\dag$}; &[-1.5mm]
  \node[bdot] (Va) {}; &
  \node[bdot] (Vd1) {}; &[2mm]
  \node[bdot] (V1) {}; &
  \node[bdot] (Vdb) {}; &[-2mm]
  \node[circ] (X) {$X$}; &[-2mm]
  \node[bdot] (Vb) {}; &
  \node[bdot] (V2) {}; &
  \coordinate (r); \\
};
\draw (l) -- (Vd2) -- (Vda) -- (Xd) -- (Va) -- (Vd1)
  -- (V1) -- (Vdb) -- (X) -- (Vb) -- (V2) -- (r);
\draw (Vdb) to[bridge=0.45cm] (Vb);
\draw (Vda) to[bridge=0.5cm] (Va);
\draw (Vd1) to[bridge=0.45cm] (V1);
\draw (Vd2) to[bridge=0.9cm] (V2);
\end{tikzpicture}
\end{aligned}
\]
is zero because all four terms (without the signs) are equal to $\Phi\bigl(\Phi(X^\dag)\,\Phi(X)\bigr)$.

\begin{Lemma}\label{lem_idemp}
Let $\calH$ be a finite-dimensional Hilbert space and $\Phi\colon\Bo(\calH)\to\Bo(\calH)$ an idempotent UCP map with the image $\calA\subseteq\Bo(\calH)$ and carrier $\calM\subseteq\calH$. Then operators in $\calA$ do not mix $\calM$ and $\calM^\perp$. Furthermore, the Choi-Effros product $\Phi(XY)$ restricted to $\calM$ coincides with the usual operator product:
\begin{equation}
\Phi(XY)|_\calM=(XY)|_\calM=X|_{\calM}Y|_{\calM}\qquad
(X,Y\in\calA),
\end{equation}
where $X|_\calM=J_\calM^\dag XJ_\calM$ and $J_\calM\colon\calM\to\calH$ is the inclusion of Hilbert spaces.
\end{Lemma}

\begin{proof}
The statement that operators in $\calA$ do not mix $\calM$ and $\calM^\perp$ means that for any $X\in\calA$, we have $(1-\Pi_\calM)X\Pi_\calM=0$. By Corollary~\ref{cor_carrier}, an equivalent condition is $((1-\Pi_\calM)X\otimes 1_\calF)V=0$. To prove this equation, let us denote its left-hand side by $C$ and calculate $C^\dag C$:
\[
\begin{aligned}
C^\dag C &=V^\dag(X^\dag(1-\Pi_\calM)X\otimes 1_\calF)V
=\,\begin{tikzpicture}
\matrix[Choi,matrix anchor=V.center]
{
  \coordinate (l); &
  \node[bdot] (Vd1) {}; &
  \node[circ] (Xd) {$X^\dag$}; &
  \node[rect] (P) {$1-\Pi_\calM$}; &
  \node[circ] (X) {$X$}; &
  \node[bdot] (V1) {}; &
  \coordinate (r); \\
};
\draw (l) -- (Vd1) -- (Xd) -- (P) -- (X) -- (V1) -- (r);
\draw (Vd1) to[bridge=0.5cm] (V1);
\end{tikzpicture}
\\[3pt]
&=\,\begin{tikzpicture}
\matrix[Choi,matrix anchor=V.center]
{
  \coordinate (l); &
  \node[bdot] (Vd2) {}; &
  \node[bdot] (Vd1) {}; &
  \node[bdot] (Vda) {}; &[-1.5mm]
  \node[circ] (Xd) {$X^\dag$}; &[-1.5mm]
  \node[bdot] (Va) {}; &
  \node[rect] (P) {$1-\Pi_\calM$}; &
  \node[bdot] (Vdb) {}; &[-2mm]
  \node[circ] (X) {$X$}; &[-2mm]
  \node[bdot] (Vb) {}; &
  \node[bdot] (V1) {}; &
  \node[bdot] (V2) {}; &
  \coordinate (r); \\
};
\draw (l) -- (Vd2) -- (Vd1) -- (Vda) -- (Xd) -- (Va)
  -- (P) -- (Vdb) -- (X) -- (Vb) -- (V1) -- (V2) -- (r);
\draw (Vdb) to[bridge=0.45cm] (Vb);
\draw (Vda) to[bridge=0.5cm] (Va);
\draw (Vd1) to[bridge=0.7cm] (V1);
\draw (Vd2) to[bridge=0.9cm] (V2);
\end{tikzpicture}\,
\\[3pt]
&=\,\begin{tikzpicture}
\matrix[Choi,matrix anchor=V.center]
{
  \coordinate (l); &
  \node[bdot] (Vd2) {}; &
  \node[bdot] (Vda) {}; &[-1.5mm]
  \node[circ] (Xd) {$X^\dag$}; &[-1.5mm]
  \node[bdot] (Va) {}; &
  \node[bdot] (Vd1) {}; &
  \node[rect] (P) {$1-\Pi_\calM$}; &
  \node[bdot] (V1) {}; &
  \node[bdot] (Vdb) {}; &[-2mm]
  \node[circ] (X) {$X$}; &[-2mm]
  \node[bdot] (Vb) {}; &
  \node[bdot] (V2) {}; &
  \coordinate (r); \\
};
\draw (l) -- (Vd2) -- (Vda) -- (Xd) -- (Va) -- (Vd1)
  -- (P) -- (V1) -- (Vdb) -- (X) -- (Vb) -- (V2) -- (r);
\draw (Vdb) to[bridge=0.45cm] (Vb);
\draw (Vda) to[bridge=0.5cm] (Va);
\draw (Vd1) to[bridge=0.7cm] (V1);
\draw (Vd2) to[bridge=0.9cm] (V2);
\end{tikzpicture}\,
=0.
\end{aligned}
\]
(In the second line, we have used the equations $\Phi^2=\Phi$ and $X=\Phi(X)$, and in the last line, \eqref{eq_idemp} and the equation $((1-\Pi_\calM)\otimes 1_\calF)V=0$.)

We now prove the relation between the Choi-Effros product and the usual operator product of $X,Y\in\calA$. By the previous argument, $X$ splits as $X|_\calM\oplus X|_{\calM^\perp}$ and $Y$ splits as $Y|_\calM\oplus Y|_{\calM^\perp}$. Therefore, $(XY)|_\calM=X|_\calM Y|_\calM$. To show that $\Phi(XY)|_\calM=(XY)|_\calM$, we apply Corollary~\ref{cor_carrier} again. Thus, the equation in question becomes $(\Phi(XY)\otimes 1_\calF)V=(XY\otimes 1_\calF)V$. In this form, it is derived using \eqref{eq_idemp} as well as the equations $X=\Phi(X)$,\, $Y=\Phi(Y)$, and $\Phi^2=\Phi$:
{\setlength\abovedisplayskip{10pt}\setlength\belowdisplayskip{-5pt}\[
\begin{aligned}
(\Phi(XY)\otimes 1_\calF)V
&=\,\begin{tikzpicture}
\matrix[Choi,matrix anchor=V.center]
{
  \coordinate (l2); & & &[-2mm] &[-2mm] & &[-2mm] &[-2mm] & & &; \\[4mm]
  \coordinate (l); &
  \node[bdot] (Vd1) {}; &
  \node[bdot] (Vda) {}; &
  \node[circ] (X) {$X$}; &
  \node[bdot] (Va) {}; &
  \node[bdot] (Vdb) {}; &
  \node[circ] (Y) {$Y$}; &
  \node[bdot] (Vb) {}; &
  \node[bdot] (V1) {}; &
  \node[bdot] (V2) {}; &
  \coordinate (r); \\
};
\draw (l) -- (Vd1) -- (Vda) -- (X) -- (Va) -- (Vdb) -- (Y) -- (Vb)
  -- (V1) -- (V2) -- (r);
\draw (Vda) to[bridge] (Va);
\draw (Vdb) to[bridge] (Vb);
\draw (Vd1) to[bridge=0.6cm] (V1);
\draw (l2) to[ramp from] (V2);
\end{tikzpicture}\,
=\,\begin{tikzpicture}
\matrix[Choi,matrix anchor=V.center]
{
  \coordinate (l2); & & &[-2mm] &[-2mm] & & &[-2mm] &[-2mm] & &; \\[4mm]
  \coordinate (l); &
  \node[bdot] (Vd1) {}; &
  \node[bdot] (Vda) {}; &
  \node[circ] (X) {$X$}; &
  \node[bdot] (Va) {}; &
  \node[bdot] (V1) {}; &
  \node[bdot] (Vdb) {}; &
  \node[circ] (Y) {$Y$}; &
  \node[bdot] (Vb) {}; &
  \node[bdot] (V2) {}; &
  \coordinate (r); \\
};
\draw (l) -- (Vd1) -- (Vda) -- (X) -- (Va) -- (V1)
  -- (Vdb) -- (Y) -- (Vb) -- (V2) -- (r);
\draw (Vda) to[bridge] (Va);
\draw (Vdb) to[bridge] (Vb);
\draw (Vd1) to[bridge=0.6cm] (V1);
\draw (l2) to[ramp from] (V2);
\end{tikzpicture}
\\[4pt]
&=(XY\otimes 1_\calF)V.
\end{aligned}
\]}%
\end{proof}

\begin{proof}[Proof of Theorem~\ref{th_idemp_structure}.]
Let $\calM$ and $\calA$ be the carrier and the image of $\Phi$, respectively, and let $J_\calM\colon\calM\to\calH$ and $\Delta\colon\calA\to\Bo(\calH)$ be the corresponding inclusion maps. (At this point, $\calA$ is just a vector space, and $\Delta$ is a linear map.) Let us also consider the maps
\[
\begin{array}{c@{\,}l}
\Co_\calM\colon X\mapsto J_\calM^\dag XJ_\calM &
\colon\Bo(\calH)\to\Bo(\calM),\vspace{2pt}\\
\Lambda\colon X\mapsto\Phi\bigl(J_\calM X J_\calM^\dag\bigr) &
\colon\Bo(\calM)\to\Bo(\calH).
\end{array}
\]
It is evident that $\Co_\calM$ is a UCP map and that $\Lambda$ is completely positive; we will soon see that it is also unital. By construction, $\Lambda$ factors through $\calA=\Img\Phi$, that is, $\Lambda=\Delta\Gamma$ for some map $\Gamma\colon\Bo(\calM)\to\calA$. We can use equation \eqref{PhiX_M} to infer that for any $X\in\Bo(\calH)$,
\[
\Phi(X)
=\Phi(\Pi_\calM X\Pi_\calM)
=\Phi\bigl(J_\calM J_\calM^\dag X J_\calM J_\calM^\dag\bigr)
=\Lambda\Co_\calM(X).
\]
Thus, $\Phi=\Lambda\Co_\calM=\Delta\Gamma\Co_\calM$, which also implies that $\Lambda(1_\calM)=\Lambda\Co_\calM(1_\calH)=\Phi(1_\calH)=1_\calH$.

Since $\Phi$ is idempotent, $\Delta(A)=\Phi\Delta(A) =\Delta\Gamma\Co_\calM\Delta(A)$ for all $A\in\calA$; hence, $\Gamma\Co_\calM\Delta=1_\calA$ due to the injectivity of $\Delta$. This is a convenient summary:
\[
\begin{tikzcd}
& \Bo(\calH) \arrow[d,"\Co_\calM"] \\
\calA \arrow[ur,"\Delta"] & \Bo(\calM) \arrow[l,"\Gamma"]
\end{tikzcd}\qquad\quad
\Delta\Gamma=\Lambda,\quad\: \Lambda\Co_\calM=\Phi.\quad\:
\Gamma\Co_\calM\Delta=1_\calA.
\]

Let $w=\Co_\calM\Delta$ so that $\Gamma w=1_\calA$. The last equation implies that $w$ is injective, $\Gamma$ is surjective, and the UCP map $\Psi=\Co_\calM\Lambda=w\Gamma$ is idempotent. Therefore,
\[
\Img w=\Img\Psi=\Ker(1-\Psi)\subseteq\Bo(\calM).
\]
This subspace of $\Bo(\calM)$ is closed, contains the unit operator, and is invariant under the involution $X\mapsto X^\dag$. Furthermore, the second part of Lemma~\ref{lem_idemp} says that $w$ carries the Choi-Effros product on $\calA$ to the to the usual operator product on $\Bo(\calM)$, so $\Img w$ is closed under the operator product. Thus, $\Img w$ is a $C^*$ subalgebra of $\Bo(\calM)$. We may identify it with $\calA$, which makes the latter into a $C^*$ algebra and $w$ into an inclusion of $C^*$ algebras. At the same time, $\Gamma$ becomes a UCP map because $\Psi=w\Gamma$ is UCP map.

The last assertion in the theorem (that operators in the image of $\Delta$ are block-diagonal with respect to the decomposition $\calH=\calM\oplus\calM^\perp$) is just the first part of Lemma~\ref{lem_idemp}.
\end{proof}

Now, we determine the specific form of the maps involved. Since $w\colon\calA\to\Bo(\calM)$ is an inclusion of $C^*$ algebras and $\dim\calM<\infty$, the structure of $w$ is described by Proposition~\ref{prop_hom_structure}:
\begin{equation}\label{Aw}
w(A_1,\dots,A_m)=\sum_{j=1}^{m}W_j^\dag(A_j\otimes 1_{\calE_j})W_j\qquad
\text{for }\, (A_1,\dots,A_m)\in\calA=\bigoplus_{j=1}^{m}\Bo(\calL_j).
\end{equation}
Let us also use the fact that for each $A\in\calA$, the operator $\Delta(A)$ is the direct sum of $\Delta(A)|_\calM=w(A)$ and $\Delta(A)|_{\calM^\perp}$. The latter can be obtained from $A$ by some UCP map $\Sigma\colon\calA\to\Bo(\calM^\perp)$. Thus,
\begin{equation}\label{Delta_structure}
\Delta(A)=\sum_{j=1}^{m}J_\calM W_j^\dag
(A_j\otimes 1_{\calE_j})W_j J_\calM^\dag
+J_{\calM^\perp}\ts \Sigma(A)\ts J_{\calM^\perp}^\dag\qquad
(A\in\calA).
\end{equation}

\begin{Proposition}\label{prop_Gamma}
Let $w\colon\calA\to\Bo(\calM)$ have the form \eqref{Aw}, where $(W_1,\dots,W_m)$ are the components of the unitary map $W\colon\calM\to\bigoplus_{j=1}^{m}\calL_j\otimes\calE_j$. Then any UCP map $\Gamma\colon\Bo(\calM)\to\calA$ satisfying the equation $\Gamma w=1_\calA$ admits the component-wise representation
\begin{equation}\label{Gamma}
\Gamma_j\colon\Bo(\calM)\to\Bo(\calL_j),\qquad
\Gamma_j(X)=\Tr_{\calE_j}\bigl(W_j XW_j^\dag
(1_{\calL_j}\otimes\gamma_j)\bigr)
\end{equation} 
for some density matrices $\gamma_j$ on $\calE_j$.
\end{Proposition}
(Equation \eqref{Gamma} is the general form of a conditional expectation in finite dimensions.)

\begin{proof}
Each component $\Gamma_j\colon\Bo(\calM)\to\Bo(\calL_j)$ can be written in the Choi form,
\[
\Gamma_j(X)=L_j^\dag(X\otimes1_{\calF_j})L_j,\qquad
\text{where}\quad L_j\colon\calL_j\to\calM\otimes\calF_j,\qquad
L_j^\dag L_j=1_{\calL_j}.
\]
Thus, the equation $\Gamma w=1_\calA$ becomes $L_j^\dag\bigl((W_k^\dag(A_k\otimes 1_{\calE_k})W_k)\otimes 1_{\calF_j}\bigr)L_j=A_j$ for all $(A_1,\dots,A_m)$. Equivalently, let
{\setlength\abovedisplayskip{0pt}\[
M_{kj}=(W_k\otimes 1_{\calF_j})L_j
=\,\begin{tikzpicture}
\matrix[Choi,matrix anchor=W.center]
{
  \coordinate[label=left:$\scriptstyle\calF_j$] (l2); &[1mm] &[4mm] &[0.5mm]
  \\[1mm]
  \coordinate[label=left:$\scriptstyle\calE_k$] (l1); & & & \\[1mm]
  \coordinate[label=left:$\scriptstyle\calL_k$] (l); &
  \node[rect] (W) {$W_k$}; &
  \node[rect] (L) {$L_j$}; &
  \coordinate[label=right:$\scriptstyle\calL_j$] (r); \\
};
\draw (l) -- (W) -- node[auto] {$\scriptstyle\calM$} (L) -- (r);
\path
  (l1) edge[ramp from] (W)
  (l2) edge[ramp from] (L);
\end{tikzpicture}\:;
\]}%
then $\Tr\bigr(M_{kj}^\dag(A_k\otimes 1_{\calE_k\otimes\calF_j})M_{kj}B_j\bigr)
=\Tr(A_jB_j)$ for all $A,B\in\calA$. Eliminating $A$ and $B$, we obtain a condition on the product of $M_{kj}^\dag$ and $M_{kj}$ over $\calE_k\otimes\calF_j$. Specifically,
\[
R_{kj}^\dag R_{kj}=\delta_{kj}\ts 1_{\calL_j}\Tr_{\calL_j},\qquad
\text{where}\quad
R_{kj}\colon X\mapsto\Tr_{\calL_k}(M_{kj}X)\,
\colon\,\calL_j\otimes\calL_k^*\to\calE_k\otimes\calF_j.
\]
It follows that $R_{jk}=\delta_{kj}\ts\xi_j\Tr_{\calL_j}$ for some $\xi_j\in\calE_j\otimes\calF_j$, that is, $(W_k\otimes 1_{\calF_j})L_j=\delta_{kj}(1_{\calL_j}\otimes\xi_j)$. Multiplying the last equation by $W_k^\dag\otimes 1_{\calF_j}$ on the left and summing over $k$, we get
\[
L_j=(W_j^\dag\otimes 1_{\calF_j})(1_{\calL_j}\otimes\xi_{j}).
\]
Plugging this formula in the equation $\Gamma_j(X)=L_j^\dag(X\otimes1_{\calF_j})L_j$ gives the required representation with $\gamma_j=\Tr_{\calF_j}(\xi_j\xi_j^\dag)$.
\end{proof}

In conclusion, let us consider consider the dual maps. It follows from Proposition~\ref{prop_Gamma} that $\Enc=(\Gamma\Co_\calM)^*$ has the form \eqref{Enc}. Similarly, equation \eqref{Delta_structure} implies that $\Dec=\Delta^*$ has the form \eqref{Dec}. Finally, $\Dec\,\Enc=(\Gamma\Co_\calM\Delta)^*=1_{\calA^*}$ and $\Enc\,\Dec=(\Delta\Gamma\Co_\calM)^*=\Phi^*$.

\section{Almost-idempotent UCP maps}\label{sec_almost_idemp}

\subsection{The associated \texorpdfstring{$\eps$-$C^*$}{epsilon-C*} algebra}\label{sec_assoc_ecsa}

Let $\calH$ be an arbitrary nonzero Hilbert space, and let us consider a UCP map $\Phi\colon\Bo(\calH)\to\Bo(\calH)$ such that
\begin{equation}
\|\Phi^2-\Phi\|_\cb\le\eta,
\end{equation}
where $\eta$ is a sufficiently small nonnegative number. First, we recall some constructions from Section~\ref{sec_ch_intro}. This definition is based on Proposition~\ref{prop_P} and equation \eqref{abs_sgn}:
\begin{equation}\label{tilde_Phi}
\wt{\Phi}=\theta(2\Phi-1)
=\frac{1}{2}\Bigl(1+\sgn(2\Phi-1)\Bigr)
=\frac{1}{2}\Bigl(1+(2\Phi-1)\bigl(1-4(\Phi-\Phi^2)\bigr)^{-1/2}\Bigr).
\end{equation}
The right-hand side involves a Taylor expansion in $4(\Phi-\Phi^2)$, which converges if $\eta<1/4$. The map $\wt{\Phi}$ has the properties
\begin{equation}
\wt{\Phi}^2=\wt{\Phi},\qquad
\|\wt{\Phi}-\Phi\|_\cb\le O(\eta),\qquad
\wt{\Phi}(1)=1,\qquad
\wt{\Phi}(X^\dag)=\wt{\Phi}(X)^\dag\text{ for all }X\in\Bo(\calH),
\end{equation}
which allow for the definition of $\calA$:
\begin{equation}
\calA=\Img\wt{\Phi}=\Ker(1-\wt{\Phi})\subseteq\Bo(\calH).
\end{equation}
As such, $\calA$ is a closed subspace of $\Bo(\calH)$; it contains the unit operator $I=1_\calH$ and is invariant under the involution $X\mapsto X^\dag$. We define the multiplication (approximate Choi-Effros product) on $\calA$ by the equation
\begin{equation}\label{Choi-Effros}
X\star Y=\wt{\Phi}(XY)\qquad (X,Y\in\calA).
\end{equation}

\begin{Theorem}\label{th_almost_idemp}
The space $\calA$ with the norm, involution, and unit inherited from $\Bo(\calH)$ and the multiplication $(X,Y)\mapsto X\star Y$ is an extended $O(\eta)$-$C^*$ algebra.
\end{Theorem}

The nontrivial part of the proof is concerned with the following equations,
\begin{align}
\label{Phi_assoc1}
\Phi\Bigl(\Phi\bigl(\Phi(X)\,\Phi(Y)\bigr)\,\Phi(Z)\Bigr)
&=\Phi\bigl(\Phi(X)\,\Phi(Y)\,\Phi(Z)\bigr)+O(\eta)\ts\|X\|\ts\|Y\|\ts\|Z\|,
\\[2pt]
\label{Phi_assoc2}
\Phi\Bigl(\Phi(X)\,\Phi\bigl(\Phi(Y)\,\Phi(Z)\bigr)\Bigr)
&=\Phi\bigl(\Phi(X)\,\Phi(Y)\,\Phi(Z)\bigr)+O(\eta)\ts\|X\|\ts\|Y\|\ts\|Z\|,
\end{align}
but we will begin with the easy part.

\begin{proof}[Proof of Theorem~\ref{th_almost_idemp} using equations \eqref{Phi_assoc1} and \eqref{Phi_assoc2}.]
We assume that these equations are true for all $\eta$-idempotent UCP maps, and in particular, for $1_{\Ma{n}}\otimes\Phi$. So the subsequent arguments are applicable not only to the algebra $\calA$ but also to $\Ma{n}\otimes\calA$.

The axioms of an $\eps$-$C^*$ algebra not involving the multiplication are clearly satisfied, and it is also evident that $X\star I=X=I\star X$ and $(X\star Y)^\dag=Y^\dag\star X^\dag$ for all $X,Y\in\calA$.

Let us check the approximate submultiplicativity of the norm. If $X,Y\in\calA$, then
\begin{equation}
\|X\star Y\| =\|\wt{\Phi}(XY)\| \le\|\Phi(XY)\|+O(\eta)\|XY\|
\le (1+O(\eta))\ts\|X\|\ts\|Y\|,
\end{equation}
where we have used the inequalities $\|\wt{\Phi}-\Phi\|_\cb\le O(\eta)$ and $\|\Phi\|_\cb\le 1$.

Using the same inequalities, one may safely replace $\Phi$ with $\wt{\Phi}$ in \eqref{Phi_assoc1} and \eqref{Phi_assoc2}. Subtracting the resulting equations, we get
\[
\Bigl\|\wt{\Phi}\Bigl(\wt{\Phi}\bigl(\wt{\Phi}(X)\,\wt{\Phi}(Y)\bigr)\,
\wt{\Phi}(Z)\Bigr)
-\wt{\Phi}\Bigl(\wt{\Phi}(X)\,\wt{\Phi}\bigl(\wt{\Phi}(Y)\,
\wt{\Phi}(Z)\bigr)\Bigr)\Bigr\|
\le O(\eta)\ts\|X\|\ts\|Y\|\ts\|Z\|.
\]
If $X,Y,Z\in\calA$, then $\wt{\Phi}(X)=X$,\, $\wt{\Phi}(Y)=Y$, and $\wt{\Phi}(Z)=Z$, so that the previous equation becomes the $\eps$-associativity axiom for $\eps=O(\eta)$:
\begin{equation}
\|(X\star Y)\star Z-X\star (Y\star Z)\| \le O(\eta)\ts\|X\|\ts\|Y\|\ts\|Z\|.
\end{equation}

Finally, the operator inequality $\Phi(X^\dag)\,\Phi(X)\le\Phi(X^\dag X)$ (see \eqref{PhiXdX}) implies that $\|\Phi(X^\dag X)\|\ge\|\Phi(X)\|^2$. Replacing $\Phi$ with $\wt{\Phi}$ and assuming that $X\in\calA$, we get
\begin{equation}
\|X^\dag\star X\|\ge (1-O(\eta))\ts\|X\|^2.
\end{equation}
\end{proof}

\begin{proof}[Proof of \eqref{Phi_assoc1} and \eqref{Phi_assoc2} in the finite-dimensional case.]
Let $(\calF,V)$ be the Choi representation of $\Phi$, and let us introduce the notation
\begin{equation}
\Pi=VV^\dag=\,
\begin{tikzpicture}
\matrix[Choi,matrix anchor=l.center]
{
  \coordinate[label=left:$\scriptstyle\calF$] (l1); & & &
  \coordinate[label=right:$\scriptstyle\calF$] (r1); \\
  \coordinate[label=left:$\scriptstyle\calH$] (l); &
  \node[bdot] (V) {}; &
  \node[bdot] (Vd) {}; &
  \coordinate[label=right:$\scriptstyle\calH$] (r); \\
};
\draw (l) -- (V) -- (Vd) -- (r);
\draw (l1) to[ramp from] (V);
\draw (Vd) to[ramp to] (r1);
\end{tikzpicture}\,,\qquad\qquad
\begin{tikzpicture}
\matrix[Choi,matrix anchor=l.center]
{
  \coordinate[label=left:$\scriptstyle\calF$] (l1); & &
  \coordinate[label=right:$\scriptstyle\calF$] (r1); \\
  \coordinate[label=left:$\scriptstyle\calH$] (l); &
  \node[base rect] (P) {\rule{0pt}{4mm}}; &
  \coordinate[label=right:$\scriptstyle\calH$] (r); \\
};
\draw (l) -- (P.base west);
\draw (P.base east) -- (r);
\draw (l1) -- ([yshift=4mm]P.base west);
\draw ([yshift=4mm]P.base east) -- (r1);
\end{tikzpicture}\:
=1-\Pi.
\end{equation}
First, we show that
\begin{equation}\label{1-Pi_X}
\begin{tikzpicture}
\matrix[Choi,matrix anchor=l.center]
{
  \coordinate (l2); & & &[-2mm] &[-2mm] & & & ; \\
  \coordinate (l1); & & & & & & &; \\[2mm]
  \coordinate (l); &
  \node[base rect] (P) {\rule{0pt}{6mm}}; &
  \node[bdot] (Vd) {}; &
  \node[circ] (X) {$X$}; &
  \node[bdot] (V) {}; &
  \node[bdot] (V1) {}; &
  \node[bdot] (V2) {}; &
  \coordinate (r); \\
};
\draw (l) -- (P.base west);
\draw (P.base east) -- (Vd) -- (X) -- (V) -- (V1) -- (V2) -- (r);
\draw (Vd) to[bridge] (V);
\draw (l1) -- ([yshift=6mm]P.base west);
\draw ([yshift=6mm]P.base east) to[ramp from] (V1);
\draw (l2) to[ramp from] (V2);
\end{tikzpicture}\,
=O\bigl(\sqrt{\eta}\bigr)\|X\|.
\end{equation}
Indeed,
\[
\begin{aligned}
\hspace{2cm}&\hspace{-2cm}
\begin{tikzpicture}
\matrix[Choi,matrix anchor=l.center]
{
  \coordinate (l); &
  \node[bdot] (Vd2) {}; &
  \node[bdot] (Vd1) {}; &
  \node[bdot] (Vda) {}; &[-1.5mm]
  \node[circ] (Xd) {$X^\dag$}; &[-1.5mm]
  \node[bdot] (Va) {}; &  
  \node[base rect] (P) {\rule{0pt}{7mm}}; &
  \node[bdot] (Vdb) {}; &[-2mm]
  \node[circ] (X) {$X$}; &[-2mm]
  \node[bdot] (Vb) {}; &
  \node[bdot] (V1) {}; &
  \node[bdot] (V2) {}; &
  \coordinate (r); \\
};
\draw (l) -- (Vd2) -- (Vd1) -- (Vda) -- (Xd) -- (Va) -- (P.base west);
\draw (P.base east) -- (Vdb) -- (X) -- (Vb) -- (V1) -- (V2) -- (r);
\draw (Vda) to[bridge=5mm] (Va);
\draw (Vdb) to[bridge=4.5mm] (Vb);
\draw (Vd1) to[ramp to] ([yshift=7mm]P.base west);
\draw ([yshift=7mm]P.base east) to[ramp from] (V1);
\draw (Vd2) to[bridge=10mm] (V2);
\end{tikzpicture}
\\[3pt]
&=\,\begin{tikzpicture}
\matrix[Choi,matrix anchor=l.center]
{
  \coordinate (l); &
  \node[bdot] (Vd2) {}; &
  \node[bdot] (Vd1) {}; &
  \node[bdot] (Vda) {}; &[-1.5mm]
  \node[circ] (Xd) {$X^\dag$}; &[-1.5mm]
  \node[bdot] (Va) {}; &
  \node[bdot] (Vdb) {}; &[-2mm]
  \node[circ] (X) {$X$}; &[-2mm]
  \node[bdot] (Vb) {}; &
  \node[bdot] (V1) {}; &
  \node[bdot] (V2) {}; &
  \coordinate (r); \\
};
\draw (l) -- (Vd2) -- (Vd1) -- (Vda) -- (Xd) -- (Va)
  -- (Vdb) -- (X) -- (Vb) -- (V1) -- (V2) -- (r);
\draw (Vdb) to[bridge=0.45cm] (Vb);
\draw (Vda) to[bridge=0.5cm] (Va);
\draw (Vd1) to[bridge=0.7cm] (V1);
\draw (Vd2) to[bridge=0.9cm] (V2);
\end{tikzpicture}\,
-\,\begin{tikzpicture}
\matrix[Choi,matrix anchor=l.center]
{
  \coordinate (l); &
  \node[bdot] (Vd2) {}; &
  \node[bdot] (Vd1a) {}; &
  \node[bdot] (Vda) {}; &[-1.5mm]
  \node[circ] (Xd) {$X^\dag$}; &[-1.5mm]
  \node[bdot] (Va) {}; &
  \node[bdot] (V1a) {}; &
  \node[bdot] (Vd1b) {}; &
  \node[bdot] (Vdb) {}; &[-2mm]
  \node[circ] (X) {$X$}; &[-2mm]
  \node[bdot] (Vb) {}; &
  \node[bdot] (V1b) {}; &
  \node[bdot] (V2) {}; &
  \coordinate (r); \\
};
\draw (l) -- (Vd2) -- (Vda) -- (Xd) -- (Va)
  -- (Vdb) -- (X) -- (Vb) -- (V2) -- (r);
\draw (Vdb) to[bridge=0.45cm] (Vb);
\draw (Vda) to[bridge=0.5cm] (Va);
\draw (Vd1a) to[bridge=0.7cm] (V1a);
\draw (Vd1b) to[bridge=0.7cm] (V1b);
\draw (Vd2) to[bridge=0.9cm] (V2);
\end{tikzpicture}
\\[3pt]
&= \Phi^2\bigl(\Phi(X^\dag)\,\Phi(X)\bigr)
-\Phi\bigl(\Phi^2(X^\dag)\,\Phi^2(X)\bigr)
=O(\eta)\|X\|^2.
\end{aligned}
\]

Now, consider a diagram that has two parts similar to the left-hand side of \eqref{1-Pi_X}, one containing $X$ and the other containing $Z$:
\begin{equation}\label{WXYZOeta}
W=\,
\begin{tikzpicture}
\matrix[Choi,matrix anchor=l.center]
{
  \coordinate (l); &
  \node[bdot] (Vd3) {}; &
  \node[bdot] (Vd2) {}; &
  \node[bdot] (Vd1) {}; &
  \node[bdot] (Vda) {}; &[-2mm]
  \node[circ] (X) {$X$}; &[-2mm]
  \node[bdot] (Va) {}; &  
  \node[base rect] (Pa) {\rule{0pt}{6mm}}; &
  \node[bdot] (Vdb) {}; &[-2mm]
  \node[circ] (Y) {$Y$}; &[-2mm]
  \node[bdot] (Vb) {}; &
  \node[bdot] (V1) {}; &
  \node[base rect] (Pc) {\rule{0pt}{9mm}}; &
  \node[bdot] (Vdc) {}; &[-2mm]
  \node[circ] (Z) {$Z$}; &[-2mm]
  \node[bdot] (Vc) {}; &
  \node[bdot] (V2) {}; &
  \node[bdot] (V3) {}; &
  \coordinate (r); \\
};
\draw (l) -- (Vd3) -- (Vd2)-- (Vd1) -- (Vda) -- (X) -- (Va) -- (Pa.base west);
\draw (Pa.base east) -- (Vdb) -- (Y) -- (Vb) -- (V1) -- (Pc.base west);
\draw (Pc.base east) -- (Vdc) -- (Z) -- (Vc) -- (V2) -- (V3) -- (r);
\draw (Vda) to[bridge] (Va);
\draw (Vdb) to[bridge] (Vb);
\draw (Vdc) to[bridge] (Vc);
\draw (Vd1) to[ramp to] ([yshift=6mm]Pa.base west);
\draw ([yshift=6mm]Pa.base east) to[ramp from] (V1);
\draw (Vd2) to[ramp to] ([yshift=9mm]Pc.base west);
\draw ([yshift=9mm]Pc.base east) to[ramp from] (V2);
\draw (Vd3) to[bridge=12mm] (V3);
\end{tikzpicture}\,
=O(\eta)\ts\|X\|\ts\|Y\|\ts\|Z\|.
\end{equation}
Expanding both rectangles as $1-VV^\dag$, we get
\begin{align*}
W &=\begin{aligned}[t]
&\phantom{\hbox{}+\hbox{}}\,
\begin{tikzpicture}[ampersand replacement=\&]
\matrix[Choi,matrix anchor=l.center]
{
  \coordinate (l); \&
  \node[bdot] (Vd3) {}; \&
  \node[bdot] (Vd2) {}; \&
  \node[bdot] (Vd1) {}; \&
  \node[bdot] (Vda) {}; \&[-2mm]
  \node[circ] (X) {$X$}; \&[-2mm]
  \node[bdot] (Va) {}; \&  
  \node[bdot] (Vdb) {}; \&[-2mm]
  \node[circ] (Y) {$Y$}; \&[-2mm]
  \node[bdot] (Vb) {}; \&
  \node[bdot] (V1) {}; \&
  \node[bdot] (Vdc) {}; \&[-2mm]
  \node[circ] (Z) {$Z$}; \&[-2mm]
  \node[bdot] (Vc) {}; \&
  \node[bdot] (V2) {}; \&
  \node[bdot] (V3) {}; \&
  \coordinate (r); \\
};
\draw (l) -- (Vd3) -- (Vd2)-- (Vd1) -- (Vda) -- (X) -- (Va)
  -- (Vdb) -- (Y) -- (Vb) -- (V1)
  -- (Vdc) -- (Z) -- (Vc) -- (V2) -- (V3) -- (r);
\draw (Vda) to[bridge] (Va);
\draw (Vdb) to[bridge] (Vb);
\draw (Vdc) to[bridge] (Vc);
\draw (Vd1) to[bridge=6mm] (V1);
\draw (Vd2) to[bridge=8mm] (V2);
\draw (Vd3) to[bridge=10mm] (V3);
\end{tikzpicture}
\\[3pt]
&-\,
\begin{tikzpicture}[ampersand replacement=\&]
\matrix[Choi,matrix anchor=l.center]
{
  \coordinate (l); \&
  \node[bdot] (Vd3) {}; \&
  \node[bdot] (Vd2) {}; \&
  \node[bdot] (Vd1a) {}; \&
  \node[bdot] (Vda) {}; \&[-2mm]
  \node[circ] (X) {$X$}; \&[-2mm]
  \node[bdot] (Va) {}; \&  
  \node[bdot] (Vdb) {}; \&[-2mm]
  \node[circ] (Y) {$Y$}; \&[-2mm]
  \node[bdot] (Vb) {}; \&
  \node[bdot] (V1a) {}; \&
  \node[bdot] (V2) {}; \&
  \node[bdot] (Vd1c) {}; \&
  \node[bdot] (Vdc) {}; \&[-2mm]
  \node[circ] (Z) {$Z$}; \&[-2mm]
  \node[bdot] (Vc) {}; \&
  \node[bdot] (V1c) {}; \&
  \node[bdot] (V3) {}; \&
  \coordinate (r); \\
};
\draw (l) -- (Vd3) -- (Vd2)-- (Vd1a) -- (Vda) -- (X) -- (Va)
  -- (Vdb) -- (Y) -- (Vb) -- (V1a) -- (V2)
  -- (Vd1c) -- (Vdc) -- (Z) -- (Vc) -- (V1c) -- (V3) -- (r);
\draw (Vda) to[bridge] (Va);
\draw (Vdb) to[bridge] (Vb);
\draw (Vdc) to[bridge] (Vc);
\draw (Vd1a) to[bridge=6mm] (V1a);
\draw (Vd1c) to[bridge=6mm] (V1c);
\draw (Vd2) to[bridge=8mm] (V2);
\draw (Vd3) to[bridge=10mm] (V3);
\end{tikzpicture}
\\[3pt]
&-\,
\begin{tikzpicture}[ampersand replacement=\&]
\matrix[Choi,matrix anchor=l.center]
{
  \coordinate (l); \&
  \node[bdot] (Vd3) {}; \&
  \node[bdot] (Vd2) {}; \&
  \node[bdot] (Vd1a) {}; \&
  \node[bdot] (Vda) {}; \&[-2mm]
  \node[circ] (X) {$X$}; \&[-2mm]
  \node[bdot] (Va) {}; \&  
  \node[bdot] (V1a) {}; \&  
  \node[bdot] (Vd1b) {}; \&  
  \node[bdot] (Vdb) {}; \&[-2mm]
  \node[circ] (Y) {$Y$}; \&[-2mm]
  \node[bdot] (Vb) {}; \&
  \node[bdot] (V1b) {}; \&
  \node[bdot] (Vdc) {}; \&[-2mm]
  \node[circ] (Z) {$Z$}; \&[-2mm]
  \node[bdot] (Vc) {}; \&
  \node[bdot] (V2) {}; \&
  \node[bdot] (V3) {}; \&
  \coordinate (r); \\
};
\draw (l) -- (Vd3) -- (Vd2)-- (Vd1a) -- (Vda) -- (X) -- (Va) -- (V1a)
  -- (Vd1b) -- (Vdb) -- (Y) -- (Vb) -- (V1b)
  -- (Vdc) -- (Z) -- (Vc) -- (V2) -- (V3) -- (r);
\draw (Vda) to[bridge] (Va);
\draw (Vdb) to[bridge] (Vb);
\draw (Vdc) to[bridge] (Vc);
\draw (Vd1a) to[bridge=6mm] (V1a);
\draw (Vd1b) to[bridge=6mm] (V1b);
\draw (Vd2) to[bridge=8mm] (V2);
\draw (Vd3) to[bridge=10mm] (V3);
\end{tikzpicture}
\\[3pt]
&+\,
\begin{tikzpicture}[ampersand replacement=\&]
\matrix[Choi,matrix anchor=l.center]
{
  \coordinate (l); \&
  \node[bdot] (Vd3) {}; \&
  \node[bdot] (Vd2) {}; \&
  \node[bdot] (Vd1a) {}; \&
  \node[bdot] (Vda) {}; \&[-2mm]
  \node[circ] (X) {$X$}; \&[-2mm]
  \node[bdot] (Va) {}; \&  
  \node[bdot] (V1a) {}; \&  
  \node[bdot] (Vd1b) {}; \&  
  \node[bdot] (Vdb) {}; \&[-2mm]
  \node[circ] (Y) {$Y$}; \&[-2mm]
  \node[bdot] (Vb) {}; \&
  \node[bdot] (V1b) {}; \&
  \node[bdot] (V2) {}; \&
  \node[bdot] (Vd1c) {}; \&
  \node[bdot] (Vdc) {}; \&[-2mm]
  \node[circ] (Z) {$Z$}; \&[-2mm]
  \node[bdot] (Vc) {}; \&
  \node[bdot] (V1c) {}; \&
  \node[bdot] (V3) {}; \&
  \coordinate (r); \\
};
\draw (l) -- (Vd3) -- (Vd2)-- (Vd1a) -- (Vda) -- (X) -- (Va) -- (V1a)
  -- (Vd1b) -- (Vdb) -- (Y) -- (Vb) -- (V1b) -- (V2)
  -- (Vd1c) -- (Vdc) -- (Z) -- (Vc) -- (V1c) -- (V3) -- (r);
\draw (Vda) to[bridge] (Va);
\draw (Vdb) to[bridge] (Vb);
\draw (Vdc) to[bridge] (Vc);
\draw (Vd1a) to[bridge=6mm] (V1a);
\draw (Vd1b) to[bridge=6mm] (V1b);
\draw (Vd1c) to[bridge=6mm] (V1c);
\draw (Vd2) to[bridge=8mm] (V2);
\draw (Vd3) to[bridge=10mm] (V3);
\end{tikzpicture}
\end{aligned}
\\[5pt]
&=\begin{aligned}[t]
&\phantom{\hbox{}+\hbox{}}
\Phi^2\Bigl(\Phi\bigl(\Phi(X)\,\Phi(Y)\bigr)\,\Phi(Z)\Bigr)
-\Phi\Bigl(\Phi^2\bigl(\Phi(X)\,\Phi(Y)\bigr)\,\Phi^2(Z)\Bigr)\\
&-\Phi^2\bigl(\Phi^2(X)\,\Phi^2(Y)\,\Phi(Z)\bigr)
+\Phi\Bigl(\Phi\bigl(\Phi^2(X)\,\Phi^2(Y)\bigr)\,\Phi^2(Z)\Bigr)
\end{aligned}
\\[4pt]
&=\Phi\Bigl(\Phi\bigl(\Phi(X)\,\Phi(Y)\bigr)\,\Phi(Z)\Bigr)
-\Phi\bigl(\Phi(X)\,\Phi(Y)\,\Phi(Z)\bigr)
+O(\eta)\ts\|X\|\ts\|Y\|\ts\|Z\|.
\end{align*}
The combination of this result and equation \eqref{WXYZOeta} gives \eqref{Phi_assoc1}. Equation \eqref{Phi_assoc2} is obtained similarly.
\end{proof}

To prove the result in full generality (in finite and infinite dimensions), we need to replace the Choi representation with the Stinespring representation and find suitable generalizations of the Hilbert spaces $\calH_n=\calH\otimes\calF^{\otimes n}$, isometries $V_n=V\otimes 1_\calF^{\otimes(n-1)}\colon\calH_{n-1}\to\calH_{n}$, and \hbox{$*$-homomorphisms} $u_n\colon X\mapsto X\otimes1_\calF\colon \Bo(\calH_{n-1})\to\Bo(\calH_{n})$. Such a ``Stinespring stack''
\begin{equation}
\begin{tikzcd}[column sep=24pt]
\calH_0 \arrow[r,hookrightarrow,"V_1"] &
\calH_1 \arrow[r,hookrightarrow,"V_2"] &
\calH_2 \arrow[r,hookrightarrow,"V_3"] & \cdots
\end{tikzcd}\,,\qquad
\begin{tikzcd}[column sep=22pt]
\Bo(\calH_0) \arrow[r,"u_1"] &
\Bo(\calH_1) \arrow[r,"u_2"] &
\Bo(\calH_2) \arrow[r,"u_3"] & \cdots
\end{tikzcd}
\end{equation}
is constructed inductively, starting from $\calH_0=\calH$ and $\Phi_0=\Phi$. For each $n=1,2,\ldots$, we define $(\calH_n,u_n,V_n)$ to be the canonical Stinespring representation of $\Phi_{n-1}$; then $\Phi_n\colon X\mapsto u_n(V_n^\dag XV_n)$. Thus,
\begin{equation}\label{uVdV}
u_n(V_n^\dag XV_n)=\Phi_n(X)=V_{n+1}^\dag\, u_{n+1}(X)\, V_{n+1}\qquad
(X\in\calH_n).
\end{equation}
This corresponds to the first diagram in Proposition~\ref{prop_KLHG} for $\calL=\calH_{n-1}$,\, $\calK=\calH=\calH_n$,  and $\calG=\calH_{n+1}$. The commutativity of the second diagram is expressed by the equation
\begin{equation}\label{uVVd}
u_{n+1}(V_n YV_n^\dag)=V_{n+1}\, u_n(Y)\, V_{n+1}^\dag\qquad
(Y\in\calH_{n-1}).
\end{equation}
Equations \eqref{uVdV} and \eqref{uVVd} can be applied iteratively. Let
\begin{equation}
V_{n,m}=V_n\cdots V_{m+1}\colon\calH_m\to\calH_n,\qquad
u_{n,m}=u_n\cdots u_{m+1}\colon\Bo(\calH_m)\to\Bo(\calH_n)\quad\:
(n\ge m);
\end{equation}
in particular, $V_{n,n-1}=V_n$ and $u_{n,n-1}=u_n$. Then
\begin{alignat}{2}
\label{uVdV-multi}
u_{m+k,m}(V_{n,m}^\dag XV_{n,m})
&=V_{n+k,m+k}^\dag\,\ts u_{n+k,n}(X)\,V_{n+k,m+k}\qquad && (X\in\calH_n),
\\[2pt]
\label{uVVd-multi}
u_{n+k,n}(V_{n,m}YV_{n,m}^\dag)
&=V_{n+k,m+k}\,\ts u_{m+k,m}(Y)\,V_{n+k,m+k}^\dag\qquad && (Y\in\calH_m).
\end{alignat}
We will use only \eqref{uVdV}, \eqref{uVdV-multi} but not \eqref{uVVd}, \eqref{uVVd-multi}.

\begin{Remark}
Since $V_1,V_2,\ldots$ are isometries, one may regard them as inclusions and define
\begin{equation}
\calH_\infty=\overline{\bigcup_{n}\calH_n},\qquad
u_\infty\colon X\mapsto\lim_{n\to\infty}
V_{\infty,n+1}\,u_{n+1}(V_{\infty,n}^\dag XV_{\infty,n})\,V_{\infty,n+1}^\dag\,
\colon\, \Bo(\calH_\infty)\to\Bo(\calH_\infty),
\end{equation}
where the bar denotes the completion, and the limit is in the weak operator topology.
\end{Remark}

\begin{proof}[General proof of equations \eqref{Phi_assoc1} and \eqref{Phi_assoc2}.]
First, we construct an analogue of the left-hand side of \eqref{1-Pi_X} covered by $k$ additional layers:
\begin{equation}
R_k(X)=u_{2+k,1}\mkern-1mu\Bigl(
(I_1-V_1V_1^\dag)\,u_1(\Phi(X))
\Bigr)\,V_{2+k}V_{1+k},
\end{equation}
where $I_n=1_{\calH_n}$. The norm of $R_k(X)$ is estimated as follows,
\[
\begin{aligned}
R_k(X)^\dag\,R_k(X)
&= V_{1+k}^\dag V_{2+k}^\dag\,\ts u_{2+k,1}\mkern-1mu\Bigl(
u_1(\Phi(X^\dag))\,(I_1-V_1V_1^\dag)\,u_1(\Phi(X))
\Bigr)\, V_{2+k}V_{1+k}
\\
&= u_{k,0}\Bigl(V_1^\dag\,\ts u_1\mkern-1mu\bigl(
V_1^\dag\,u_1(\Phi(X^\dag))\,(I_1-V_1V_1^\dag)\,u_1(\Phi(X))\,V_1
\bigr)\, V_1\Bigr)\\
&= u_{k,0}\Bigl(\Phi^2\bigl(\Phi(X^\dag)\,\Phi(X)\bigr)
-\Phi\bigl(\Phi^2(X^\dag)\,\Phi^2(X)\bigr)\Bigr)
=O(\eta)\|X\|^2,
\end{aligned}
\]
where the second line was obtained using \eqref{uVdV-multi} and the equation $u_{1+k,0}=u_{1+k,1}u_1$:
\[
V_{1+k}^\dag V_{2+k}^\dag\,\ts u_{2+k,1}(Z)\, V_{2+k}V_{1+k}
=V_{1+k}^\dag\,\ts u_{1+k,0}(V_1^\dag ZV_1)\,V_1
= u_{k,0}\mkern-1mu\bigl(V_1^\dag\,u_1(V_1^\dag ZV_1)\,V_1\bigr).
\]

By analogy with equation \eqref{WXYZOeta}, we have
\begin{equation}\label{WXYZOeta-1}
W=V_1^\dag\,R_1(X^\dag)^\dag\,u_{3,0}(\Phi(Y))\,V_3\,R_0(Z)
=O(\eta)\ts\|X\|\ts\|Y\|\ts\|Z\|.
\end{equation}
The expression for $W$ is transformed as follows:
\begin{align*}
W &= V_1^\dag V_{2}^\dag V_{3}^\dag\,\ts u_{3,1}\mkern-1mu\Bigl(
u_1(\Phi(X))\,(I_1-V_1V_1^\dag)\Bigr)\,u_{3,0}(\Phi(Y))\,V_3\,\ts
u_2\mkern-1mu\Bigl((I_1-V_1V_1^\dag)\,u_1(\Phi(Z))\Bigr)\,V_{2}V_{1}
\\
&= V_1^\dag V_{2}^\dag\,\ts u_{2,0}\mkern-1mu\Bigl(
V_1^\dag\,u_1(\Phi(X))\,(I_1-V_1V_1^\dag)\,u_1(\Phi(Y))\,V_1\Bigr)\,
u_2\mkern-1mu\Bigl((I_1-V_1V_1^\dag)\,u_1(\Phi(Z))\Bigr)\,V_{2}V_{1}
\\
&= V_1^\dag V_{2}^\dag\,\ts u_2\mkern-1mu\Bigl(
u_1\mkern-1mu\Bigl(
V_1^\dag\,u_1(\Phi(X))\,(I_1-V_1V_1^\dag)\,u_1(\Phi(Y))\,V_1
\Bigr)\,(I_1-V_1V_1^\dag)\,u_1(\Phi(Z))\Bigr)\,V_{2}V_{1}
=\Phi(Q),
\end{align*}
where
\begin{align*}
Q&=V_1^\dag\,\ts
u_1\mkern-1mu\Bigl(
V_1^\dag\,u_1(\Phi(X))\,(I_1-V_1V_1^\dag)\,u_1(\Phi(Y))\,V_1
\Bigr)\,(I_1-V_1V_1^\dag)\,u_1(\Phi(Z))\,V_1
\\[4pt]
&=\begin{aligned}[t]
&\phantom{\hbox{}+\hbox{}}
V_1^\dag\,\ts u_1\mkern-1mu\bigl(\Phi\bigl(\Phi(X)\,\Phi(Y)\bigr)\bigr)\,
u_1(\Phi(Z))\,V_1
-V_1^\dag\,\ts u_1\mkern-1mu\bigl(\Phi\bigl(\Phi(X)\,\Phi(Y)\bigr)\bigr)\,
V_1V_1^\dag\,u_1(\Phi(Z))\,V_1\\
&-V_1^\dag\,\ts u_1\mkern-1mu\bigl(\Phi^2(X)\,\Phi^2(Y)\bigr)\,
u_1(\Phi(Z))\,V_1
+V_1^\dag\,\ts u_1\mkern-1mu\bigl(\Phi^2(X)\,\Phi^2(Y)\bigr)\,
V_1V_1^\dag\,u_1(\Phi(Z))\,V_1
\end{aligned}
\\[4pt]
&=\begin{aligned}[t]
&\phantom{\hbox{}+\hbox{}}
\Phi\bigl(\Phi\bigl(\Phi(X)\,\Phi(Y)\bigr)\,\Phi(Z)\bigr)
-\Phi^2\bigl(\Phi(X)\,\Phi(Y)\bigr)\,\Phi^2(Z)\\
&-\Phi\bigl(\Phi^2(X)\,\Phi^2(Y)\,\Phi(Z)\bigr)
+\Phi\bigl(\Phi^2(X)\,\Phi^2(Y)\bigr)\,\Phi^2(Z)
\end{aligned}
\\[4pt]
&=\Phi\bigl(\Phi\bigl(\Phi(X)\,\Phi(Y)\bigr)\,\Phi(Z)\bigr)
-\Phi\bigl(\Phi(X)\,\Phi(Y)\,\Phi(Z)\bigr)
+O(\eta)\ts\|X\|\ts\|Y\|\ts\|Z\|.
\end{align*}
Thus, $W=\Phi\bigl(\Phi\bigl(\Phi(X)\,\Phi(Y)\bigr)\,\Phi(Z)\bigr)
-\Phi\bigl(\Phi(X)\,\Phi(Y)\,\Phi(Z)\bigr)
+O(\eta)\ts\|X\|\ts\|Y\|\ts\|Z\|$. This equation together with \eqref{WXYZOeta-1} implies \eqref{Phi_assoc1}. Equation \eqref{Phi_assoc2} is obtained in the same way.
\end{proof}

\subsection{Approximate factorization through a \texorpdfstring{$C^*$}{C*} algebra}\label{sec_approx_fact}

In this section, we denote tensor extensions of maps as $\Phi_n=1_{\Ma{n}}\otimes\Phi$ but do not put such subscripts on $C^*$ norms (on $\Ma{n}\otimes\Bo(\calH)$ or similar algebras) because, unlike in the general discussion of extended $\eps$-$C^*$ algebras, these norms are standard.

\begin{Theorem}\label{th_factorization}
Let $\calH$ be a nonzero finite-dimensional Hilbert space and $\Phi\colon\Bo(\calH)\to\Bo(\calH)$ a UCP map such that $\|\Phi^2-\Phi\|_\cb\le\eta$. Then there are a finite-dimensional $C^*$ algebra $\calB$ and UCP maps $\Delta\colon\calB\to\Bo(\calH)$ and $\Upsilon\colon\Bo(\calH)\to\calB$ such that
\begin{align}
\label{DelUps}
\|\Delta\Upsilon-\Phi\|_\cb&\le O(\eta),\\[2pt]
\label{UpsDel2}
\bigl\|\Upsilon_n\bigl(\Delta_n(X)\Delta_n(Y)\bigr)-XY\bigr\|
&\le O(\eta)\ts\|X\|\ts\|Y\|\qquad (X,Y\in\Ma{n}\otimes\calB).
\end{align}
(In particular, \eqref{UpsDel2} implies that $\|\Upsilon\Delta-1_\calB\|_\cb\le O(\eta)$.)
\end{Theorem}

\begin{proof}[Discussion and an outline of the proof.]
Let us consider the idempotent map $\wt{\Phi}=\theta(2\Phi-1)$ such that $\|\wt{\Phi}-\Phi\|_\cb\le O(\eta)$. By Theorem~\ref{th_almost_idemp}, the subspace $\calA=\Img\wt{\Phi}$ equipped with the Choi-Effros product $Z\star W=\wt{\Phi}(ZW)$ is an extended $O(\eta)$-$C^*$ algebra. One corollary of Theorem~\ref{th_factorization} is that the map
\[
u\colon\calB\to\calA,\qquad u(X)=\wt{\Phi}(\Delta(X))
\]
is an extended $O(\eta)$-isomorphism. Indeed, $u(X)\approx\Phi(\Delta(X))\approx\Delta(\Upsilon(\Delta(X))\approx\Delta(X)$ with $O(\eta)\ts\|X\|$ accuracy, and hence, $u(X)\star u(Y)\approx\Phi(\Delta(X)\Delta(Y))\approx\Delta(\Upsilon(\Delta(X)\Delta(Y)))\approx\Delta(XY)$. The same relations hold for $u_n$ and $\Delta_n$. The inverse of $u$ is given, with the same accuracy, by the restriction of $\Upsilon$ on $\calA$.

Theorem~\ref{th_factorization} is proved by reversing those arguments. By Theorem~\ref{th_main_ext}, there exist a finite-dimensional $C^*$ algebra $\calB$ and an extended $O(\eta)$-isomorphism $v\colon\calB\to\calA$. Let $\wt{\Delta}$ be $v$ followed by the inclusion $\calA\to\Bo(\calH)$, and let $\wt{\Upsilon}$ be $\wt{\Phi}$ with the target space $\calA$, followed by $v^{-1}$. These maps are not UCP but meet the other requirements. Indeed, it is immediate that
\begin{equation}\label{tilde_DelUps}
\wt{\Delta}\wt{\Upsilon}=\wt{\Phi},\qquad \wt{\Upsilon}\wt{\Delta}=1_\calB,\qquad
\|\wt{\Delta}\|_\cb\le 1+O(\eta),\qquad \|\wt{\Upsilon}\|_\cb\le 1+O(\eta).
\end{equation}
Since $v$ maps $I_\calB$ to $1_\calH$ and carries the $\calB$ product to the Choi-Effros product with $O(\eta)$ accuracy (including tensor extensions), these bounds hold:
\begin{gather}
\label{tilde_Del1}
\|\wt{\Delta}(I_\calB)-1_\calH\|\le O(\eta),\\[2pt]
\label{tilde_Del2}
\bigl\|\wt{\Phi}_n\bigl(\wt{\Delta}_n(X)\wt{\Delta}_n(Y)\bigr)
-\wt{\Delta}_n(XY)\bigr\|
\le O(\eta)\ts\|X\|\ts\|Y\|\quad\: (X,Y\in\Ma{n}\otimes\calB).
\end{gather}
Equations \eqref{tilde_DelUps}--\eqref{tilde_Del2}, as well as the commutation with the involution, fully characterize $\wt{\Delta}$ and $\wt{\Upsilon}$ for our purposes. In particular, these properties imply that
\[
\bigl\|\wt{\Upsilon}_n\bigl(\wt{\Delta}_n(X)\wt{\Delta}_n(Y)\bigr)-XY\bigr\|
\le O(\eta)\ts\|X\|\ts\|Y\|\qquad (X,Y\in\Ma{n}\otimes\calB).
\]
To complete the proof, we will approximate $\wt{\Delta}$ and $\wt{\Upsilon}$ by some UCP maps $\Delta$ and $\Upsilon$ in the completely bounded norm.
\end{proof}

Since $\calB$ is a finite-dimensional $C^*$ algebra, $\calB=\bigoplus_{j=1}^{m}\Bo(\calL_j)$ (up to an isomorphism). Let us represent the diagonal of $\Bo(\calL_j)$ as a unitary $1$-design, i.e. $D_j=\sum_{s}p_{js}\ts U_{js}^\dag\otimes U_{js}\in\Bo(\calL_j)\otimes\Bo(\calL_j)$, where
\begin{gather}
\label{diag_j01}
p_{js}\ge 0,\quad\: \sum_{s}p_{js}=1,\qquad U_{js}^{\dag}U_{js}=1_{\calL_j},\\[2pt]
\label{diag_j2}
\sum_{s}p_{js}\ts XU_{js}^{\dag}\otimes U_{js}
=\sum_{s}p_{js}\ts U_{js}^{\dag}\otimes U_{js}X\quad\:
\text{for all }\, X\in\Bo(\calL_j).
\end{gather}
(See \eqref{Pauli_diag} for an explicit example.) The diagonal of the entire algebra $\calB$ is $D=\sum_{s}p_{s}\ts U_{s}^\dag\otimes U_{s}$, where $s=(s_1,\dots,s_m)$,
\begin{equation}
p_{s_1,\dots,s_m}=p_{1s_1}\cdots p_{ms_m},\qquad
U_{s_1,\dots,s_m}=U_{1s_1}\oplus\cdots\oplus U_{ms_m}.
\end{equation}

Now, we define a map $\Delta'\colon\calB\to\Bo(\calH)$ by the equation
\begin{equation}
\Delta'(X)=\sum_{s}p_{s}\ts\Phi\bigl(\wt{\Delta}(XU_{s}^{\dag})\ts\wt{\Delta}(U_{s})\bigr)
=\sum_{s}p_{s}\ts\Phi\bigl(\wt{\Delta}(U_{s}^{\dag})\ts\wt{\Delta}(U_{s}X)\bigr)\qquad (X\in\calB).
\end{equation}
It is evident that $\Delta'$ commutes with the involution. The complete positivity is shown as follows: if $X\in 1_{\Ma{n}}\otimes\calB$ is positive, it can be represented as $Y^{\dag}Y$, and
\[
\Delta'_n(Y^{\dag}Y)
=\sum_{s}p_{s}\ts\Phi_n\bigl(\wt{\Delta}_n(Y^{\dag}(I_n\otimes U_{s}^{\dag}))\,
\wt{\Delta}_n((I_n\otimes U_{s})Y)\bigr)\ge 0.
\]
Due to equation \eqref{tilde_Del2}, we have $\Delta'_n(X)=\wt{\Delta}_n(X)+O(\eta)\ts\|X\|$ for all $X\in 1_{\Ma{n}}\otimes\calB$, implying that $\|\Delta'-\wt{\Delta}\|_\cb\le O(\eta)$. And using \eqref{tilde_Del1}, we conclude that
\begin{equation}
\Delta\colon X\mapsto (\Delta'(I_\calB))^{-1/2}\ts\Delta'(X)\ts(\Delta'(I_\calB))^{-1/2}
\end{equation}
is a UCP map such that $\|\Delta-\wt{\Delta}\|_\cb\le O(\eta)$.

Before proceeding to the construction of $\Upsilon$, let us list some properties of $\Delta$:
\begin{alignat}{2}
\label{Delta_norm}
\bigl|\|\Delta_n(X)\|-\|X\|\bigr|
&\le O(\eta)\ts\|X\|\qquad& &(X\in\Ma{n}\otimes\calB),\\[2pt]
\label{PhiDelta1}
\|\Phi\Delta-\Delta\|_\cb&\le O(\eta),\qquad& &\\[2pt]
\label{PhiDelta2}
\bigl\|\Phi_n\bigl(\Delta_n(X)\Delta_n(Y)\bigr)-\Delta_n(XY)\bigr\|
&\le O(\eta)\ts\|X\|\ts\|Y\|\qquad& &(X,Y\in\Ma{n}\otimes\calB),\\[2pt]
\label{PhiDelta3}
\bigl\|\Phi_n\bigl(\Delta_n(X)\Delta_n(Y)\Delta_n(Z)\bigr)-\Delta_n(XYZ)\bigr\|
&\le O(\eta)\ts\|X\|\ts\|Y\|\ts\|Z\|\qquad& &(X,Y,Z\in\Ma{n}\otimes\calB).
\end{alignat}
The first three of these follow from \eqref{tilde_DelUps} and \eqref{tilde_Del2}. To prove the last property, we do the following calculation with $O(\eta)\ts\|X\|\ts\|Y\|\ts\|Z\|$ accuracy, where the second step is the application of equation \eqref{Phi_assoc1} to the UCP map $\Phi_n$:
\[
\begin{aligned}
\Phi_n\bigl(\Delta_n(X)\ts\Delta_n(Y)\ts\Delta_n(Z)\bigr)
&\approx\Phi_n\bigl(\Phi_n(\Delta_n(X))\,\Phi_n(\Delta_n(Y))\,
\Phi_n(\Delta_n(Z))\bigr)\\
&\approx\Phi_n\Bigl(\Phi_n\bigl(\Phi_n(\Delta_n(X))\,\Phi_n(\Delta_n(Y))\bigr)\,
\Phi_n(\Delta_n(Z))\Bigr)\\
&\approx\Phi_n\bigl(\Phi_n\bigl(\Delta_n(X)\ts\Delta_n(Y)\bigr)\ts
\Delta_n(Z)\bigr)
\approx\Phi_n\bigl(\Delta_n(XY)\ts\Delta_n(Z)\bigr)
\approx\Delta_n(XYZ).
\end{aligned}
\]
We will also need a Choi representation of $\Delta$,
\begin{gather}
\label{Choi_Delta}
\Delta(X)=\sum_{j=1}^{m}W_j^\dag(X_j\otimes 1_{\calE_j})W_j\qquad
(X=(X_1,\dots,X_m)\in\calB),\\
\text{where}\qquad
W_j\colon\calH\to\calL_j\otimes\calE_j,\qquad \sum_{j=1}^{m}W_j^{\dag}W_j=1_{\calH}.
\end{gather}

\begin{Lemma}\label{lem_RC}
The operator
\begin{equation}\label{R_def}
R_j =\sum_{s}p_{js}(U_{js}^{\dag}\otimes 1_{\calE_j})
W_jW_j^{\dag}(U_{js}\otimes 1_{\calE_j})
\end{equation}
has the form $1_{\calL_j}\otimes C_j$ for some $C_j\in\Bo(\calE_j)$. Furthermore, $1-O(\eta)\le\|C_j\|\le 1$.
\end{Lemma}
\begin{proof}
Due to the property \eqref{diag_j2} of the diagonal, $R_j$ commutes with $X\otimes 1_{\calE_j}$ for all $X\in\Bo(\calL_j)$. Hence, $R_j=1_{\calL_j}\otimes C_j$ for some $C_j$. The upper bound $\|C_j\|=\|R_j\|\le 1$ follows from the fact that $\|W_j\|\le 1$. To prove the lower bound, we note that $\|\Phi(W_j^{\dag}R_jW_j)\|\le\|C_j\|$. The left-hand side can be estimated using \eqref{Choi_Delta} (with a single nonzero $X_j$) and \eqref{PhiDelta2}:
\[
\Phi(W_j^{\dag}R_jW_j)
=\sum_{s}p_{js}\ts\Phi\bigl(\Delta(U_{js}^{\dag})\ts\Delta(U_{js})\bigr)
=\sum_{s}p_{js}\ts\Delta(U_{js}^{\dag}U_{js})+O(\eta)
=\Delta(1_{\calL_j})+O(\eta).
\]
Thus, $\|\Phi(W_j^{\dag}R_jW_j)\|\ge\|\Delta(1_{\calL_j})\|-O(\eta)$, where $\|\Delta(1_{\calL_j})\|\ge 1-O(\eta)$ due to \eqref{Delta_norm}.
\end{proof}

Let $\xi_j\in\calE_j$ be a unit vector such that $\bigl|\|C_j\xi_j\|-1\bigr|\le O(\eta)$. Let us also consider the Choi representation $\Phi(X)=V^\dag(X\otimes 1_\calF)V$ with $V\colon\calH\to\calH\otimes\calF$ and derive new linear maps from $V$:
\begin{equation}
L_j\colon\calL_j\to\calH\otimes\calF,\qquad
L_j=\sum_{s}p_{js}\ts\bigl(\Delta(U_{js}^\dag)\otimes 1_\calF\bigr)
V W_j^\dag(U_{js}\otimes\xi_j).
\end{equation}
We now construct a manifestly completely positive map $\Upsilon'\colon\Bo(\calH)\to\calB$ by components,
\begin{equation}
\Upsilon'=(\Upsilon'_1,\dots,\Upsilon'_m),\qquad \text{where}\quad
\Upsilon'_j\colon\Bo(\calH)\to\Bo(\calL_j),\quad\:
\Upsilon_j'(X)=L_j^\dag\bigl(\Phi(X)\otimes 1_\calF\bigr)L_j.
\end{equation}
Because $\Phi$ is $\eta$-idempotent, $\|\Upsilon'\Phi-\Upsilon'\|_\cb\le O(\eta)$. Let us prove that $\|\Upsilon'\Delta-1_\calB\|_\cb\le O(\eta)$, meaning that if $Y=(Y_1,\dots,Y_m)\in\Ma{n}\otimes\calB$, then $\|(\Upsilon'_j)_n(\Delta_n(Y))-Y_j\|\le O(\eta)\ts\|Y\|$, where $\|Y\|=\max_k\|Y_k\|$. We calculate $(\Upsilon'_j)_n(\Delta_n(Y))$ with $O(\eta)\ts\|Y\|$ accuracy, omitting the subscript $n$:
\[
\begin{alignedat}{2}
\Upsilon'_j(\Delta(Y))
&=\sum_{r,s}p_{jr}p_{js}\ts
(U_{jr}^\dag\otimes\xi_j^\dag)W_j\,
\Phi\bigl(\Delta(U_{jr})\ts\Phi(\Delta(Y))\Delta(U_{js}^\dag)\bigr)\,
W_j^\dag(U_{js}\otimes\xi_j)&&\\
&\approx \sum_{r,s}p_{jr}p_{js}\ts
(U_{jr}^\dag\otimes\xi_j^\dag)W_j\,\Delta(U_{jr}YU_{js}^\dag)\,
W_j^\dag(U_{js}\otimes\xi_j)&\hspace{-2cm}
(\text{due to \eqref{PhiDelta1} and \eqref{PhiDelta3}})&\\
&=\sum_{r,s}p_{jr}p_{js}\ts
(U_{jr}^\dag\otimes\xi_j^\dag)W_jW_j^\dag\ts
\bigl(U_{jr}Y_jU_{js}^\dag\otimes 1_{\calE_j}\bigr)\ts
W_jW_j^\dag(U_{js}\otimes\xi_j)\quad&
(\text{due to \eqref{Choi_Delta}})&\\
&=(1_{\calL_j}\otimes\xi_j^\dag)R_j^\dag\ts
(Y_j\otimes 1_{\calE_j})\ts R_j(1_{\calL_j}\otimes\xi_j)\quad&
(\text{due to \eqref{R_def}})&\\[2pt]
&=Y_j\otimes\bigl(\xi_j^{\dag}C_j^{\dag}C_j\xi_j\bigr)\approx Y_j&\hspace{-6cm}
(\text{due to Lemma~\ref{lem_RC} and the assumption $\|C_j\xi_j\|\approx 1$})&.
\end{alignedat}
\]
Thus, $\Upsilon'\approx \Upsilon'\Phi\approx \Upsilon'\wt{\Phi}= \Upsilon'\wt{\Delta}\wt{\Upsilon}\approx \Upsilon'\Delta\wt{\Upsilon}\approx \wt{\Upsilon}$ with $O(\eta)$ accuracy. Finally, we define the UCP map
\begin{equation}
\Upsilon\colon X\mapsto (\Upsilon'(1_\calH))^{-1/2}\ts\Upsilon'(X)\ts(\Upsilon'(1_\calH))^{-1/2},
\end{equation}
which also has the property $\|\Upsilon-\wt{\Upsilon}\|_\cb\le O(\eta)$.

\section*{Acknowledgments}

I thank Daniel Ranard for the motivation and challenge to work on the almost-idempotent channel problem as well as for multiple discussions and comments, Michael Freedman for nice discussions and comments, and Isaac Kim for pointing to related results~\cite{KoIm01,HJPW03}. This work was supported by the Simons Foundation under grant 376205 and by the Institute of Quantum Information and Matter, a NSF Physics Frontiers Center.

\end{document}